\documentclass{article}

\usepackage{amsmath}
\usepackage{amsfonts}
\usepackage{amsthm}
\usepackage{amssymb}
\usepackage{bm}
\usepackage[applemac]{inputenc}
\usepackage{graphicx}
\usepackage[font=small,labelfont=bf]{caption}
\usepackage{mathpazo}

\usepackage{todonotes}

\usepackage{enumerate}

\usepackage{tikz}
\usepackage{tikz-cd}

\makeatletter

\newtheoremstyle{mystyle}{}{}{\slshape}{2pt}{\scshape}{.}{ }{} 

\newtheorem{thm}{Theorem}[section]
\newtheorem{corollary}[thm]{Corollary}
\newtheorem{cor}[thm]{Corollary}
\newtheorem{prop}[thm]{Proposition}
\newtheorem{lemme}[thm]{Lemma}
\newtheorem{lemma}[thm]{Lemma}

\newtheorem{fact}[thm]{Fact}
\newtheorem{conjecture}[thm]{Conjecture}
\newtheorem{question}[thm]{Question}

\theoremstyle{definition}
\newtheorem{defi}[thm]{Definition}
\newtheorem{definition}[thm]{Definition}
\theoremstyle{mystyle}
\newtheorem{ex}[thm]{Example}

\theoremstyle{remark}
\newtheorem{rem}[thm]{Remark}

\newcommand{\wmc}{\subseteq_{\text{wmc}}}

\newcommand{\rel}[1]{\mathrel{#1}}
\newcommand{\bin}[1]{\mathbin{#1}}

\newcommand{\ignore}[1]{}

\DeclareMathOperator{\tp}{tp}
\DeclareMathOperator{\dlr}{opD}

\DeclareMathOperator{\dcl}{dcl}
\DeclareMathOperator{\acl}{acl}

\DeclareMathOperator{\aut}{Aut}
\DeclareMathOperator{\rk}{rk}

\def\indsym#1#2{%
 \setbox0=\hbox{$\m@th#1x$}%
 \kern\wd0%
 \hbox to 0pt{\hss$\m@th#1\mid$\hbox to 0pt{$\m@th#1^{#2}$\hss}\hss}%
 \lower.9\ht0\hbox to 0pt{\hss$\m@th#1\smile$\hss}%
 \kern\wd0}
\newcommand{\ind}[1][]{\mathop{\mathpalette\indsym{#1}}}

\def\nindsym#1#2{%
 \setbox0=\hbox{$\m@th#1x$}%
 \kern\wd0%
 \hbox to 0pt{\hss$\m@th#1\not$\kern1.4\wd0\hss}
 \hbox to 0pt{\hss$\m@th#1\mid$\hbox to 0pt{$\m@th#1^{#2}$\hss}\hss}%
 \lower.9\ht0\hbox to 0pt{\hss$\m@th#1\smile$\hss}%
 \kern\wd0}

\title{NIP $\omega$-categorical structures: the rank 1 case}
\author{Pierre Simon\footnote{Partially supported by NSF (grants no. 1665491 and 1848562) and a Sloan fellowship.}}

\date{}

\begin{document}

\maketitle

\begin{abstract}
We classify primitive, rank 1, $\omega$-categorical structures having polynomially many types over finite sets. We show that there are only finitely many such structures with a fixed number of 4 types and that they are built out of finitely many linear or circular orders interacting in a restricted number of ways. As an example of application, we deduce the classification of primitive structures homogeneous in a language consisting of $n$ linear orders as well as all reducts of such structures.
\end{abstract}

\section{Introduction}

Since the work of Lachlan on finite homogeneous structures,   interactions between homogeneous structures and model theory have been very fruitful in both directions. Lachlan \cite{Lachlan:stable_homogeneous} realized that the property of stability and the toolbox that comes with it were relevant in the finite case. Geometric stability theory had its birth in Zilber's work on totally categorical structures \cite{Zilber_book} and this in turn lead to a fairly detailed understanding of the $\omega$-stable $\omega$-categorical structures (\cite{CHL}, \cite{Hr_totallycategorical}). Following a suggestion of Lachlan, this analysis was then generalized to smoothly approximable structures. This was done first by Kantor, Liebeck, Macpherson \cite{KLM} in the primitive case using classification of finite simple groups, and then by Cherlin and Hrushovski \cite{CherHru} in the general case by model-theoretic methods. In that latter work, many features of simple theories first appeared. The present paper fits in this line of research and begins the study of yet another class of $\omega$-categorical structures defined by a model theoretic condition.


To define this class, let us restrict first to the case of structures homogeneous in a finite relational language (which we also call \emph{finitely homogeneous}). If $M$ is such a structure, then given any finite $A\subseteq M$, the number of 1-types over $A$ (that is, the number of orbits under the stabilizer of $A$) is finite. For a given $n$, we let $f_M(n)$ be the maximal number of 1-types over a set $A\subseteq M$ of size $n$. For instance, if $M=(\mathbb Q,\leq)$, then $f_M(n)=2n+1$. If $M=(G,R)$ is a model of the random graph, then $f_M(n)=2^n+n$. A well-known theorem of Sauer and Shelah implies that this function has either polynomial or exponential growth. 
Following the model-theoretic terminology, we call a finitely homogeneous structure $M$ \emph{NIP} if the function $f_M$ has polynomial growth. (\emph{NIP} stands for the negation of the independence property. We like to think of those structures as being \emph{geometric} in some sense.) For instance, dense linear orders are NIP, whereas the random graph is not. Intuitively, NIP structures have no random-like behavior. Another important example of NIP structure is the Fra\" iss\'e limit of finite trees (where a tree $(T,\leq,\wedge)$ is a partial order such that the predecessors of a point form a chain, and $a\wedge b$ is the infimum of $\{a,b\}$).

Within structures homogeneous in a finite relational language, there is another characterization of NIP obtained by counting orbits on unordered $k$-tuples, or equivalently finite substructures of size $k$ up to isomorphism. If $M$ is homogeneous in a finite relational language (or more generally an $\omega$-categorical relational structure), define $\pi_M(k)$ as being the number of substructures of $M$ of size $k$. Cameron showed in \cite{Cameron:orbits2} that this function is always non-decreasing and in \cite{Cameron:linear} he classified the case where $\pi_M$ is constant equal to 1. Macpherson \cite{Mac:orbits} showed that if $M$ is primitive, then $\pi_M$ is either constant equal to 1 or grows at least exponentially. A number of structures for which the growth is no faster than exponential are given by Cameron in \cite{Cameron:trees}: they are all order-like or tree-like structures. Cameron also remarks there that those seem to be essentially the only examples of such structures known at the time. In \cite{Mac:rapid_growth}, Macpherson shows that for structures homogeneous in a finite relational language, there is a gap in the possible growth rates of the function $\pi_M$. Using the aforementioned Sauer--Shelah theorem, we can state a stronger version of his result: if $M$ is NIP, then $\pi_M(k)=O(2^{ck \ln k})$ for some $c>0$ (see the remark after Fact \ref{fait:sauer-shelah}). If $M$ has IP (is not NIP), then $\pi_M(k)\geq 2^{p(k)}$ for some polynomial $p(X)$ of degree at least 2. Hence finitely homogeneous structures with $\pi_M$ of exponential growth are a subclass of NIP homogeneous structures. See e.g. \cite[Section 6.3]{Mac_survey} for many more results on this function.

We conjecture that NIP finitely homogeneous structures can be reasonably well classified, and in particular that there are only countably many up to bi-interpretability. We will give some precise conjectures at the end of this paper. What we have in mind is that those structures are all built out of linear orders, possibly branching into trees. However, we are for now not capable of saying much in the general case, and introduce another condition, which should be thought of as forbidding trees in the structure: we ask that there is a rank function on definable sets satisfying certain axioms. This limits the size of a nested sequence of definable equivalence relations. In model theory, this condition is called \emph{rosiness}. It is always satisfied by binary structures, so one may want to think of this work as studying binary NIP homogeneous structures, though our actual hypothesis are \emph{a priori} more general.
We will actually relax the homogeneity assumption to $\omega$-categoricity. Similarly, NIP, which we defined by counting types, becomes a condition on formulas. Under those hypotheses, we conjecture that the results on $\omega$-categorical $\omega$-stable and quasi-finite structures essentially go through \emph{mutatis mutandis}. In particular, we should have coordinatization by rank 1 sets and quasi-finite axiomatization. We deal here only with the rank 1 primitive case, for which we give a complete classification, up to bi-definability. The general finite rank case is studied in a subsequent work \cite{RosFinHom} with Alf Onshuus.

As a rather straightforward application, we classify primitive homogenous multi-orders (also called finite-dimensional permutation structures): that is primitive structures homogeneous in a language consisting of $n$ linear orders. For $n=2$, this was solved by Cameron \cite{cameron:permutations} and for $n=3$ by Braunfeld \cite{Braunfeld:3d}, where the general case is conjectured. We show that for any $n$, there is only one primitive homogeneous multi-order, where no two orders are equal or reverse of each other: the Fra\"iss\' e limit of finite sets with $n$ orders. We also classify all reducts of such structures, generalizing the work of Linman and Pinsker \cite{Linman-Pinsker} on the case $n=2$.

\medskip
Looking at it from the point of view of model theory, one can see this work as a development of the study of (rosy) NIP structures along the lines of stable theories. We hope that it will eventually lead to new insights into general NIP structures. At any rate, the results demonstrate that there is a richer theory of NIP than one suspected only a few years ago and that this world is much more structured and closer to stability than was expected. It does not seem completely unreasonable to hope for classification results for some subclasses of NIP in the spirit of Shelah's classification for superstable theories, where cardinal dimensions will be complemented by isomorphism types of linear orders (which are shown to exist in \cite{linear_orders}). But we are not quite there yet.

\subsection{Summary of results}
We are concerned with structures $M$ such that:

\smallskip
$(\star)$ $M$ is an $\omega$-categorical, rank 1, primitive, unstable NIP structure,
\smallskip
\\where ``rank 1" means that, in $M^{eq}$, there is no infinite definable set $E$ and uniformly definable family $(X_t)_{t\in E}$ of infinite subsets of $M$ which is $k$-inconsistent for some $k$: that is for any $k$ values $t_1,\ldots,t_k \in E$, we have $X_{t_1}\cap \cdots \cap X_{t_k} = \emptyset$.
Those hypotheses will be fully enforced only in Section \ref{sec:classification}. In sections before that, we study $\omega$-categorical linear and circular orders under a weakening of the rank 1 assumption, but make no use of NIP. Results there might be of some use in the classification of other classes of ordered homogeneous structures. We then give a fairly explicit description of structures satisfying $(\star)$ up to bi-definability. They all admit an interpretable finite cover composed of a disjoint union of linear and circular order, independent of each other.

Here are some examples of structures that satisfy the hypotheses.

\begin{ex}\label{ex:basic}
	\begin{itemize}
		\item A dense linear order or any of its 3 non-trivial reducts: a betweenness relation, circular order or separation relation.

		\item The Fra\" iss\'e limit of finite sets equipped with $n$ orders.
				\item The class of structures equipped with two linear orders $\leq_1,\leq_2$ and a binary relation $R$ that satisfies $a'\leq_1 a \rel R b \leq_2 b' \Rightarrow a' \rel{R} b'$ and $\neg a\rel{R} a$ is a Fra\" iss\'e class. Its Fra\" iss\'e limit satisfies $(\star)$. This kind of structure will be studied in Section \ref{sec:intertwinings}.
		\item The class of finite sets equipped with a circular order $C$ and an equivalence relation $E$ all of whose classes have exactly two elements is a Fra\" iss\'e class. The quotient by $E$ of the Fra\" iss\'e limit of this class satisfies $(\star)$. It does not admit a circular order definable over $\acl^{eq}(\emptyset)$ but does have one definable over any one parameter.
	\end{itemize}	
\end{ex}

As a consequence of the classification we obtain the following theorems (the terminology will be explained later).

\begin{thm}\label{th:main1}
Given an integer $n$, there are, up to bi-definability, finitely many $\omega$-categorical primitive NIP structures $M$ of rank 1 having at most $n$ 4-types.
\end{thm}


\begin{thm}\label{th:list of properties}
	If $M$ is an $\omega$-categorical, primitive, rank 1, NIP, unstable structure, then:
	\begin{enumerate}
		\item over $\emptyset$, there is an interpretable set $W$, which is a finite union of circular orders and admits a finite-to-one map to $M$;
		\item up to inter-definability, $M$ is homogeneous in a finite relational language and finitely axiomatizable;
		\item after naming a finite set of points, $M$ admits elimination of quantifiers in a binary language and has a definable linear order;
		\item $M$ is distal of finite op-dimension;
		\item $M$ has trivial geometry: $\acl(A)=A$ for every $A\subseteq M$, equivalently the stabilizer of any finite $A\subseteq M$ in the automorphism group of $M$ has no finite orbit on $M\setminus A$.
	\end{enumerate}
\end{thm}

Statement 1 follows from the construction in Section \ref{sec:classification}.
Statements 2 and 3 are proved in Section \ref{sec:homogeneous}, along with distality. Statement 5 also follows from the discussion there. Finiteness of op-dimension is Proposition \ref{prop:bounding op-dimension}.

As regards homogeneous multi-orders, we prove the following.

\begin{thm}
	Let $(M;\leq_1,\ldots,\leq_n)$ be homogeneous, primitive and such that each $\leq_i$ defines a linear order on $M$. Assume that no two of those orders are equal or reverse of each other. Then $M$ is the Fra\"iss\'e limit of finite sets with $n$ orders.
	
\end{thm}

The proof of this last theorem requires only a small part of the paper, namely Sections \ref{sec:prelim}, \ref{sec:linear orders} and \ref{sec:permutations}. The imprimitive case is classified in \cite{perm_imprimitive}, joint with Samuel Braunfeld.

\subsection{Overview of the proof}

Let $M$ be an $L$-structure that satisfies $(\star)$.
The starting point for this work is the result proved in \cite{linear_orders} that any NIP $\omega$-categorical unstable structure interprets a linear order. In fact more is true: Guingona and Hill introduce in \cite{GuinHill} the notion of op-dimension, which tells us the maximal number of independent orders that a structure (or type) can have. The main theorem of \cite{linear_orders} says---in the $\omega$-categorical case---that if $M$ is NIP of op-dimension at least $n$, then we can find some infinite definable set $X$ on which we can interpret $n$ linear orders. By transitivity of $M$, the family of conjugates of $X$ covers $M$.

In Section \ref{sec:linear orders}, we show that any extra structure on a rank 1 linear order must be dense with respect to the order and that different definable orders can interact only in a few prescribed ways. This is extended to circular orders in Section \ref{sec:circular orders}. (Those sections make no use of NIP.) This allows us to glue the orders coming from various conjugates of $X$ together. Each order might then wrap around itself, yielding a circular order. We construct in this way a $0$-definable finite family $W$ of linear and circular orders. We also show that this $W$ is a finite cover of $M$, that is it admits a finite-to-one map onto $M$. This is all done in Section \ref{sec:classification}.

We then have to analyze the additional structure on $W$. Using op-dimension, we show that any additional structure must come from stable formulas. By rank 1, those formulas cannot fork. Using finiteness of the number of non-forking extensions, those formulas can be defined from \emph{local} equivalence relations with finitely many classes. Here \emph{local} means that the equivalence relation is only defined locally, on bounded intervals of the orders, and may not glue as an equivalence relation on the whole structure. Such relations are studied in Section \ref{sec:local relations}, in which a purely topological discussion shows that they must come from connected finite covers of circular orders.

%

\subsection*{Acknowledgments}

Thanks to  Alf Onshuus, Dugald Macpherson, Udi Hrushovski, Gregory Cherlin and Sam Braunfeld for helpful discussions and for looking over parts of those results. Thanks also to David Bradley-Williams for pointing out the application to multi-orders. Finally, many thanks are due to the referees for reading through several versions of this paper and producing remarkably detailed reports that helped a lot in improving correctness and readability.

\section{Preliminaries}\label{sec:prelim}

\subsection{Model theoretic terminology}

We will use standard model-theoretic notation and terminology. Lowercase letters such as $a,b,c,x,y,z$ will usually denote finite tuples of elements or variables: $a\in M$ means $a\in M^{|a|}$. We will sometimes write say $\bar a,\bar x$ if we want to emphasize this.

For the sake of completeness, we recall some basic definitions. More details can be found in any introductory book on model theory, for instance \cite{Marker} or \cite{TentZieg}. We first give the general definitions that make sense in arbitrary structures, and then give equivalent formulations in terms of automorphism groups, that are only valid in the $\omega$-categorical case, over finite set of parameters.

We work in a structure $M$, in a countable language $L$. Let $B\subseteq M$ be any set. A subset $X\subseteq M^k$ is \emph{definable over $B$}, or \emph{$B$-definable} if it is the solution set of a first-order formula $\phi(x;b)$, where $b$ is a tuple of parameters from $B$. A set is \emph{definable} if it is definable over some $B$. We write \emph{$0$-definable} to mean $\emptyset$-definable. The notation $M\models \phi(a;b)$ and $a\models \phi(x;b)$ both mean that $\phi(a;b)$ is true in $M$. Since we will work throughout in a fixed structure $M$, we will usually not indicate it and simply write $\models \phi(a;b)$ instead of $M\models \phi(a;b)$. If $\pi(x)$ is a set of formulas all with variable $x$, we write $a\models \pi(x)$ to mean $a\models \phi(x)$ for all $\phi(x)\in \pi(x)$.

The \emph{type} (or \emph{complete type}) of a tuple $a \in M^k$ over  $B$, denoted $\tp(a/B)$, is the set of formulas $\phi(x;b)$, $|x|=|a|$, with parameters $b$ in $B$ that hold of $a$. If $B=\emptyset$, we may omit it. If $p=\tp(a/B)$, we usually write $p\vdash \phi(x;b)$ to mean $\phi(x;b)\in p$. The set of types in $k$ variables over $B$ is denoted $S_k(B)$. We sometimes omit $k$ if it is clear from the context, or irrelevant. We write $a\equiv_B a'$ to mean that $a$ and $a'$ have the same type over $B$. We will often abuse terminology by saying that a formula $\phi(x)$ is a complete type if it implies a complete type, or in other words can be uniquely extended to a complete type. Similarly, we will say that a definable set $X$ is a complete type if it is the set of realization of a complete type, or equivalently is defined by a complete type.

The \emph{definable closure} of a set $A$, denoted $\dcl(A)$, is the set of elements $c\in M$ for which there exists a formula $\phi(x;a)$, $a$ a tuple from $A$, of which $c$ is the only solution in $M$. Similarly the \emph{algebraic closure} of a set $A$, denoted $\acl(A)$, is the set of elements $c\in M$ for which there exists a formula $\phi(x;a)$, $a$ a tuple from $A$, satisfied by $c$ and which has only finitely many solutions in $M$.

It is often important to consider not only definable subsets of $M$ (or $M^k$), but also quotients of definable subsets by definable equivalence relations. A convenient way to do this is to introduce a multisorted structure $M^{eq}$ in which such quotients are represented by definable sets. More precisely, $M^{eq}$ has a sort $M_E$ for every $0$-definable equivalence relation $E$ on some $M^n$. The sort $M_E$ is interpreted as the quotient of $M$ by $E$. The sort $M_{=}$ is identified with $M$ and equipped with the same structure as $M$. Furthermore, for each $E$ as above, we equip $M^{eq}$ with the canonical projection map $\pi_E$ from $M^n$ to $M_E$. One can then show that, for any $A\subseteq M$, a subset of $M_E$ is $A$-definable in $M^{eq}$ if and only if its pre-image under $\pi_E$ is $A$-definable in the original $M$. In particular, the original $M$ and the copy of $M$ inside $M^{eq}$ have the same definable sets.

\smallskip
We recall that a countable structure $M$ is \emph{$\omega$-categorical} if any of the following equivalent conditions is satisfied:

-- For any $n<\omega$, there are finitely many types of the form $\tp(a/\emptyset)$, with $a\in M^n$.

-- For any finite $A\subseteq M$, there are finitely many 1-types $\tp(a/A)$, with $a\in M$ a singleton.

-- For any $n<\omega$, the action of $\aut(M)$ on $M^n$ has finitely many orbits.

-- Any countable $N$ elementarily equivalent to $M$ is isomorphic to it.

\smallskip
Assume from now on that $M$ is $\omega$-categorical. Then one can define most model-theoretic notions using the automorphism group alone (at least over finite parameter sets). Let $A\subseteq M$ be finite. A subset $X\subseteq M^n$ is $A$-definable if and only if it is (setwise) invariant under the group $\aut(M)_A$ of automorphisms fixing $A$ pointwise. In particular, $X$ is $0$-definable if and only if it is $\aut(M)$-invariant. 

Still assuming that $A$ is finite, two tuples $a$ and $a'$ have the same type over $A$, denoted $a\equiv_A a'$, if and only if there is an automorphism of $M$ fixing $A$ pointwise and sending $a$ to $a'$. Thus types over $A$ are in natural bijection with orbits of $\aut(M)_A$.

An element $c$ of $M$ is in the definable closure of $A$ if and only if it is fixed by $\aut(M)_A$. Similarly, $c$ is in the algebraic closure of $A$ if and only if its orbit under $\aut(M)_A$ is finite.


We will often consider the algebraic closure evaluated in $M^{eq}$: $\acl^{eq}(a)$. This can be though of as containing a name for each equivalence class of $a$ under a 0-definable equivalence relation with finitely many classes. In particular, a subset $X\subseteq M^n$ is definable over $\acl^{eq}(\emptyset)$ if and only if it has finitely many conjugates under the automorphism group of $M$. If $a$ and $b$ are in $M$, then $a\in \acl(b)$ if and only if $a \in \acl^{eq}(b)$, since the definable subsets of $M^k$ are the same seen in $M$ or in $M^{eq}$. The \emph{strong type} of $a$ over $A$ is the type of $a$ over $\acl^{eq}(A)$: two elements have the same strong type over $A$ if they are equivalent for every $A$-definable equivalence relation with finitely many classes.

Let $A\subseteq M$ be any set of parameters and let $X\subseteq M^k$ be an $A$-definable set. We say that $X$ is \emph{transitive} over $A$ if any two elements of $X$ have the same type over $A$. Note that since $X$ is $A$-definable, any element of $M^k$ having the same type as a member of $X$ is itself in $X$. Thus an $A$-definable set $X$ is transitive over $A$ if and only if $\aut(M)_A$ acts transitively on it. This is just another way of saying that $X$ is a complete type over $A$. Similarly, we say that the $A$-definable set $X$ is \emph{primitive} over $A$ if the action of $\aut(M)_A$ on $X$ is primitive, or equivalently $X$ does not admit any non-trivial $A$-definable equivalence relation. If $A=\emptyset$, then we will usually omit ``over $A$".

\smallskip
We say that two structures $M$ and $N$ are \emph{inter-definable} if they have the same universe and the same $0$-definable sets (in all cartesian powers). Hence $M$ and $N$ are essentially the same structure, but in possibly different languages. We say that $M$ and $N$ are \emph{bi-definable} if $M$ is inter-definable with a structure isomorphic to $N$ (or equivalently $N$ is inter-definable with a structure isomorphic to $M$).

\smallskip
\textbf{Assumption}: Throughout this paper, we work in an $\omega$-categorical structure $M$ in a language $L$. That assumption will in general not be recalled, and is implicitly assumed in all statements.

\subsection{Homogeneous structures}

We will call a countable structure $M$ in a relational language $L$ \emph{homogeneous} if for any finite $A\subseteq M$ and $\sigma\colon A\to M$ a partial isomorphism (that is, $\sigma\colon  A \to \sigma(A)$ is an isomorphism, where $A$ and $\sigma(A)$ are equipped with the induced structure from $M$), there is an automorphism $\tilde \sigma\colon  M \to M$ that extends $\sigma$. This is also sometimes called \emph{ultrahomogeneous}.

We call a structure $M$ \emph{finitely homogeneous} if it is homogeneous and its language is finite and relational. A structure $M$ is \emph{finitely homogenizable} if it is inter-definable with a finitely homogeneous structure. Note that any finitely homogenenizable structure is $\omega$-categorical: since the language is finite relational, the number of isomorphism types of substructures of a fixed size $n<\omega$ is finite, hence by homogeneity, the action of $\aut(M)$ on $M^n$ has finitely many orbits.

It is easy to see that a structure $M$ is finitely homogenizable if and only if there is $k<\omega$ such that the following two conditions hold:

\begin{itemize}
	\item There are finitely many types of $k$-tuples of elements of $M$.
	\item For any $n<\omega$, any two $n$-types $p(\bar x)$ and $q(\bar x)$ of tuples of elements of $M$ are equal if and only if they have the same restriction to any set of $k$-variables.
\end{itemize}

\subsection{Linear orders and their reducts}

There is only one countable homogeneous linear order: $(\mathbb Q,\leq)$. It is also the only $\omega$-categorical linear order with transitive automorphism group. Its reducts follow from Cameron's result on highly homogeneous permutations groups \cite{Cameron:linear}: there are five of them. Apart from the trivial reduct to pure equality, there are three unstable proper reducts:

\begin{itemize}
	\item the generic betweenness relation $(\mathbb Q;B(x,y,z))$, where \[B(x,y,z) \leftrightarrow (x\leq y \leq z)\vee (z\leq y \leq x);\]
	\item the generic circular order $(\mathbb Q;C(x,y,z))$, where \[C(x,y,z) \leftrightarrow (x\leq y\leq z)\vee (z\leq x\leq y) \vee (y\leq z\leq x);\]
	\item the generic separation relation $(\mathbb Q;S(x,y,z,t))$, where \begin{eqnarray*}S(x,y,z,t) \leftrightarrow (C(x,y,z)\wedge C(y,z,t)\wedge C(z,t,x)\wedge C(t,x,y)) \vee \\ (C(t,z,y)\wedge C(z,y,x)\wedge C(y,x,t) \wedge C(x,t,z)).\end{eqnarray*}
\end{itemize}

The automorphism group of the betweenness relation is generated by the automorphisms of the linear order along with a bijection that reverses the order, for instance $x\mapsto -x$. Similarly, the automorphism group of the separation relation is generated from that of the circular order along with an order-reversing bijection.

Depending on the context, order will mean either linear order or circular order; by default linear.
Linear and circular orders will play an essential role in this paper, but the betweenness and separation relations will not explicitly appear. They will be accounted for in the analysis by having every order come with a dual in order-reversing bijection with it. Thus the betweenness relation for example will be present in our classification as the quotient of two linear orders in order-reversing bijection.

\subsection{Rank}\label{sec:rank}

We define rank as in \cite[Section 2.2.1]{CherHru}, restricting to the $\omega$-categorical context. This notion of rank also coincides with what is now called thorn-rank, which is defined for any structure: see \cite[Definition 4.1, Remark 4.1.9]{On1}. We start by giving the definition as it appears in \cite{CherHru}, then discuss it briefly before giving an equivalent definition. Hopefully, this will help the reader gain some intuition of this notion.

\begin{definition}\label{def:rank0}
  Given a definable set $D\subseteq M^{eq}$ and ordinal $\alpha$, we define inductively $\rk(D)\geq \alpha$:
  \begin{itemize}
  \item $\rk(D)\geq 1$ if and only if $D$ is infinite;
  \item $\rk(D)\geq \alpha+1$ if there is in $M^{eq}$ a definable set $D_1$, a definable finite-to-one map $\pi\colon D_1 \to D$ and a map $f\colon D_1 \to D_2$ such that $D_2$ is infinite and $\rk(f^{-1}(d))\geq \alpha$ for all $d\in D_2$.
  \item for limit $\lambda$, $\rk(D)\geq \lambda$ if $\rk(D)\geq \alpha$ for all $\alpha<\lambda$.
  \end{itemize}  
\end{definition}

The rank of a definable set $D$ is either an ordinal or $\infty$ in the case where $\rk(D)\geq \alpha$ for all $\alpha$. We say that a structure $M$ is \emph{ranked} if $rk(M)<\infty$. The rank of a type $\tp(a/b)$, denoted $\rk(a/b)$, is the minimal rank of a $b$-definable set containing $a$ (in $M^{eq}$).

\medskip
To get some intuition on this definition, start by considering the case where $D_1 = D$ and $\pi$ is the identity. Then the definition implies that if there is $f\colon D \to D_2$ with $D_2$ infinite (hence of rank $\geq 1$) and all fibers of rank $n$, then the rank of $D$ is at least $n+1$. In fact, we will see that in structures where the rank is finite, we have the stronger property that given any definable $f\colon D\to D_2$, if $\rk(D_2) = n$ and all fibers of $f$ have rank $m$, then $D$ has rank $n+m$. This is a familiar property shared by many notions of dimension. It would be tempting to define the rank as the minimal notion of dimension satisfying this property (along with the first condition above on infinite sets). In fact we see that the definition we gave is slightly more complicated. We add the possibility to replace $D$ by a finite cover $D_1$ of it. This does not change the rank ($\rk(D_1) = \rk(D)$, as we will see below), but might be required to find the map $f$. Here is an example illustrating  why this can be necessary.

Let $M$ be an infinite set with no structure and let $D$ be the set of subsets of $M$ of size 2. This is naturally a definable set in $M^{eq}$: it is a quotient of $D_1:=\{(a,b)\in M^2 : a\neq b\} \subseteq M^2$. The projection map $\pi \colon D_1 \to D$ has all fibers of size 2. It is not too hard to see that there is no definable map $f:D\to D_2$ to some infinite definable set $D_2$ in $M$ with infinite fibers (intuitively, we cannot choose an element from each subset in a definable way). However, there is such a map $f\colon D_1 \to M$, for instance $f \colon (a,b) \mapsto a$ and hence $D_1$ has rank 2 as it should be.

Definition \ref{def:rank0} is rather impractical given the quantification on the finite cover $D_1$ and the fact the it takes place in $M^{eq}$. We now give a second definition which might be more palatable, and prove right after that the two definitions are equivalent.

By a \emph{uniformly definable} family $(X_t : t\in E)$, we mean that $E$ is a definable set and there is a formula $\phi(x;y)$ such that for $t\in E$, $X_t$ is the set defined by $\phi(x;t)$. For the purposes of the following definition, we will say that a family $(X_t : t\in E)$ is \emph{weakly $k$-inconsistent} ($k$ a natural number) if any $k$ \emph{pairwise distinct} members of the family have trivial intersection.

\begin{definition}\label{def:rank}
  This definition takes place in $M$ (not $M^{eq}$). Given a definable set $D\subseteq M^{k}$ and ordinal $\alpha$, we define inductively $\widetilde {\rk}(D)\geq \alpha$:
  \begin{itemize}
  \item $\widetilde \rk(D)\geq 1$ if and only if $D$ is infinite;
  \item $\widetilde \rk(D)\geq \alpha+1$ if there is a uniformly definable family $(X_t:t\in E)$ of subsets of $D$ which is weakly $k$-inconsistent for some $k$, contains infinitely many pairwise distinct sets, and such that $\widetilde \rk(X_t)\geq \alpha$ for each $t\in E$;
  \item for limit $\lambda$, $\widetilde \rk(D)\geq \lambda$ if $\widetilde \rk(D)\geq \alpha$ for all $\alpha<\lambda$.
  \end{itemize}
\end{definition}

Note that since $M^{eq}$ does not add new definable subsets of $M^k$, this definition in fact does not change if we allow $E$ and the family $(X_t:t\in E)$ to be definable in $M^{eq}$.

\begin{prop}
For any structure $M$ and any definable set $D\subseteq M^k$, we have $\rk(D) = \widetilde \rk(D)$.
\end{prop}
\begin{proof}

We first show that for any $\alpha$, if the definable map $\pi\colon D' \to D$ is finite-to-one and onto, then $\widetilde\rk(D')\geq \alpha \iff \widetilde\rk(D)\geq \alpha$. This can be shown by an easy induction on $\alpha$: In one direction, given a weakly $k$-inconsistent family of subsets of $D$, we can pull it back by $\pi$ to obtain a weakly $k$-inconsistent family of subsets of $D'$. In the other direction, let $l$ be the maximal size of a fiber of $\pi$ (which exists by $\omega$-categoricity). Then if $(X_t:t\in E)$ is a weakly $k$-inconsistent family of subsets of $D'$, it is not hard to see that the family of images $(\pi(X_t):t\in E)$ is a weakly $kl$-inconsistent family of subsets of $D$.

We now show by induction on $\alpha$ that $\rk(D) \geq \alpha \iff \widetilde \rk(D)\geq \alpha$. For $\alpha = 1$, this follows from the first point in both definitions. Assume we know it for some $\alpha$ and that $\rk(D)\geq \alpha+1$ as witnessed by $D_1, D_2, \pi$ and $f$. Here, $D_1$ and $D_2$ are definable subsets of $M^{eq}$. For $d\in D_2$, let $X_d = \pi(f^{-1}(t)) \subseteq D$. Then the family $(X_d : d\in D_2)$ is uniformly definable in $M^{eq}$. By construction, there is a natural number $k$ such that this family is $k$-inconsistent. By definition of $M^{eq}$, there is some $E\subseteq M^l$ and a definable surjection $g\colon E \to D_2$ in $M^{eq}$. For $t \in E$, let $X'_t = X_{g(t)}$. Then the family $(X'_t:t\in E)$ is uniformly definable in $M^{eq}$, hence also in $M$ since $M^{eq}$ does not add any definable subsets in powers of $M$. It is also weakly $k$-inconsistent. By induction hypothesis $\widetilde \rk(f^{-1}(d)) \geq \alpha$ for each $d\in D_2$ hence also $\widetilde \rk(X'_t)\geq \alpha$ for each $t\in E$ by the first paragraph of this proof. Therefore $\widetilde \rk (D)\geq \alpha+1$.

Assume now that $\widetilde \rk(D) \geq \alpha+1$ as witnessed by the weakly $k$-inconsistent family $(X_t:t\in E)$. Let $\sim$ be the equivalence relation on $E$ defined by $t\sim t'$ if $X_t = X_{t'}$. Let $D_2 = E/\sim$ seen as a subset of $M^{eq}$. For $t\in D_2$, we define $X_t$ in the natural way. Let $D_1 = \{(a,t)\in D\times D_2 : a\in X_t\}$. Then the canonical projection $\pi\colon D_1 \to D$ is $k$-to-one. Let $f:D_1 \to D_2$ be the other canonical projection. For any $t\in D_2$, $\widetilde \rk(f^{-1}(t))\geq \alpha$ as it admits a one-to-one projection to $X_t$. By induction, $\rk(f^{-1}(t))\geq \alpha$. Hence by definition, $\rk(D)\geq \alpha+1$.
\end{proof}

If $M_=$ is an infinite set with no structure, then it has rank 1 since every infinite definable subset of $M_=$ is cofinite, hence there are no infinite weakly $k$-inconsistent families of definable subsets of $M_=$. Another example of a rank 1 structure is a dense linear order $(M,\leq)$. Here is a sketch of a proof of this: If $(X_t:t\in E)$ is a uniformly definable weakly $k$-inconsistent family of intervals of $M$, then the set of left endpoints of those intervals cannot contain an interval, hence it has to be finite. Similarly for the set of right end points. It follows that the family contains only finitely many  sets. A same argument can be made for a family of unions of $n$ intervals. By quantifier elimination, every definable set is a union of finitely many intervals, hence we are done.

\medskip
We state some basic properties of the rank. See \cite[Section 2.2.1]{CherHru} for proofs of (1)-(5); (6) and (7) are simple consequences of (5).

\begin{prop}\label{prop:basic rank}
	\begin{enumerate}
		\item $\rk(a/b)=0$ if and only if $a\in \acl^{eq}(b)$.
		\item $\rk(D_1 \cup D_2)=\max (\rk(D_1),\rk(D_2))$.
		\item If $B_1\subseteq B_2$, then $\rk(a/B_1)\geq \rk(a/B_2)$
		\item If $D$ is definable over $B$, then there is $a\in D$ such that $\rk(a/B)=\rk(D)$.
		\item If $\rk(a/bc)$ and $\rk(b/c)$ are finite, then so is $\rk(ab/c)$ and we have \[\rk(ab/c)=\rk(a/bc)+\rk(b/c).\]
		  In particular, if $a'\in \acl^{eq}(ab)$, then by point 1 above, $\rk(a'/ab)=0$, hence $\rk(a'/b)\leq  \rk(aa'/b)= \rk(a/b)$.
                \item In a finite rank structure we have \[\rk(ab) = \rk(a/b) + \rk(b) = \rk(b/a)+\rk(a).\]
                \item For tuples $a, b, a'$ in a finite rank structure, we have $a' \in \acl(ab)$ if and only if $\rk(aa'/b) = \rk(a/b)$. If this holds, we have in particular $\rk(a'/b) \leq \rk(a/b)$. 
	\end{enumerate}
\end{prop}

From point 5, we deduce that if $M$ has finite rank, then any finite tuple of elements of $M$, or indeed $M^{eq}$, has finite rank.

The operator $\acl$ always defines a closure relation in the sense that $\acl(\acl(A))=\acl(A)$ for all $A$ and $\acl(A)\subseteq \acl(B)$ whenever $A\subseteq B$. Assuming that $\rk(M)=1$, then it furthermore satisfies the exchange property: for all $A\subseteq M$ and two singletons $a,b\in M$, we have:
\[b\in \acl(Aa)\setminus \acl(A) \iff a\in \acl(Ab)\setminus \acl(A).\]

A set $M$ equipped with a closure operator is a \emph{pregeometry} if the closure operator satisfies the exchange property. Thus if $M$ has rank one, then algebraic closure defines a pregeometry on $M$.

We can then define independent sets and bases as one does for vector spaces, with $\rk$ playing the role of dimension. We will only make mild use of this fact.


\begin{lemma}\label{lem:exist maximal rank}
  Let $p(x)\in S(A)$ be a type of finite rank $n$ and let $B\supseteq A$, then there is a tuple $d$ realizing $p$ such that $\rk(d/B) = n$.
\end{lemma}
\begin{proof}
  Let $\pi(x)$ be the conjunction of all formulas of the form $\neg \psi(x,b)$, with $b\in B$ such that for some formula $\phi(x,a)$ in $p(x)$, $\rk(\phi(x,a)\wedge \psi(x,b))<n$. Then $\pi(x)$ contains $p(x)$ and by point 2 above, it is finitely consistent. Now take $d$ realizing $\pi(x)$.
\end{proof}

\begin{lemma}\label{lem:equivalence relation rank}
  Let $A$ be a set of parameters; let $D$ be an $A$-definable set and $E$ an $A$-definable equivalence relation on $D$. Let $D_1 \subseteq D$ be an $E$-class which is not algebraic over $A$ (in other words, the corresponding element of $M^{eq}$ is not in $\acl^{eq}(A)$). Then $\rk(D_1)<\rk(D)$.
\end{lemma}
\begin{proof}
Seeing $D/E$ as a subset of $M^{eq}$, let $d_1 \in D/E$ be the element of $M^{eq}$ corresponding to $D_1$. Let $F \subseteq D/E$ be the definable set of elements having the same type as $d_1$ over $A$. Consider the definable family $(D_t : t\in F)$, where $D_t\subseteq D$ is the $E$-class coded by $t$. Then for each $t\in F$, $\rk(D_t) = \rk(D_1)$ since $\tp(t/A) = \tp(d_1/A)$. As $d_1$ is not algebraic over $A$, $F$ is infinite. Furthermore the $D_t$'s are pairwise disjoint. It follows form the definition of the rank that $\rk(D_1)=\rk(D_t) < \rk(D)$.
\end{proof}

Still following \cite{CherHru}, we define rank independence.

\begin{definition}
	($M$ has finite rank.) Say that two tuples $a$ and $b$ are \emph{independent over $E$} and write $a\ind_E b$ if \[ \rk(ab/E)=\rk(a/E)+\rk(b/E).\]
\end{definition}

This is a symmetric notion in $a$ and $b$ and it satisfies transitivity: $a$ and $bc$ are independent over $E$ if and only if $a$ and $b$ are independent over $Ec$ and $a$ and $c$ are independent over $E$.

Finally, we show the following useful property.

\begin{lemme}\label{lem:independence_acl}
($M$ has finite rank.) Two tuples $a$ and $b$ are independent over $e$ if and only if $\rk(a/be) = \rk(a/e)$. Furthermore if this holds then $\acl^{eq}(ae)\cap \acl^{eq}(be)=\acl^{eq}(e)$.
\end{lemme}
\begin{proof}
We have in general $\rk(ab/e) = \rk(a/be) + \rk(b/e)$. The first statement then follows at once form the definition of independence.
  
Let now $d\in \acl^{eq}(ae) \cap \acl^{eq}(be)$ be a tuple. Assume that $d\notin \acl^{eq}(e)$, so that $\rk(d/e)\geq 1$.
We have:
\begin{align*}
\rk(a/bde) &\leq \rk(a/de) \\
&= \rk(ad/e) - \rk(d/e) \quad &\text{ using Proposition \ref{prop:basic rank}(6)}\\
&=\rk(a/e) - \rk(d/e) \quad &\text{ by Proposition \ref{prop:basic rank}(1) and }d\in \acl^{eq}(ae)\\
&\leq \rk(a/e) - 1 \quad &\text{ as }d\notin \acl^{eq}(e)\\
&<\rk(a/e).
\end{align*}

On the other hand, since $d\in \acl^{eq}(be)$, $\rk(abde) = \rk(abe)$ and $\rk(bde) = \rk(be)$. Applying Proposition \ref{prop:basic rank}(6) on both sides of

\[ \rk(abde) = \rk(abe)\]
we get \[ \rk(a/bde) +\rk(bde) = \rk(a/be) + \rk(be)\]
and since $\rk(bde) = \rk(be)$, \[ \rk(a/bde) = \rk(a/be).\]

Finally, $\rk(a/be)= \rk(a/e)$ since $a\ind_e b$.

From the series of inequalities above we have $\rk(a/bde) < \rk(a/e)$, but we have just shown $\rk(a/bde) = \rk(a/e)$; contradiction.
\end{proof}

\subsection{Stability}

Recall that a formula $\phi(x;y)$ is \emph{stable} (in some structure $M$) if for some integer $k$, we cannot find tuples $(a_i:i<k)$ and $(b_j:j<k)$ such that \[\phi(a_i;b_j) \iff i\leq j.\]

We say that the structure $M$ is \emph{stable} if all formulas are stable. Stability is preserved under elementary equivalence and we say that a complete theory $T$ is stable if some/any model of $T$ is stable.

We are concerned in this paper with unstable structures, but stable formulas will appear briefly at the end of the analysis in Section \ref{sec:stable structure}. There, we will need the following fact, which the reader not well acquainted with stability theory can take as a black box.





\begin{fact}\label{fact:finitely many non-forking extensions}
	Let $M$ be a ranked $\omega$-categorical structure and let $\phi(x;y)$ be a stable formula. Let $p\in S(A)$ be a type over some set $A\subseteq M$ and let $B\subseteq M$ be any set, then the set \[\left\{\tp_{\phi}(a/B) : a\models p, a\ind_A B\right\}\] is finite.
\end{fact}
\begin{proof}
	(Assuming knowledge of stability theory: see for instance \cite[Chapter 1]{PillayBook}.) First, note that by \cite[Theorem 5.1.1]{On1}, forking and thorn-forking are the same for stable formulas. Hence if $a\models p$, $a\ind_A B$, then the partial type $p\cup \tp_{\phi}(a/B)$ does not fork over $A$. Since $\phi$ is stable, there are only finitely many non-forking extensions of $p$ to a $\phi$-type over $B$.
\end{proof}



\subsubsection{Strongly minimal sets}
We check that Theorem \ref{th:main1} holds in the stable case and for that assume familiarity with stability theory. None of this will be used later.

A structure $M$ is \emph{strongly minimal} if for any $N$ elementarily equivalent to $M$, any definable (over parameters) subset of $N$ is either finite or cofinite. If $M$ is $\omega$-categorical, then it is enough to check the condition for $N=M$. The classification of strongly minimal primitive $\omega$-categorical structures was established by Zilber \cite{Zilber_book} using model-theoretic methods. The paper \cite{CHL} gives an exposition of this result, as well as a shorter proof attributed to Cherlin and Mills, using the classification of finite simple groups. The results are expressed in terms of the {geometry} coming from algebraic closure. We have explained in Section \ref{sec:rank} how any rank 1 structure is equipped with a pregeometry whose closure operation is given by algebraic closure. If $M$ is primitive, this pregeometry is in fact a \emph{geometry}, meaning that $\acl(a)=\{a\}$ for any element $a\in M$. By the $\acl$-geometry of $M$, we mean the set $M$ equipped with the closure operator $\acl$.

\begin{fact}\label{fact:Zilber}
	If $M$ is strongly minimal, primitive and $\omega$-categorical, then either:
	\begin{enumerate}
		\item $M$ is a pure set;
		\item the $\acl$-geometry on $M$ is that of an infinite-dimensional affine space over a finite field;
		\item the $\acl$-geometry on $M$ is that of an infinite-dimensional projective space over a finite field.
	\end{enumerate}
\end{fact}

Cases 2 and 3 do not completely determine $M$ up to bi-definability, but they do determine it up to finitely many possibilities corresponding to automorphism groups $G$ with $AGL_{\omega}(F_q)\subseteq G \subseteq A\Gamma L_{\omega}(F_q)$ in the affine case and $PGL_{\omega}(F_q)\subseteq G \subseteq P\Gamma L_{\omega}(F_q)$ in the projective case.

\begin{prop}\label{prop:stable main}
	A rank 1, primitive, stable, $\omega$-categorical structure $M$ is strongly minimal. For a given $n<\omega$, there are, up to bi-definability, finitely many such structures having at most $n$ 4-types.
\end{prop}
\begin{proof}
If $M$ is stable of finite rank, then rank-independence is the same thing as forking-indepedence: see \cite[Theorem 5.1.1]{On1}. Thus if $M$ is stable of rank 1, it is superstable of $U$-rank 1.
If $M$ is furthermore primitive, then $x=x$ is a complete strong type over $\emptyset$ and therefore for any definable set $D\subseteq M$, either $D$ or its complement forks over $\emptyset$. Hence by $U$-rank 1, either $D$ or its complement is finite. Therefore a stable, rank 1, primitive, $\omega$-categorical structure is strongly minimal.

	Fact \ref{fact:Zilber} and the remark following it describe the possibilities. We can assume that $M$ is not a pure set. Assume first that $M$ is affine over a field $F_q$, $q=p^n$. Then if we fix a point $a$ as the origin, making $M$ linear, and take $b,c$ colinear, we have $c=\lambda\cdot b$ for some $\lambda \in F_q$, defined in the worst case up to an element of $Gal(F_q/F_p)$. That Galois group has size $n$ and therefore the number of orbits goes to infinity with $q$. Hence so does the number of 3-types.
	The projective case is similar, except that we need to name two points to serve as 0 and $\infty$ and obtain that the number of 4-types goes to infinity with $q$.
\end{proof}

\subsection{NIP and op-dimension}


We recall some basic facts about NIP theories and refer the reader to \cite{NIPbook} for more details.

\begin{definition}\label{def_nip}
	A formula $\phi(x;y)$ is \emph{NIP} in $M$ if for some integer $k$, we cannot find tuples $(a_i:i<k)$ and $(b_J: J\in \mathfrak P(k))$ in $M$ with: \[M\models \phi(a_i;b_J)\iff i\in J.\]
\end{definition}

If a formula $\phi(x;y)$ is NIP, then it stays so in any structure $N$ elementarily equivalent to $M$. We say that the theory $T$ is NIP if for some/any model of $T$, all formulas are NIP.

By a result of Shelah, if all formulas $\phi(x;y)$ with $|x|=1$ are NIP, then the theory is NIP. Stable theories are NIP and so is for example the theory of dense linear orders.

The NIP condition can be characterized by counting $\phi$-types over finite sets. See \cite[Chapter 6]{NIPbook}. In the finitely homogeneous case, this becomes a particularly natural condition.

\begin{fact}\label{fait:sauer-shelah}
	A structure $M$ homogeneous in a finite relational language is NIP if and only if there is a polynomial $P(X)$ such that the number of types over any finite set $A$ is bounded by $P(|A|)$.
\end{fact}

Note in particular, that if $M$ is NIP and homogeneous in a finite relational language, then the size of $S_n(\emptyset)$ is bounded by $P(1)\cdot P(2) \cdots P(n-1)$, where $P(X)$ is the polynomial given by the previous fact. Hence $|S_n(\emptyset)|=O(2^{cn\ln(n)})$ for some $c>0$.

We now give a short account of \cite{linear_orders} which establishes that NIP unstable theories interpret linear orders. First, we define op-dimension as in \cite{GuinHill}, which will allow us to determine how many independent orders we can hope to find.

\begin{definition}
	An \emph{ird-pattern} of length $\kappa$ for the partial type $\pi(x)$ is given by:
	
	\begin{itemize}
		\item a family $(\phi_\alpha(x;y_\alpha):\alpha<\kappa)$ of formulas;
		\item an array $(b_{\alpha,k}:\alpha<\kappa,k<\omega)$ of tuples, with $|b_{\alpha,k}|=|y_\alpha|$;
	\end{itemize}
	such that for any $\eta\colon  \kappa \to \omega$, there is $a_\eta\models \pi(x)$ such that for any $\alpha<\kappa$ and $k<\omega$, we have \[ \models \phi_\alpha(a_\eta;b_{\alpha,k}) \iff \eta(\alpha)<k.\]
\end{definition}

\begin{rem}
	This definition is from \cite[III.7.1]{Sh:c}. The letters \emph{ird} stand for \emph{independent orders}.
\end{rem}

\begin{definition}
	We say that $T$ has \emph{op-dimension} less than $\kappa$, and write $\dlr(T)<\kappa$ if, in a saturated model of $T$, there is no ird-pattern of length $\kappa$ for the partial type $x=x$.
\end{definition}

If a structure is NIP, then it has op-dimension less than $|T|^+$ (otherwise, we can assume $\phi_\alpha=\phi$ is constant and then $\phi$ has IP: we can take $\{b_{\alpha,0}:\alpha<\omega\}$ as the $a_i$'s in Definition \ref{def_nip}). Conversely, if for some cardinal $\kappa$, we have $\dlr(T)<\kappa$, then $T$ is NIP. (If $\phi(x;y)$ has IP, we can find by compactness an ird-pattern of any length with $\phi_\alpha =\phi$.)

We also define the op-dimension of a partial type $p(\bar x)$: $\dlr(p)<\kappa$ if, in a saturated model of $T$, there is no ird-pattern of length $\kappa$ for $p(\bar x)$. We let $\dlr(\bar a) = \dlr(\tp(\bar a/\emptyset))$ and $\dlr(\bar a/A) = \dlr(\tp(\bar a/A))$.

We say that $\dlr(p)=n$ if we have $\dlr(p)<n+1$, but not $\dlr(p)<n$ (and the same for $T$ instead of $p$).

\begin{fact}\label{fact:opdim}
	Op-dimension is sub-additive: $\dlr(\bar a\bar b/A) \leq \dlr(\bar b/A\bar a) + \dlr(\bar a/A)$. In particular, if $\bar b\subseteq \acl^{eq}(A\bar a)$, then $\dlr(\bar a\bar b/A) = \dlr(\bar a/A)$.
\end{fact}
\begin{proof}
	The first statement is \cite[Theorem 2.2]{GuinHill}. See also \cite[Section 4]{linear_orders}. The ``in particular" part follows from the fact that if $\bar b$ is algebraic over $\bar a$, then $\dlr(\bar b/\bar a)=0$ which is clear from the definition since the points $a_\eta$ there have to be pairwise distinct.
\end{proof}

By a linear quasi-order, we mean a reflexive, transitive relation $\leq$ for which any two elements are related. If $\leq$ is a linear quasi-order, then the associated strict order $<$ is defined by \[a<b \iff a\leq b \wedge \neg(b\leq a).\]

Furthermore, the relation $aEb \iff (a\leq b)\wedge (b\leq a)$ is an equivalence relation and $\leq$ induces a linear order on the quotient.

The main result of \cite{linear_orders} in the $\omega$-categorical case is the following.


\begin{fact}[\cite{linear_orders}, Theorem 6.14]\label{fact:linear orders from op-dimension}
If the theory $T$ is $\omega$-categorical, NIP, $\dlr(T)\geq n>0$, then there is a finite set $A$, a set $D$ definable and transitive over $A$ and $n$ $A$-definable linear quasi-orders $\leq_1,\ldots,\leq_n$ on $D$, such that the structure $(D;\leq_1,\ldots,\leq_n)$ contains an isomorphic copy of every finite structure $(X_0;\leq_1,\ldots,\leq_n)$ equip\-ped with $n$ linear orders.
\end{fact}

Note that by transitivity, for each $i$, the quotient of $D$ by the equivalence relation associated with $\leq_i$ is infinite and, using $\omega$-categoricity, $\leq_i$ induces on it a dense linear order without endpoints.

\subsubsection{Distality}

Distality was introduced in \cite{distal}. It is meant to capture the notion of a purely unstable NIP structure. We give here the equivalent definition from \cite{ExtDef2}.

\begin{definition}\label{def:distal}
	A structure $M$ is called \emph{distal} if for every formula $\phi(x;y)$, there is a formula $\psi(x;z)$ such that for any finite set $A\subseteq M$ and tuple $a\in M^{|x|}$, there is $d\in A^{|z|}$ such that $\psi(a;d)$ holds and for any instance $\phi(x;b)\in \tp(a/A)$, we have the implication \[M\models (\forall x) \psi(x;d)\to \phi(x;b).\]
\end{definition}

Assume that $M$ is homogeneous in a finite relational language. Then if $M$ is distal, there is an integer $k$ such that for any finite set $A$ and singleton $a\in M$, there is $A_0\subseteq A$ of size $\leq k$ such that $\tp(a/A_0)\vdash \tp(a/A)$. (That is, if $\tp(a'/A_0)=\tp(a/A_0)$, then $\tp(a'/A)=\tp(a/A)$.) In fact, the converse is also true, as can be seen by induction on $|x|$ in the definition above, but we will not need this.

For instance, DLO is distal, and we can take $k=2$. We will see in Theorem \ref{th:distal is fin axiom} that a distal finitely homogeneous structure is always finitely axiomatizable.

\section{Linear orders}\label{sec:linear orders}

We will consider definable linear orders $(V,\leq)$, meaning that the underlying set $V$ is parameter-definable and so is the order relation $\leq$. We will often abuse notation by denoting the pair $(V,\leq)$ by $V$, sometimes by $\leq$. If we have two definable orders $(V_0,\leq_0)$, $(V_1,\leq_1)$, it may happen that the underlying sets $V_0$, $V_1$ are equal. This will, however, be irrelevant for most of what we say and it might be more convenient to think of $V_0$ and $V_1$ as two disjoint copies of the same set. In any case, $V_0$ will mean the set equipped with the order $\leq_0$ and $V_1$ the set equipped with the order $\leq_1$. The reverse of the order $(V,\leq)$ is $(V,\geq)$.

Orders are always equipped with the order topology, and product of orders with the product topology. Hence, in the situation above, $V_0\times V_1$ is equipped with the product topology coming from $\leq_0$ on the first coordinate and $\leq_1$ on the second, regardless of whether the underlying sets $V_0$ and $V_1$ are equal or not.


\begin{lemme}\label{lem0} Let $(V,\leq)$ be a 0-definable infinite linear order, which is a complete type over $\emptyset$. Then the order $\leq$ is dense and for any $a\in V$, $\acl(a)\cap V = \{a\}$.
\end{lemme}
\begin{proof}
If $\leq$ is not dense, then some point $a\in V$ has an immediate successor. Since $V$ is a complete type over $\emptyset$, all points have a successor and hence the order is discrete. By $\omega$-categoricity, $V$ is finite.

If $b\in \acl(a)\cap V$, say $b>a$, then again as $V$ is a complete type, there is $b_1>b$, $b_1\in \acl(b)$ and iteratively $b_{k+1}>b_k$, $b_{k+1}\in \acl(b_k)$. This gives infinitely many elements in $\acl(a)$, contradicting $\omega$-categoricity.
\end{proof}

A \emph{convex equivalence relation} on an order $(V,\leq)$ is an equivalence relation $E$ with convex classes: that is: \[ a\leq b\leq c \wedge a\rel{E}c \Longrightarrow a\rel{E}b.\] Such a relation is non-trivial if it has more than one class and is not equality.

\begin{defi}
	Let $(V,\leq)$ be an $A$-definable linear order.
\begin{itemize}
	\item We say that $(V,\leq)$ has \emph{topological rank 1} if it does not admit any definable (over parameters) convex equivalence relation $E$ with infinitely many infinite classes.
	\item We say that $(V,\leq)$ is \emph{weakly transitive} over $A$ if it is a dense order and any $A$-definable subset of $V$ is either empty or dense in $V$.
	\item We say that $(V,\leq)$ is \emph{minimal} over $A$ if it is weakly transitive over $A$, and has topological rank 1. If $A=\emptyset$, then we omit it.
\end{itemize}
\end{defi}

The name \emph{topological rank 1} comes from the fact that a rank 1 structure, in the sense of Section \ref{sec:rank}, cannot have a definable equivalence relation with infinitely many infinite classes. Here, we forbid such equivalence relations that have convex classes. We will not define topological rank in general.

Note that if $(V,\leq)$ is transitive over $A$, in the sense that it is a complete type over $A$, then it is weakly transitive over $A$. As an example, consider the structure $(\mathbb Q;\leq ,P)$, where $\leq$ is the usual order on $\mathbb Q$ and $P(x)$ is a unary predicate that is dense co-dense in $\mathbb Q$. Then the order $(\mathbb Q;\leq)$ is weakly transitive (over $\emptyset$), but is not a complete type.

\begin{lemma}\label{lem:preservation under dense subsets}
	\begin{enumerate}
		\item If $(V,\leq)$ is weakly transitive over $A$, then it has neither first nor last element.
		\item A definable subset of a topological rank 1 linear order has itself topological rank 1.
		\item If $(V,\leq)$ is an $A$-definable dense order without first or last element, and $W\subseteq V$ is a dense $A$-definable subset of $V$, then $V$ has topological rank 1 (resp. is weakly transitive over $A$, resp. is minimal over $A$) if and only if $W$ has the same property
	\end{enumerate}
\end{lemma}
\begin{proof}
	1. If $V$ has a first (resp. last) element $a$, then $\{a\}$ is an $A$-definable subset of $V$ that is neither empty nor dense.
	
	2. Let $(V,\leq)$ have topological rank 1 and $W\subseteq V$ be definable (over some parameters). Let $E$ be a definable convex equivalence relation on $W$ with infinitely many infinite classes. Define a relation $\overline E$ on $V$ by: $\overline E(a,b)$ holds if all the points of $W$ in the interval $a\leq x\leq b$ are in one $E$-equivalence class. Then $\overline E$ is a definable convex equivalence relation on $V$ with infinitely many infinite classes. This contradicts  $V$ having topological rank 1.
	
	3. If $E$ is a definable convex equivalence relation on $V$, then its restriction $E|_W$ to $W$ is also a definable convex equivalence relation. Furthermore if $E$ has infinitely many infinite classes on $V$, each of those classes has infinite intersection with $W$ by density, hence $E|_W$ shows that $W$ does not have topological rank 1. Along with the first point, this shows that $V$ has topological rank 1 if and only if $W$ has topological rank 1.
	
	Assume that $V$ is not weakly transitive. Let $X\subseteq V$ be $A$-definable, non-empty and not dense in $V$. If $X$ is finite, let $a$ be its first element. Then the set $\{x\in W: x<a\}$ is $A$-definable, non-empty (as $V$ has no first element and $W$ is dense in $V$) and not dense in $W$ (as $V$ has no last element). This shows that $W$ is not weakly transitive over $A$. If now $X$ is infinite, let $\overline X$ be its topological closure in $V$. Then $W\cap \overline X$ is $A$-definable, non-empty and not dense in $W$. Thus $W$ is not weakly transitive over $A$. Conversely, it is immediate that if $W$ is not weakly transitive over $A$ then neither is $V$.
\end{proof}

\begin{lemme}\label{lem1} Let $(V,\leq)$ be a definable dense order of topological rank 1. Then any definable closed (or open) subset of $V$ is a finite union of convex sets.
\end{lemme}
\begin{proof}
Let $X\subseteq V$ be a definable closed subset. Consider the equivalence relation $E_X$ which holds of a pair $(a,b)$ in $V^2$ if either $a=b$ or there is no element of $X$ in the interval $a\leq x\leq b$. This is a convex equivalence relation. Moreover, any $E_X$-class is either of the form $\{a\}$, $a\in X$, or of the form $a<x<b$, with $a,b\in X\cup \{\pm \infty\}$. Since $(V,\leq)$ is dense, classes of the second type are infinite. By topological rank 1, there can be only finitely many such classes. This implies that the complement of $X$ is a finite union of convex sets. Then so is $X$.
\end{proof}

Note that with $(V, \leq)$ as above, if $A\subseteq V$ is the closure of a complete type $X$ over some $\bar c$, then $A$ is a convex set. Indeed, if say $A = A_0\cup \cdots\cup A_{n-1}$ where the $A_i$'s are disjoint convex sets and $n$ is minimal, then $A$ and each $A_i$ is definable over $\bar c$. Thus for each $i<n$, $X\cap A_i$ is definable over $\bar c$ and non-empty. By completeness of $X$, $n=1$.

\subsection{Intertwinings}\label{sec:intertwinings}

Let $(V,\leq)$ be an $A$-definable dense order with no first or last element. By a \emph{cut} in $V$ we mean an initial segment of it which is neither empty nor the whole of $V$ and has no last element. We let $\overline V$ be the set of definable (over any parameters) cuts of $V$. Let $\phi(x;y)$ be a formula without parameters. The set $C_{\phi}:=\{b : \phi(V;b)$ is a cut of $V\}$ is definable over $A$. The set of cuts of $V$ definable by a formula of the form $\phi(x;b)$ can be identified with the quotient of $C_{\phi}$ by the equivalence relation $b\sim b' \iff (\forall x\in V)(\phi(x;b)\leftrightarrow \phi(x;b'))$. Hence the set of cuts in $V$ that can be defined by an instance of $\phi(x;y)$ is naturally an $A$-definable set in $M^{eq}$. If $\Phi$ is a finite set of formulas as above, write $C_{\Phi}= \bigcup_{\phi\in \Phi} C_{\phi}$. This is also an $A$-definable set in $M^{eq}$. Now $\overline V=\bigcup_{\Phi} C_{\Phi}$, where $\Phi$ runs over all finite set of formulas of the form $\phi(x;y)$, is naturally a directed union of $A$-definable sets. (It would more rigorous to describe it as a direct limit of $A$-definable sets, but we will do without introducing such formalism.) In all arguments using $\overline V$, one can replace $\overline V$ with a big enough definable subset of it of the form $C_{\Phi}$.

If there is a finite set $\Phi$ of formulas such that every definable cut of $V$ can be defined by some $\phi(x;b)$ for $\phi(x;y) \in \Phi$, then $\overline V$ can be indentified with $C_{\Phi}$. In this case, $\overline V$ is a bona fide definable set in $M^{eq}$. In that case, we will say that \emph{$\overline V$ is definable}. We will see in Corollary \ref{cor:trivial geometry} that this is true in a rank 1 structure. Hence the reader will not loose much by assuming that this is the case when reading this paper.

A function $f\colon X\to \overline V$ is said to be definable over some $B\supseteq A$ if there is a $B$-definable binary relation $F\subseteq X\times V$ such that for all $a\in X$, the fiber $F_a:=\{x\in V : (a,x)\in F\} \subseteq V$ is equal to $f(a)$. This is consistent with the view of $\overline V$ as a union of definable sets: a function $f\colon X\to \overline V$ is $B$-definable if and only if it takes values inside a fixed definable subset $C_{\Phi}$ of $\overline V$ and is $B$-definable in the usual sense.

We identify $V$ with a (definable) subset of $\overline V$ by $a\mapsto \{x\in V: x< a\}$. The order $\leq$ naturally extends to $\overline V$, where it coincides with inclusion. Note that $V$ is dense in $\overline V$. Note also that if $a\in V$ and $c\in \overline V$, then $a\in c$ is equivalent to $a<c$, where $<$ is meant in $\overline V$ with the identification just discussed. We will use both notations.

\begin{lemme}\label{lem:minimal set of cuts}
Let $(V,\leq)$ be definable and minimal over some $A$. Any $A$-definable non-empty subset of $\overline V$ is dense in $\overline V$.	
\end{lemme}
\begin{proof}
	Let $X\subseteq \overline V$ be $A$-definable. We define a relation $E_X$ on $V$ by: \[a\rel{E_X}b \iff (a=b) \vee (\forall x\in \overline V)(a\leq x \leq  b \to x\notin X).\]
	Then $E_X$ is a convex $A$-definable equivalence relation on $V$ and by topological rank 1, it has only finitely many infinite classes. Assume it has an infinite class, then that class is $A$-definable and by weak transitivity, it is the whole of $V$. This implies that $X$ is empty. If there is no infinite class, then by density of $V$, all classes have one element, which implies that $X$ is dense in $\overline V$.
\end{proof}

\begin{lemma}\label{lem:slice of completion}
	Let $(V,\leq)$ be definable and minimal over some $A$ and let $W\subseteq \overline V$ be an $A$-definable subset of $\overline V$ containing $V$. Then $W$ is minimal over $A$.
\end{lemma}
\begin{proof}
	We know that $V$ is dense in $\overline V$, hence also in $W$. The result then follows from Lemma \ref{lem:preservation under dense subsets}.
\end{proof}

\begin{lemme}\label{lem:finite dcl}
Given a finite tuple $\bar a$ and an $\bar a$-definable dense order $V$, $\dcl^{eq}(\bar a)\cap \overline V$ is finite.
\end{lemme}
\begin{proof}
Formally, the conclusion says that there is some number $k<\omega$ such that $\dcl^{eq}(\bar a)\cap V_0$ has size at most $k$ for all $\bar a$-definable $V_0\subseteq \overline V$. Let $V_0$ be such a set and let $m_1<\cdots<m_n$ be in $\dcl^{eq}(\bar a)\cap V_0$. By density of $V$ in $\overline V$, we can find $b_0,b_1,\ldots,b_{n-1}\in V$ with $m_1< b_1<m_2< \cdots< b_{n-1}< m_n$. Each $b_i$ has a different type over $\bar a$, and hence there are at least $n-1$ different types of elements of $V$ over $\bar a$. Hence $\dcl^{eq}(\bar a)\cap V_0$ has size bounded by the number of 1-types over $\bar a$ of elements of $V$, which is finite by $\omega$-categoricity.
\end{proof}

\begin{defi}
Let $(V,\leq_V)$ and $(W,\leq_W)$ be orders, definable and weakly transitive over $A$. We say that they are \emph{intertwined} over $A$  if there is an $A$-definable non-decreasing map $f\colon V\to \overline W$. If $A$ is clear from the context, we omit it.
\end{defi}

Note that this is the same thing as saying that there is an $A$-definable binary relation $R\subseteq V\times W$ such that \[(a \bin{R} b)\wedge (a'\leq_V a) \wedge (b \leq_W b') \Longrightarrow a' \bin{R} b'.\] Indeed, the relation $R$ is defined from $f$ by \[x\bin{R}y \iff f(x)\leq_{\overline W} y \iff \neg F(x,y),\] where $F$ is associated to $f$ as above. Observe also that by weak transitivity, no element of $\overline W$ is definable over $A$, hence the image of $f$ has to be cofinal and coinitial in $\overline W$.

\begin{lemme}\label{lem:intertwining}
For any fixed $A$, intertwining is an equivalence relation on orders that are definable and weakly transitive over $A$.
\end{lemme}
\begin{proof}
Any order is intertwined with itself via the identity function. If $R$ as above is an intertwining relation from $V$ to $W$, then $R'$ defined by $x\bin{R'}y \iff \neg y\bin{R}x$ is an intertwining relation from $W$ to $V$. Finally if $R$ is an intertwining relation from $V$ to $W$ and $S$ an intertwining relation from $W$ to $Z$, then $T$ defined by $x\rel Ty \iff (\exists z)(x\rel Rz \wedge z\rel Sy)$ intertwines $V$ and $Z$.
\end{proof}

Working over some base $A$, let $V$ and $W$ be two weakly transitive orders and $f\colon V\to \overline W$ an intertwining map. If $W$ has topological rank 1, then the image of $f$ must be dense in $\overline W$ (otherwise we can define an equivalence relation as in the proof of Lemma \ref{lem1}; it cannot have finitely many classes as $W$ is weakly transitive). If $V$ has topological rank 1, then $f$ is injective: $f(x)=f(y)$ is a convex equivalence relation on $V$; it cannot have finitely many infinite classes by weak transitivity and cannot have infinitely many by topological rank 1. Hence all classes are singletons and $f$ is injective. We conclude that if both $V$ and $W$ have topological rank 1, an intertwining gives an increasing injection of $V$ into a dense subset of $\overline W$. Furthermore, the map $f$ extends to an increasing bijection $\tilde f\colon  \overline V \to \overline W$ defined as follows: if $c\in \overline V$ is a cut in $V$, seen as a subset of $V$, we let $\tilde f(c)=\{y\in W : y<f(x)$ for some $x\in c\}$. Since $f$ is increasing and $c$ has no last element, $\tilde f(c)$ also has no last element and is a definable cut in $W$. One sees at once that $\tilde f$ extends $f$ and is increasing. Also if $d\in \overline W$ is a definable cut in $W$, then $c:=\{x\in V:f(x)<d\}$ is a definable cut in $V$ and $\tilde f(c)=d$. Hence $\tilde f\colon \overline V\to \overline W$ is a bijection. It follows that we can---and will---think of $V$ and $W$ as having a common definable completion, or equivalently as being dense in each other's completion.


\begin{lemme}\label{lem:unique intertwining}
Working over $A$, if $V$ and $W$ are minimal linear orders which are intertwined, then there is a unique $A$-definable intertwining map $f\colon V\to \overline W$.
\end{lemme}
\begin{proof}
Assume that we are given two increasing maps $f,g\colon V\to \overline W$, both definable over $A$. Keeping only the parameters needed to define $V,W,f$ and $g$, we may assume that $A$ is finite. The two maps $f$ and $g$ extend uniquely to increasing bijections from $\overline V$ to $\overline W$, still denoted by $f$ and $g$. If for some $a\in V$, $f(a)<g(a)$, then we have $a<f^{-1}(g(a))$ and hence $g(a)<g(f^{-1}(g(a)))$. Continuing in this way we find \[a<f^{-1}(g(a))<f^{-1}(g(f^{-1}(g(a))))<\cdots,\] which gives infinitely many elements in $\dcl(Aa)\cap \overline V$, contradicting Lemma \ref{lem:finite dcl}.
\end{proof}

It will follow from Lemma \ref{lem:no self-intertwining} that even over a larger set of parameters, there cannot be another intertwining map from an interval of $V$ to an interval of $W$.

The following lemma will be useful later and can also be seen as a warm-up for the next proposition as the proof will have a similar flavor.

\begin{lemme}\label{lem:minimal suborder}
Let $(V,\leq)$ be definable and minimal over some $A$. There is a subset $B \supseteq A$ and a $B$-definable bounded convex subset $W\subseteq V$ such that $W$ is minimal over $B$.
\end{lemme}
\begin{proof}
Let $a\in V$ be any point. Assume that there is an element in $\dcl(Aa)\cap \overline V$ which is larger than $a$. Take $m\in \dcl(Aa)\cap \overline V$ to be minimal such. Then the convex subset $(a,m):=\{x\in V : a<x<m\}$ is $Aa$-definable. It has topological rank 1 since $V$ does and it is weakly transitive over $Aa$ by minimality of $m$ (no cut of it is definable over $Aa$). Hence it is minimal over $Aa$ and we have want we want by setting $B= Aa$ and $W=(a,m)$.

Assume now that $\dcl(Aa) \cap \overline V$ contains no element greater than $a$. Let $b\in V$, $b>a$. If $\dcl(Aab)\cap \overline V$ contains no element strictly between $a$ and $b$, then the interval $(a,b)\subseteq V$ is definable and minimal over $Aab$ and we can take $B=Aab$, $W=(a,b)$. If $\dcl(Aab)\cap \overline V$ contains an element strictly between $a$ and $b$, let $m$ be minimal such. Then the convex subset $(a,m)$ is definable and minimal over $Aab$.
\end{proof}

We now study definable subsets of cartesian powers of a minimal order.

\begin{prop} \label{lem2}
 Working over some $A$, let $(V,\leq)$ be a minimal definable linear order. Let $p(x_0,\ldots,x_{n-1})\in S(A)$ be a type in $V^n$ such that $p\vdash x_0<x_1<\ldots<x_{n-1}$. Then given open intervals $I_0<\cdots <I_{n-1}$ of $V$, we can find $a_i \in I_i$ such that $(a_0,\ldots,a_{n-1})\models p$.
\end{prop}
\begin{proof}
	For simplicity of notation, assume $A=\emptyset$. The strategy of the proof is as follows: we first ignore the type $p$ and produce by induction on $l<\omega$, types $r_l\in S_l(\emptyset)$ which satisfy the conclusion of the proposition. We then show how the existence of $r_{2n}$ implies that $p$ itself has the required density property by sandwiching elements of a realization of $p$ between elements of a realization of $r_{2n}$.
	
	For any finite tuple $\bar d$, let $m(\bar d)$ denote the maximal element of $\dcl^{eq}(\bar d)\cap \overline V$. Note that for a fixed tuple of variables $\bar y$, the relation $\phi(x;\bar y) : = x>m(\bar y)$ is invariant under $\aut(M)$, and therefore definable.
	
	We construct an increasing sequence of types $r_l(x_0,\ldots,x_{l-1})\in S(\emptyset)$, $l>0$, of elements of $V^l$. For $l=1$, let $a_0\in V$ by any element and set $r_1=\tp(a_0)$ and $m_0=m(a_0)\in \overline V$. Pick any point $a_1>m_0$ and let $r_2=\tp(a_0,a_1)$. We continue in this way: having constructed $r_l=\tp(a_0,\ldots,a_{l-1})$, let $m_{l-1}=m(a_{\leq l})$\footnote{Where $a_{\leq l}:= a_0,\ldots,a_l$}. Pick any $a_{l}>m_{l-1}$ and set $r_{l+1}=\tp(a_0,\ldots,a_{l})$. We note that \[r_{l+1}(x_0,\ldots,x_l) \vdash x_l > m(x_0,\ldots,x_{l-1}).\]
	
	This being done, let $I_0<\cdots<I_{l-1}$ be open intervals of $V$. We claim that we can find $(b_0,\ldots,b_{l-1})\models r_l$ such that $m(b_{\leq k})$ lies in $I_k$ for each $k$. We do this by induction. Assume that $b_{<k}$ have been selected and set $m=m(b_{<k})$ (if $k=0$, take $m=-\infty$). Define the relation $E_k$ on $V_{>m}$ by $v \rel{E_k} w$ if either $v=w$, or for no $s$ with $\tp(b_{<k},s)=r_{k+1}$ do we have $v\leq m(b_{<k}s)\leq w$. This is an equivalence relation with convex classes. By the topological rank 1 assumption, it must have finitely many infinite classes. The infima and suprema of those classes are elements of $\overline V$ definable over $b_{<k}$. However, by definition, no cut above $m(b_{<k})$ is definable over $b_{<k}$. Hence all classes of $E_k$ are finite and by density of the order, all classes have one element. It follows that we can find $b_k$ with $\tp(b_{\leq k})=r_{k+1}$ and $m(b_{\leq k})$ lying in $I_k$.
	
	Let now $p(x_0,\ldots,x_{n-1})$ be as in the statement of the lemma and $\bar a \models p$. Let $r=r_{2n}$. Then by the previous paragraph, we can find $\bar b\models r$ such that for each $k$, $m(b_{\leq 2k})<a_k<m(b_{\leq 2k+1})$. Pick open intervals $I_0<\cdots <I_{n-1}$ of $V$. For each $k$, let $J_{2k}<J_{2k+1}$ be two subintervals of $I_k$. Applying the previous paragraph again, we can find $\bar b'\models r$ such that for each $i$, $m(b'_{\leq i})\in J_i$. Since $\bar b$ and $\bar b'$ have the same type, there is $\sigma\in \aut(M)$ with $\sigma(\bar b)=\bar b'$. Let $\bar a'=\sigma(\bar a)$. We then have $m(b'_{\leq 2k})<a'_k<m(b'_{\leq 2k+1})$ for each $k$. By the choice of $\bar b'$, this implies $a'_k\in I_k$ as required.
\end{proof}

\begin{rem}\label{rem:lem2 in completion}
	Let $(V,\leq)$ be definable and minimal over $A$. Let $p(x_0,\ldots,x_{n-1})\in S(A)$ be a type in $\overline V^n$ such that $p\vdash x_0<\cdots <x_{n-1}$. Then there is some $A$-definable $W\subseteq \overline V$ containing $V$ such that $p$ lies in $W^n$. By Lemma \ref{lem:slice of completion}, $W$ is also minimal over $A$ and we can apply the previous proposition with $W$ instead of $V$. This shows that Proposition \ref{lem2} can be applied to types in $\overline V^n$ instead of $V^n$.
\end{rem}

\begin{corollary}\label{lem2_cor}
	Let $(V,\leq)$ be a minimal definable linear order over some $A$. Let $X\subseteq V^n$ be an $A$-definable subset, then the topological closure of $X$ is a boolean combination of sets of the form $x_i \leq x_j$.
\end{corollary}
\begin{proof}
	We can write $X=\bigcup_{i<n} Y_i$, where the $Y_i$'s are pairwise disjoint and each $Y_i$ is $A$-definable and defines a complete type over $A$. Since the closure of $X$ is the union of the closures of the $Y_i$'s, it is enough to prove the statement for each $Y_i$. We may therefore assume that $X$ defines a complete type over $A$. Let $(a_0,\ldots,a_{n-1})\in X$. For some permutation $\sigma$ of $\{0,\ldots,n-1\}$, we have $a_{\sigma(0)}\leq \ldots \leq a_{\sigma({n-1})}$. If the coordinates of $\bar a$ are pairwise distinct, then the previous proposition implies that $X$ is dense in the set defined by $x_{\sigma(0)} \leq \ldots \leq x_{\sigma({n-1})}$. In general, $X$ is dense in the intersection of that set with the set defined by the conjunction of the equations $x_{\sigma(i)} = x_{\sigma(i+1)}$ that hold in $\bar a$.
\end{proof}

In the end of this section, we give a more concrete description of intertwined orders and show that there is only one transitive structure composed of $n$ intertwined orders, up to isomorphism and permutation of the orders. (See Proposition \ref{prop:classification of intertwined orders} for a precise statement.)

\begin{prop}\label{prop:unique n intertwined orders}
	Consider the language $L_n = \{\leq, P_0,\ldots,P_{n-1},f_1,\ldots,f_{n-1}\}$, where the $P_i$'s are unary predicates and the $f_i$'s unary functions. Let the theory $T_n$ say that: 
	\begin{itemize}
		\item $\leq$ defines a dense linear order without endpoints;
		\item the $P_i$'s partition the universe and are dense (with respect to $\leq$);
		\item the function $f_i$ is the identity outside of $P_0$; its restriction to $P_0$ is a bijection between $P_0$ and $P_i$;
		\item for all $x\in P_0$, we have $x<f_1(x)<f_2(x)<\cdots <f_{n-1}(x)$;
		\item given any open intervals $I_0<I_1<\cdots <I_n$, there is $x\in I_0$ such that $f_i(x)\in I_i$ for each $1\leq i<n$.
	\end{itemize}
Then the theory $T_n$ is complete, $\omega$-categorical and has elimination of quantifiers.	
\end{prop}
\begin{proof}
	This can be shown by a straightforward back-and-forth argument. Alternatively, one can see that $T_n$ is the Fra\"iss\'e limit of the class of finite $L_n$ structures satisfying:
	
	\begin{itemize}
		\item $\leq$ defines a linear order;
		\item the $P_i$'s partition the universe;
		\item the function $f_i$ is the identity outside of $P_0$ and for $x\in P_0$, we have $P_i(f_i(x))$ and $x<f_1(x)<f_2(x)<\cdots <f_{n-1}(x)$.
	\end{itemize}
	It follows that $T_n$ has elimination of quantifiers. Hence it is complete and $\omega$-categorical (because the structure generated by a set of size $m$ has size at most $nm$).
\end{proof}

Let now $(V;\leq_0,\ldots,\leq_{n-1})$ be a structure equipped with $n$ distinct linear orders. Assume that each order $V_i:=(V,\leq_i)$ has topological rank 1 and that any two $V_i$, $V_j$ are intertwined. Further assume that the structure $V$ is transitive (that is, there is a unique 1-type over $\emptyset$). For each $i<n$, there is by Lemma \ref{lem:unique intertwining} a unique increasing 0-definable map $f_i\colon V_i \to \overline {V_0}$. Inside $V$, we interpret an $L_n$-structure $V_\ast$ as follows: the universe of $V_\ast$ is the union of $n$ disjoint copies of $V$, which we think of as representing the orders $V_0$ to $V_{n-1}$. The unary predicate $P_i$ names the $i$-th copy of $V$, which we identify with the image $f_i(V_i)$ inside $\overline{V_0}$. The order $\leq$ on $V_\ast$ is then given by the order on $\overline{V_0}$ using those identifications. Finally, the function $f_i$ sends a point $x\in P_0(V_\ast)$ to the corresponding point in $P_i(V_\ast)$: remember, that both are just copies of $V$, so $f_i$ is just the canonical identification of one copy of $V$ with the other. Define also $f_0$ as being the identity function on $V_\ast$.

Since we assumed that $V$ has a unique 1-type over $\emptyset$, then for some permutation $\sigma$ of $\{0,\ldots,n-1\}$, we have that for all $x\in P_0(V_\ast)$, \[f_{\sigma(0)}(x)< f_{\sigma(1)}(x) < \cdots < f_{\sigma(n)}(x).\]

If $\sigma$ is the identity, then $V_\ast$ is a model of $T_n$ as defined above. Otherwise, we obtain a model of $T_n$ by applying the same construction to the structure $(V;\leq_{\sigma^{-1}(0)},\ldots,\leq_{\sigma^{-1}(n-1)})$. Note that there is a unique $\sigma$ with this property.

Conversely, given a (countable) model $(V_\ast;\leq,P_0,\ldots,P_{n-1},f_1,\ldots,f_{n-1})$ of $T_n$, we can construct a structure $(V^{(n)};\leq_0,\ldots,\leq_{n-1})$ by taking as universe $V^{(n)}=P_0(V_\ast)$, interpreting $\leq_0$ as $\leq$ and $\leq_i$, $i>0$ by: \[x\leq_i y \iff f_i(x)\leq f_i(y).\]

Note that by $\omega$-categoricity of $T_n$, the structure $V^{(n)}$ is uniquely defined up to isomorphism. For each $i\leq n$, let $V^{(n)}_i$ be the definable linear order $(V^{(n)};\leq_i)$. It might seem that by going to $V^{(n)}$, we have lost the intertwining between the orders, but in fact this is not the case. Indeed, the orders $V^{(n)}_0$ and $V^{(n)}_i$, $i<n$ are intertwined in $V^{(n)}$: let $x\in V^{(n)}$ and consider the set \[g_i(x) := \{y\in V^{(n)}: (\forall z<_i y) z<_0 x\}.\] Then $g_i(x)$ is a cut of $V^{(n)}_i$ and we leave it to the reader to check that $g_i$ does define an intertwining from $V^{(n)}_0$ to $V^{(n)}_i$.

If we apply the first construction above to $V^{(n)}$, then we recover the $V_\ast$ we started with. The following proposition now follows from this discussion.

\begin{prop}\label{prop:classification of intertwined orders}
	Let $(V;\leq_0,\ldots,\leq_{n-1})$ be a transitive (countable) structure equip\-ped with $n$ distinct linear orders. Assume that each order $V_i:=(V,\leq_i)$ has topological rank 1, any two $V_i$, $V_j$ are intertwined. Then for some unique permutation $\sigma$ of $\{0,\ldots,n-1\}$, $(V;\leq_{\sigma(0)},\ldots,\leq_{\sigma(n-1)})$ is isomorphic to the structure $(V^{(n)};\leq_0,\ldots,\leq_{n-1})$ defined above. In particular, there are exactly $n!$ such structures up to isomorphism. 
\end{prop}

%

\subsection{Independent orders}

\begin{defi}
Let $V$ and $W$ be two orders, definable over some $A$. We say that $V$ and $W$ are \emph{independent} if there does not exist:

$\bullet$ a set of parameters $B\supseteq A$,

$\bullet$ $B$-definable infinite subsets $X\subseteq V$ and $Y\subseteq W$, both weakly transitive over $B$, which we equip with the induced orders from $V$ and $W$ respectively,

$\bullet$ a $B$-definable intertwining from $X$ to either $Y$ or the reverse of $Y$.
\end{defi}

Note that independence is a symmetric relation.

\begin{lemme}\label{lem:no self-intertwining} Let $(V,\leq)$ be definable and minimal over some $A$. Let $B\supseteq A$ and $I,J\subseteq V$ be two infinite $B$-definable disjoint convex subsets, weakly transitive over $B$, then $(I,\leq)$ and $(J,\leq)$ are independent.
\end{lemme}

\begin{proof}
Without loss of generality, $B$ is finite. Assume first that there is an intertwining map $f$ from $I$ to $J$, definable over $B$. Then $f$ extends to an increasing bijection from $\overline I \to \overline {J}$, which we still denote by $f$ (see the paragraph after Lemma \ref{lem:intertwining}). Assume for definiteness that $I<J$. Let $c_1,c_2$ be the infimum and supremum of $I$ respectively, seen as elements of $\overline V$. Define similarly $d_1,d_2$ for $J$. Hence we have $c_1<c_2<d_1<d_2$. By Proposition \ref{lem2} (and Remark \ref{rem:lem2 in completion}), we can find $(c'_1,c'_2,d'_1,d'_2)\equiv_A (c_1,c_2,d_1,d_2)$ such that \[c_1<c'_1<c'_2<c_2<d'_1<d_1<d_2<d'_2.\] 
Let $\sigma\in \aut(M)_A$ send $(c_1,c_2,d_1,d_2)$ to $(c'_1,c'_2,d'_1,d'_2)$. Let $I',J'$ be the images of $I, J$ respectively under $\sigma$ and set $g=\sigma\circ  f \circ \sigma^{-1}$, so that $g$ is an increasing map from $\overline{I'}$ to $\overline {J'}$.

Let $a\in I\setminus I'$, say $c_1<a<c'_1$. Then $f$ sends $a$ to a point in $J\subset J'$. So $g^{-1}$ is defined on $f(a)$ and sends it to a point in $\overline{I'}$, hence $a<g^{-1}(f(a))$. Applying $f$, we obtain $f(a)<f(g^{-1}(f(a))$, thus $g^{-1}(f(a))<g^{-1}(f(g^{-1}(f(a))))$. Iterating, we find infinitely many elements in $\dcl^{eq}(aB\sigma(B))\cap \overline V$:
\[a<g^{-1}\circ f(a)<g^{-1}\circ f \circ g^{-1}\circ f(a)<g^{-1}\circ f\circ g^{-1}\circ f\circ g^{-1}\circ f(a)<\cdots.\] This contradicts Lemma \ref{lem:finite dcl}.  The same argument shows that $I$ is not intertwined over $B$ with the reverse of $J$.

Now take $B'\supseteq B$ and $X\subseteq I$, $Y\subseteq J$ two infinite subsets, definable and weakly transitive over $B'$. Since $V$ is minimal, the closure of $X$ is a finite union of convex subsets. By weak transitivity, it is just one convex subset $I'$. Similarly the closure of $Y$ is a convex subset $J'$. An intertwining between $X$ and $Y$ induces naturally an intertwining between $I'$ and $J'$. Using the previous paragraph we see that there is no such intertwining. We conclude that $I$ and $J$ are independent.
\end{proof}

\begin{cor}\label{cor:no self-intertwining} Let $(V,\leq)$ be definable and minimal over some $A$. Let $B\supseteq A$ and $I,J\subseteq V$ be two infinite $B$-definable convex subsets. Assume that $I$ and $J$ are intertwined, then $I=J$ and the intertwining map is the identity.
\end{cor}
\begin{proof}
Let $h$ be the intertwining map from $I$ to $\overline J$. If $h$ is not the identity, then we can find some convex subset $I' \subseteq I$ such that the convex hull of $h(I')$ is disjoint from $I$ (since $h$ is increasing). Taking $I'$ smaller if necessary, we can assume that it is weakly transtive over parameters $C$ defining it. Then $h(I')$ is also weakly transitive over $C$, and $I'$ and $h(I')$ contradict the previous lemma. 
\end{proof}

\begin{cor}\label{cor:no inverse intertwining}
Let $(V,\leq)$ be definable and minimal over some $A$. Then $(V,\leq)$ is not intertwined with its reverse $(V,\geq)$.
\end{cor}
\begin{proof}
If there is an $A$-definable decreasing map $f\colon V\to \overline V$, then we can find an open interval $I\subseteq V$ such that the convex hull of the image $f(I)$ is disjoint from $I$. Let $J$ be the intersection of the convex hull of $f(I)$ with $V$. Then $I$ and $J$ contradict the previous lemma.
\end{proof}

\begin{lemme}\label{lem_ind} Let $V_0, V_1$ be two linear orders definable and minimal over some $A$. Assume that they are not independent. Then there is either an $A$-definable intertwining from $V_0$ to $V_1$ or an $A$-definable intertwining from $V_0$ to the reverse of $V_1$.
\end{lemme}
\begin{proof}
Without loss, assume that $A$ is finite and let $B\supseteq A$ be also finite. Assume that we have some $B$-definable $X_0\subseteq V_0$ and $X_1\subseteq V_1$ both weakly transitive over $B$ and a $B$-definable increasing map $f\colon X_0 \to \overline{X_1}$ (if there is a decreasing map from $X_0$ to $\overline{X_1}$, replace $V_1$ by its reverse). Restricting $X_0$, we may assume that it is transitive over $B$. Let $a\in X_0$. By topological rank 1, both $X_0$ and $X_1$ are dense in their convex hulls and  $f$ extends to an increasing map $\overline{X_0}\to \overline{X_1}$. Assume that $f(a)\notin \dcl(Aa)$. Then, there is $\sigma \in \aut(M)_{Aa}$ such that $\sigma(f(a))\neq f(a)$. Let $f' = \sigma(f)$ so that \[f'(a) = \sigma(f)(\sigma(a)) =  \sigma(f(a))\neq f(a).\]

Then there is an open interval $I$ of $V_0$ containing $a$ such that $f'$ is defined on $I$ and induces an increasing map $f'|_I\colon I\to \overline{V_1}$ (indeed, we can take $I= \sigma(\overline{X_0})$). Reducing $I$, we can assume that $f(I)$ and $f'(I)$ are disjoint. But then $f'\circ f^{-1}$ gives an intertwining map from $f(I)$ to $f'(I)$, which contradicts Lemma \ref{lem:no self-intertwining}.
	
	It follows that $f(a)\in \dcl(Aa)$. Let $g$ be the $A$-definable map sending $a$ to $f(a)$ and for simplicity assume $V_0$ is transitive over $A$ (otherwise, replace it by the locus of $\tp(a/A)$). As $X_0$ is transitive over $B$, $g$ coincides with $f$ on $X_0$ and therefore is increasing on it. Let $\tau \in \aut(M)_A$ and let $X_0'=\tau(X_0)$. Then $g$ is also increasing on $X_0'$. Assume that the convex hulls of $X_0$ and $X_0'$ in $V_0$ have an open interval $Z$ in their intersection. We can construct two increasing maps from $Z$ to $\overline V_1$: one induced by $g|_{X_0}$ and one induced by $g|_{X'_0}$. By Lemma \ref{lem:unique intertwining}, those two maps coincide. As $V_0$ is transitive over $A$, the sets $\tau(X_0)$, for $\tau$ ranging in $\aut(M)_A$ cover $V_0$. In particular, they cover the convex hull of $X_0$ in $V_0$. It follows that $g$ is increasing on the convex hull of $X_0$. Therefore as $V_0$ is transitive over $A$, $g$ is locally increasing on $V_0$: for each $a\in V_0$, there is an open convex subset of $V_0$ containing $a$ on which $g$ is increasing. Let $C_a$ denote the maximal such set. The the sets $C_a$ form an $A$-definable partition of $V_0$ into infinite convex sets. As $V_0$ has topological rank 1, $C_a=V_0$ for all $a$ and $g$ is increasing on $V_0$. It follows that $g$ intertwines $V_0$ and $V_1$.
\end{proof}

\begin{lemme}\label{lem4} 	Working over some $A$, let $(V_0,\leq_0)$, $(V_1,\leq_1)$ be two minimal independent definable orders. Let $f_0\colon  V_0\to \overline {V_0}$ and $f_1\colon V_0 \to \overline {V_1}$ be two $A$-definable functions. Then the set \[\{(f_0(x),f_1(x)):x\in V_0\}\] is dense in $\overline{V_0}\times \overline{V_1}$.
\end{lemme}
\begin{proof} First, the images of $f_0$ and $f_1$ are definable subsets of $\overline{V_0}$ and $\overline{V_1}$ respectively. By Lemma \ref{lem:slice of completion}, we can replace $V_0$ by $V_0\cup f_0(V_0)$ and $V_1$ by $V_1\cup f_1(V_0)$ and assume that $f_0$ and $f_1$ take values in $V_0$ and $V_1$ respectively.

Let $V\subseteq V_0$ be $A$-definable and transitive over $A$. Then by minimality, $V$ is dense in $V_0$ and it is enough to prove that $\{(f_0(x),f_1(x)):x\in V\}$ is dense in $\overline{V_0}\times \overline{V_1}$. Next, notice that since $V_0$ and $V_1$ are minimal over $A$ and $f_0,f_1,V$ are $A$-definable, $f_0(V)$ is dense in $\overline{V_0}$ and $f_1(V)$ is dense in $\overline{V_1}$. Fix $a\in V$ and consider the set \[X_a = \{f_0(x):x\in V, f_1(x)<_1f_1(a)\}.\] This set is non-empty by the previous sentence. Let also $Y_a$ be the closure of $X_a$. Then by Lemma \ref{lem1}, $Y_a$ is a finite union of convex sets.

The infima and suprema of those convex sets are either $\pm \infty$ or elements of $\overline{V_0}$. Let $W\subseteq \overline{V_0}$ be an $A$-definable subset containing $V_0$ along with all those elements. By Lemma \ref{lem:slice of completion}, $W$ is minimal over $A$. Assume that $Y_a$ contains a bounded interval \[c\leq_0 x \leq_0 d, \quad c,d\in W,\] and this interval is maximal in $Y_a$. By Proposition \ref{lem2} applied to $W$, there is an automorphism $\sigma$ such that $c<_0 \sigma(c)<_0 d<_0 \sigma(d)$. But then, we have neither $Y_{\sigma(a)}\subseteq Y_a$, nor $Y_a\subseteq Y_{\sigma(a)}$ and this is impossible by the definition of $X_a$. We can do the same thing if $Y_a$ contains two disjoint unbounded intervals. We conclude that $Y_a$ is either an initial segment, an end segment, or the whole of $\overline{V_0}$.

Assume that $Y_a$ is an initial segment and define $g(a)\in W$ to be its supremum. Then as $V$ is a complete type, $Y_{a'}$ is an initial segment for each $a'\in V$. Let $h\colon f_1(V)\to V_0$ send a point $b=f_1(a')$ to $g(a')$. This is well defined as $X_{a'}$ and hence $g(a')$ depends only on $f_1(a')$. Note that $h$ is non-decreasing and therefore intertwines $f_1(V)$ and $V_0$. This contradicts independence. Similarly, if $Y_a$ is an end segment, we obtain an intertwining from $f_1(V)$ to the reverse of $V_0$. We therefore conclude that $Y_a$ is equal to $\overline{V_0}$. We also have symmetrically that $\{f_0(x): x\in V, f_1(x)>_1 f_1(a)\}$ is dense in $\overline{V_0}$ for all $a\in V$.

Assume now towards a contradiction that for some bounded open interval $I\subset V_0$, the set \[H(I):=\{f_1(x): x\in V_0, f_0(x)\in I\}\] is not dense in $\overline{V_1}$ (where we have identified $I$ with its convex closure in $\overline{V_0}$). Let $J\subset V_0$ be any bounded interval. By Proposition \ref{lem2}, there is an automorphism $\sigma$ over $A$ such that $\sigma(J)\subseteq I$. Then $H(\sigma(J))\subseteq H(I)$ is not dense in $\overline{V_1}$. Therefore, also $H(J)$ is not dense in $\overline{V_1}$.

By what we know so far, $H(I)$ is cofinal and coinitial in $\overline{V_1}$ (since for any $d\in f_1(V)$, the sets $\{f_0(x): x\in V, f_1(x)<_1 d\}$ and $\{f_0(x):x\in V, f_1(x)>_1 d\}$ are dense in $\overline{V_0}$ and $f_1(V)$ is dense in $V_1$). Let $C(I) = V_1 \setminus \overline{H(I)}$. Then $C(I)$ is a non-empty finite union of bounded open intervals. Let $\tilde C(I)$ be its convex hull. If $I\subseteq J$, then $H(I)\subseteq H(J)$, so $C(I)\supseteq C(J)$ and $\tilde C(I)\supseteq \tilde C(J)$. As any two intervals are contained in a third one, any two intervals of the form $\tilde C(J)$ intersect, where $J$ is any open interval of $V_0$. Let $a\in V_1$ to the left of $\tilde C(I)$ and $b\in V_1$ to the right of it of same type as $a$. Then there is an automorphism $\sigma$ over $A$ sending $a$ to $b$. Then $\sigma(\tilde C(I))=\tilde C(\sigma(I))$ is disjoint from $\tilde C(I)$. This is a contradiction.
\end{proof}

%
%

Having described in Corollary \ref{lem2_cor} the closed definable subsets of minimal orders, and hence of products of intertwined orders, we now complete the picture with the case of pairwise independent orders.

\begin{prop} \label{lem6} Working over some set $A$, let $V_0,\ldots,V_{n-1}$ be pairwise independent minimal definable orders. Then any $A$-definable closed set $X\subseteq V_0^{k_0}\times \cdots \times V_{n-1}^{k_{n-1}}$ is a finite union of products of the form $D_0\times \ldots \times D_{n-1}$, where each $D_i$ is an $A$-definable closed subset of $V_i^{k_i}$.
\end{prop}

\begin{proof}
We assume for simplicity that $A=\emptyset$.  Let $p\in S(\emptyset)$ be a type on $\overline{V_0}^{k_0}\times \cdots \times \overline{V_{n-1}}^{k_{n-1}}$. Let $D$ be the closure in $\overline{V_0}^{k_0}\times \cdots \times \overline{V_{n-1}}^{k_{n-1}}$ of its set of realizations. Say that $p$ has property $\boxtimes$ if there are closed $0$-definable sets $D_i\subseteq V_i^{k_i}$ such that \[ D \cap (V_0^{k_0}\times \cdots \times V_{n-1}^{k_{n-1}}) = D_0 \times \cdots \times D_{n-1}.\]
We prove the following two statements by induction on $n$:

\begin{description}
\item[($A_n$)] Let $V_0,\ldots, V_{n-1}$ be pairwise independent minimal orders. If $f_i\colon V_0\to \overline V_i$, $i<n$, are $0$-definable functions, then $\{(f_0(x),f_1(x),\ldots,f_{n-1}(x)):x\in V_0\}$ is dense in $\overline{V_0}\times \cdots \times \overline{V_{n-1}}$.

\item[($B_n$)]  Let $V_0,\ldots,V_{n-1}$ be pairwise independent minimal orders. Then any  type $p\in S(\emptyset)$  on $\overline{V_0}^{k_0}\times \cdots \times \overline{V_{n-1}}^{k_{n-1}}$ has property $\boxtimes$.
\end{description}

The statement of the theorem then follows from $(B_n)$ since by $\omega$-categoricity, any definable set is a finite union of types.

\smallskip
Property $(A_1)$ follows from minimality and $(B_1)$ holds trivially. We will show that $(A_n)$ and $(B_n)$ together imply $(A_{n+1})$ and then that $(A_{n+1})$ implies $(B_{n+1})$.

\smallskip

{$(A_n)+(B_n) \Rightarrow (A_{n+1})$}: The property $(A_2)$ is Lemma \ref{lem4}, so we can assume $n>1$. We follow closely the proof of Lemma \ref{lem4}. Fix $a\in V_0$ and define
 \[X_a=\{(f_0(x),f_1(x),\ldots,f_{n-1}(x)) : x\in V_0, f_n(x)<f_n(a)\}\subseteq V_0\times \cdots \times V_{n-1}.\]
 For each $i<n$, let $Y_i\subseteq V_i$ be a complete type over $a$. Note that $\overline Y_i$ is convex in $\overline V_i$ (it is a finite union of convex sets by minimality and then is convex since it defines a complete type over $a$). Set \[\hat Y= \prod_{i<n} \overline{Y_i}\subseteq \prod_{i<n} \overline{V_i}.\] Working over the parameter $a$, the $Y_i$'s are pairwise independent minimal orders. The property $(A_{n})$ then implies that $X_a \cap \hat Y$ is either dense in $\hat Y$ or empty. It now follows that the closure $\overline {X_a}$ of $X_a$ in $\prod_{i<n} \overline{V_i}$ is a union of finitely many rectangles of the form $\prod_{i<n} I_i$, where each $I_i\subseteq \overline{V_i}$ is a convex set. The tuple of endpoints (or endcuts) of those convex sets is an element of some $\overline {V_1}^{k_1}\times \cdots \times \overline{V_{n-1}}^{k_{n-1}}$. Let $p$ be the type of that tuple over $\emptyset$. Applying $(B_n)$, we see that $p$ has property $\boxtimes$. In addition, it follows from the definition of $X_a$ that for any $a'$ having the same type as $a$, one of $\overline{X_a}$ and $\overline{X_{a'}}$ is included in the other. Since property $\boxtimes$ allows us to move the endpoints of the convex sets defining $X_a$ freely, this is only possible if $\overline{X_a}$ is either the full product $\prod_{i<n} \overline{V_i}$, or is a rectangle, unbounded on all but at most one coordinate. However, by $(B_2)$, we know that $\overline {X_a}$ must have full projection on each coordinate. Hence the only possibility is that $\overline {X_a}=\prod_{i<n} \overline{V_i}$.

We end as in Lemma \ref{lem4}. Density of $X_a$ in the product implies that for any product $\hat I=\prod_{i<n} I_i$ of open intervals, the set \[s(\hat I):=\{f_n(x) : (f_0(x),f_1(x),\ldots,f_{n-1}(x))\in \hat I\}\] is coinitial in $V_n$. By applying the same argument to the reverse order, we get that it is also cofinal. Furthermore, by minimality, the closure of $s(\hat I)$ is a finite union of convex sets. Hence, given any $\hat I$, there is a unique minimal bounded convex set $c(\hat I)\subseteq V_n$ such that $s(\hat I)$ is dense in  $V_n \setminus c(\hat I)$. If $\hat I\subseteq \hat I'$, then $c(\hat I)\supseteq c(\hat I')$. As $V_n$ is weakly transitive, the intersection of all $c(\hat I)$ is empty. Since the family of $\hat I$'s is upward-directed under inclusion, $c(\hat I_*)$ must be empty for some $\hat I_*$. But then by $(B_{n})$, for any $\hat I'$, one can find $\hat I'_* \subseteq  I_*$ which is a conjugate of $\hat I'$. Hence $c(\hat I')$ is also empty and $s(\hat I')$ is dense in $V_n$. Since this holds for any $\hat I'$, $(A_{n+1})$ follows.

\smallskip

$(A_{n+1}) \Rightarrow (B_{n+1})$: As in the proof of Proposition \ref{lem2}, to show that all types have property $\boxtimes$, it is enough to find, for all $k<\omega$, one type in $\overline V_0^k\times \cdots \times \overline V_n^k$ having property $\boxtimes$ and for which no two coordinates are equal. To this end, take $b_{0}\in V_0$. For each $i\leq n$, let $m_i(b_{0})$ denote the largest element of $\overline{V_i}$ definable from $b_{0}$. Set $a_{0,i}=m_i(b_{0})$. Then by $(A_{n+1})$ applied to the functions $m_i$, we see that $p_1:=\tp(a_{0,i}:i\leq n)$ has property $\boxtimes$: its set of realizations is dense in $\overline{V_0}\times \cdots \times \overline{V_n}$.

Assume that $b_{l}$, $a_{l,i}$ have been constructed for $l<k$, $i\leq n$, with $a_{l,i}=m_i(b_{\leq l})$. For $i\leq n$, let $X_i$ be a complete type over $b_{<k}$ of elements in $V_i$, greater than $a_{k-1,i}$. So $X_i$ is dense in $\{x\in V_i:x>a_{k-1,i}\}$. Work over $b_{<k}$ and consider the sets $X_0,\ldots,X_n$ equipped with the induced orders. They are pairwise independent. Pick any $b_{k}\in X_0$ and define $a_{k,i}=m_i(b_{\leq k})$, $i\leq n$. Then again by $(A_{n+1})$, the set of realizations of $\tp(a_{k,i}:i\leq n)$ is dense in $\overline{X_0}\times \cdots \times \overline{X_n}$. It follows inductively that the resulting type $p_k:=\tp(a_{l,i}:l\leq k,i\leq n)$ satisfies $\boxtimes$.
\end{proof}

\begin{definition}
	Let $(V,\leq)$ be a parameter-definable linear order. For $k<\omega$, a \emph{sector} of $V^k$ is a subset of $V^k$ defined by a formula $\phi(x_0,\ldots,x_{k-1})$ which is a finite boolean combination of relations of the form $x_i = x_j$ and $x_i \leq x_j$, $i,j<k$.
	
	 If $V_0,\ldots,V_{n-1}$ are pairwise independent parameter-definable orders, a \emph{sector} of $V_0^{k_0}\times \cdots \times V_{n-1}^{k_{n-1}}$ is a finite union of products $D_0\times \ldots \times D_{n-1}$, where each $D_i$ is a sector of $V_i^{k_i}$.
\end{definition}

\begin{corollary}\label{cor:lem6+}
	Working over some set $A$, let $V_0,\ldots,V_{n-1}$ be pairwise independent minimal definable orders. Let $X\subseteq V_0^{k_0}\times \cdots \times V_{n-1}^{k_{n-1}}$ be $A$-definable. Then the topological closure of $X$ is a sector of $V_0^{k_0}\times \cdots \times V_{n-1}^{k_{n-1}}$.
\end{corollary}
\begin{proof}
	This follows at once from Proposition \ref{lem6} applied to the topological closure of $X$ (which is also $A$-definable) along with Corollary \ref{lem2_cor}.
\end{proof}

We say that a betweenness relation has topological rank 1 if one (or equivalently both) of its associated linear orders has topological rank 1.
\begin{cor}\label{cor:n betweenness}
	Let $V$ be a definable transitive set and let $B_1,\ldots,B_n$ be distinct 0-definable betweenness relations on $V$ of topological rank 1. Then for any subset $I\subseteq n$, we can find $a_I,b_I,c_I\in V$ such that $B_i(a_I,b_I,c_I)$ holds if and only if $i\in I$.
\end{cor}

We now extend Proposition \ref{prop:classification of intertwined orders} to a structure equipped with any $n$ minimal linear orders.

\begin{prop}\label{prop:iso types of n orders}
	Let $(M;\leq_1,\ldots,\leq_n)$ be countable, $\omega$-categorical, transitive over $\emptyset$. Assume that each $M_i:=(M,\leq_i)$ is a linear order of topological rank 1 and that no two of them are equal or reverse of each other. Then for each $i\neq j \leq n$, exactly one of the following holds:
	\begin{itemize}
		\item $\leq_i$ and $\leq_j$ are independent;
		\item $\leq_i$ is intertwined with $\leq_j$ and if $f_{ij}\colon M_i \to \overline {M_j}$ is the unique 0-definable increasing map, we have $f_{ij}(x)<_j x$ for all $x$;
		\item $\leq_i$ is intertwined with $\leq_j$ and we have $f_{ij}(x)>_j x$ for all $x$;
		\item $\leq_i$ is intertwined with the reverse of $\leq_j$ and if $f_{ij}\colon M_i \to \overline {M_j}$ is the unique 0-definable decreasing map, we have $f_{ij}(x)<_j x$ for all $x$;
		\item $\leq_i$ is intertwined with the reverse of $\leq_j$ and we have $f_{ij}(x)>_j x$ for all $x$.
	\end{itemize}
	Furthermore, the data of which of those cases holds for each pair $i\neq j$ completely determines the isomorphism type of $M$.
\end{prop}
\begin{proof}
	The argument is similar to that of Proposition \ref{prop:classification of intertwined orders}, which we present a little bit differently. First note that by Corollary \ref{cor:no inverse intertwining}, by replacing some orders $\leq_i$ with their reverses, we can assume that the last two cases never occur. Let then $E$ be the equivalence relation on $\{1,\ldots,n\}$ which holds for $i,j$ if $M_i$ and $M_j$ are intertwined. Let $s_1,\ldots,s_k$ be representatives of the $E$-classes and for each $i\leq n$, let $t(i)$ be such that $i \rel E s_{t(i)}$. Define also $\iota_i \colon  M_i \to \overline {M_{s_{t(i)}}}$ to be the unique increasing 0-definable map intertwining $M_i$ and $M_{s_{t(i)}}$.
	
	For $l\leq k$, define $V_l\subseteq \overline{M_{s_l}}$ as the union \[V_l = \bigcup_{i\rel{E} s_l} \iota_i (M_i),\] and let $\preceq_l$ be its canonical linear order. Then $(V_l,\preceq_l)$ is a minimal, 0-definable order. Define \[\Gamma = \{(\iota_1(x),\ldots,\iota_n(x)):x\in M\} \subseteq \prod_{i\leq n} V_{t(i)}.\]	
	Now by the previous proposition, $\Gamma$ is dense in a product $D_1\times \cdots \times D_k$ of closed subsets of the $V_i$'s. By Corollary \ref{lem2_cor}, $D_k$ is dense in a set defined by a boolean combination of inequalities on variables $x_i \leq x_j$. Those inequalities are determined by inequalities $\iota_i(x) \lessgtr_{s_{t(i)}} x$ that are true in $M$ and are part of the data that we are given. We conclude by a direct back-and-forth argument as in Proposition \ref{prop:unique n intertwined orders}.
\end{proof}

We will need the following corollary later on.

\begin{cor}\label{cor:n independent orders}
Let $(M;\leq_1,\ldots ,\leq_n,\dots)$ be as in the previous proposition with possibly additional structure. Then there is a finite set of parameters $\bar d$ and some $\bar d$-definable subset $X\subseteq M$ such that $X$ is transitive over $\bar d$ and the orders $\leq_1,\ldots,\leq_n$ are pairwise independent when restricted to $X$.
\end{cor}
\begin{proof}
We construct $X$ as a subset of an intersection of intervals for the various orders. We will use the notation $M_i = (M,\leq_i)$ as in the previous proposition.

For each $i\leq n$, we pick $a_i <_i b_i$ in $M$. We can make this choice in such a way that the intervals $a_i <_i x <_i b_i$ and $a_j <_j x <_j b_j$ are independent when equipped with $<_i$ and $<_j$ respectively: if $M_i$ and $M_j$ are independent, then any choice of points will do and if they are intertwined, pick two intervals so that their convex closures in the common completion $\overline{M_i } = \overline{M_j}$ are disjoint (similarly if they are intertwined up to reversal). Having done this, define let $\bar d = (a_i,b_i)_{i\leq n}$ and let $X_0$ be the set defined by the conjunction of the formulas $a_i <_i x <_i b_i$ for $i\leq n$. The set $X_0$ need not be transitive over $\bar d$, so let $X\subseteq X_0$ be any $\bar d$-definable infinite subset of $X_0$ which is transitive over $\bar d$. Then then $n$ orders are pairwise independent on $X$ as required.
\end{proof}

\subsection{Weakly minimal orders}\label{sec:weakly minimal}

We have defined intertwining only for weakly transitive orders. This was required to ensure that being intertwined is a symmetric relation. This will be too restrictive later on, especially because the property of being weakly transitive is not invariant under adding parameters to the base. We therefore need a more general notion that allows us to talk about orders being intertwined even if they are not weakly transitive. To keep the nice properties of this relation (in particular, symmetry and uniqueness of the intertwining), we instead assume that the orders are \emph{weakly minimal}, as defined below. To motivate this definition, note that if $(V,\leq)$ is $A$-definable and has topological rank 1, then $\dcl^{eq}(A) \cap \overline V$ is finite of size $n$ say. Enumerate its elements as $s_1 < \cdots < s_n$. We use the interval notation $(a,b) := \{ x\in V : a<x<b\}$ for $a,b\in \overline V$. Then the $n+1$ convex subsets $(-\infty,s_1)$, $(s_i,s_{i+1})$ and $(s_n,+\infty)$ are each $A$-definable and minimal over $A$ (and furthermore those are the \emph{only} $A$-definable infinite convex subsets of $V$ minimal over $A$ since for any $A$-definable convex subset $W$ of $V$ the infimum and supremum of $W$ are definable over $A$). Any two of those minimal convex subsets are either independent or intertwined over $A$.  The following definition simply says that the latter never happens.

\begin{defi}
Let $A$ be a set of parameters. An $A$-definable order $(V,\leq)$ is \emph{weakly minimal over $A$} if:

$\bullet$ it is densely ordered with neither smallest no largest element;

$\bullet$ it has topological rank 1;

$\bullet$ any two distinct $A$-definable minimal convex subsets of it are independent.
\end{defi}

\begin{lemme}
Let $(V,\leq)$ be an $A$-definable linear order and $B\supseteq A$ a larger set of parameters. Then $V$ is weakly minimal over $A$ if and only if it is weakly minimal over $B$.
\end{lemme}
\begin{proof}
Assume that $V$ is weakly minimal over $A$. Let $W_1,W_2 \subseteq V$ be two $B$-definable distinct convex subsets of $V$ that are minimal over $B$. Let $W^*_1, W^*_2$ the $A$-definable convex subsets of $V$ that are minimal over $A$ and contain $W_1$ and $W_2$ respectively. If $W^*_1 = W^*_2$, then $W_1, W_2$ are independent by Lemma \ref{lem:no self-intertwining}. If $ W^*_1 \neq  W^*_2$, then $W^*_1$ and $W^*_2$ are independent and hence so are $W_1$ and $W_2$ by definition of independence.

Conversely, assume that $V$ is not weakly minimal over $A$. Let $W^*_1,  W^*_2$ be two distinct $A$-definable convex subsets of $V$ that are minimal over $A$ and intertwined. Let $f:\overline {W^*_1} \to \overline {W^*_2}$ be the $A$-definable increasing bijection. Let $W_1$ be an infinite convex subset of $W^*_1$ which is definable and minimal over $B$. Let $W_2$ be equal to $f(\overline {W^*_1})\cap W^*_2$. Then $W_2$ is definable and minimal over $B$ and $W_1$ and $W_2$ are intertwined. Hence $V$ is not weakly minimal over $B$.
\end{proof}

It follows from the lemma that we can drop ``over $A$" in the definition of weakly minimal: given a parameter-definable order $(V,\leq)$, we can say whether it is weakly minimal or not without mentioning a set of parameters over which it is defined since the definition does not depend on the set of parameters actually chosen. (If $V$ is definable over both $A$ and $A'$, let $B=A\cup A'$ and apply the previous lemma to the pair $(A,B)$ and to the pair $(A',B)$.)

\begin{defi}
Let $V, W$ be two linear orders, definable over some $A$ and weakly minimal. We say that $V$ and $W$ are \emph{intertwined over $A$} if there is an $A$-definable increasing bijection $f:\overline V \to \overline W$.
\end{defi}

\begin{prop}
If $V$ and $W$ are weakly minimal $A$-definable orders which are intertwined over $A$, then there is a unique intertwining between $V$ and $W$.
\end{prop}
\begin{proof}
Assume that $V$ and $W$ are intertwined and let $f:\overline V \to \overline W$ be the $A$-definable increasing bijection. Enumerate $\dcl^{eq}(A)\cap \overline V$ as $s_1 < \cdots < s_n$ and enumerate $\dcl^{eq}(A) \cap \overline W$ as $t_1 < \cdots < t_m$. Formally set $s_0$ and $t_0$ to be equal to $-\infty$ and $s_{n+1}$, $t_{m+1}$ to be equal to $+\infty$. Let $i<n+1$. The subset $(s_i, s_{i+1})$ of $V$ is definable and minimal over $A$. Since $f$ is definable over $A$, both $f(s_i)$ and $f(s_{i+1})$ are in $\dcl^{eq}(A)\cap \overline W$ and the subset $(f(s_i),f(s_{i+1})$ is an $A$-definable convex subset of $W$. (Here, we have increased $f$ formally to map $\pm \infty$ to $\pm \infty$.) Furthermore, $(f(s_i),f(s_{i+1})$ is minimal over $A$: if not, there would be some $j<m+1$ such that $f(s_i) < t_j < f(s_{i+1})$, but then $f^{-1}(t_j)$ would be an element of $\dcl^{eq}(A)\cap \overline V$ lying strictly between $s_i$ and $s_{i+1}$. Therefore for some $j$, $f(s_i)=t_j$ and $f(s_{i+1})=t_{j+1}$. Since $f$ is increasing, this implies that $n=m$ and $f(s_i)=f(t_i)$ for all $i<n+1$. Now, for each $i<n+1$, there is by Lemma \ref{lem:unique intertwining}, a unique intertwining between $(s_i,s_{i+1})$ and $(t_i,t_{i+1})$. This shows that $f$ is completely determined and there is a unique intertwining between $V$ and $W$.
\end{proof}

It follows from the proposition that we can drop ``over $A$" in the definition of being intertwined: if $V$ and $W$ are definable both over $A$ and over $B$, then they are intertwined over $A$ if and only if they are intertwined over $B$ and the (unique) intertwining is definable over any set of parameters over which $V$ and $W$ are defined.

Note that if $V$ and $W$ are weakly minimal, but not intertwined, they are not necessarily independent: It could be that some proper convex subset of $V$ is intertwined with a convex subset of $W$, but that intertwining does not extend to the whole of $V$. For instance let $(V_1,\leq_1), (V_2,\leq_2)$ be two definable independent minimal orders, definable over $\emptyset$ and assume for simplicity that $V_1, V_2$ are two disjoint subsets of the model. Set $V$ be the linear order $V_1 + V_2$, that is $V$ has underlying set $V_1 \cup V_2$ equipped with the linear ordering obtained by making the elements of $V_1$ smaller than those of $V_2$ and keeping the existing orders on $V_1$ and $V_2$ respectively. Then $V$ is weakly minimal and is neither independent nor intertwined with $V_1$.

Proposition \ref{lem6} and Corollary \ref{cor:lem6+} go through for weakly minimal orders instead of minimal orders, except that we have to extend our notion of sector to allow for definable convex subsets.

\begin{definition}
	let $A\subseteq M^{eq}$ and let $(V,\leq)$ be an $A$-definable linear order. For $k<\omega$, an \emph{$A$-sector} of $V^k$ is a subset of $V^k$ defined by a formula $\phi(x_0,\ldots,x_{k-1})$ which is a finite boolean combination of relations of the form:
	
	\begin{itemize}
		\item $x_i = x_j$, for $i,j<k$;
		\item $x_i \leq x_j$, for $i,j<k$;
		\item $x_i = a$, for $i<k$ and $a\in \dcl^{eq}(A)\cap V$;
		\item $x_i \leq a$, for $i<k$ and $a\in \dcl^{eq}(A)\cap \overline V$.
	\end{itemize}
\end{definition}

In other words, an $A$-sector is a subset of $V^k$ which is quantifier-free definable from the order along with unary predicates for $A$-definable cuts of $V$.

\begin{definition}
	let $A\subseteq M^{eq}$ and let $V_0,\ldots,V_{m-1}$ be pairwise independent $A$-definable linear orders. For $n<\omega$, an \emph{$A$-sector} of $V_0^{k_0}\times \cdots \times V_{m-1}^{k_{m-1}}$ is a finite union of sets of the form $D_0\times \cdots \times D_{m-1}$, where each $D_i$ is an $A$-sector of $V_{i}^{k_i}$.
\end{definition}

\begin{prop} \label{lem6+} Let $V_0,\ldots,V_{n-1}$ be pairwise independent weakly minimal $A$-definable orders. Let $X\subseteq V_0^{k_0}\times \cdots \times V_{n-1}^{k_{n-1}}$ be an $A$-definable set. Then the closure of $X$ is an $A$-sector of $V_0^{k_0}\times \cdots \times V_{n-1}^{k_{n-1}}$.
\end{prop}
\begin{proof}
Each $V_i$ decomposes over $A$ as a finite union of points in $\dcl(A)\cap V_i$ and minimal convex subsets with end-cuts in $\dcl^{eq}(A)\cap \overline {V_i}$. Note that all of those are $A$-sectors of $V_i$. Any two of those minimal convex subsets are independent. The product $V_0^{k_0}\times \cdots \times V_{n-1}^{k_{n-1}}$ then decomposes over $A$ as a finite union of products of powers of $A$-definable minimal orders and $A$-definable points, such that any two distinct minimal orders are independent. It is then enough to prove the proposition for one such product, but that is given by Corollary \ref{cor:lem6+}.
\end{proof}

Let $V$ be a weakly minimal linear order. We write $W \wmc V$ to say that $W\subseteq V$ is a convex subset of $V$ which is weakly minimal. Since $V$ is weakly minimal, this is equivalent to asking that $W$ is a convex subset of $V$ that has no first nor last element.

\section{Circular orders}\label{sec:circular orders}

Most of the results above generalize to circular orders, though some extra arguments are required.

Let $(V,C)$ be a circular order. We will abuse notation by writing say $a<b<c<d$ to mean that $a,b,c,d$ are pairwise distinct and $(a,b,c,d)$ lie in this order on $V$: that is $C(a,b,c)\wedge C(b,c,d)\wedge C(c,d,a)\wedge C(d,a,b)$. So $a<b$ only means that $a\neq b$ and $a<b<c$ is equivalent to $a\neq b\neq c \wedge C(a,b,c)$. Hopefully, this will not lead to confusion. For any $a<b$ on $V$, the set defined by $a<x<b$ is called an \emph{open interval} of $V$. Any interval has a canonical linear order on it coming from the circular order on $V$. The notations are consistent in the sense that if $I\subseteq V$ is an open interval, and $c,d,e\in I$, then we have $c<d<e$ in the sense of the circular order if and only if we have $c<d<e$ in the sense of the induced linear order on $I$.

For $a\in V$, we let $V_{a\to}=V\setminus \{a\}$, equipped with the linear order inherited from $C$. We say that $V$ has \emph{topological rank 1} if it does not admit a parameter-definable convex equivalence relation with infinitely many infinite classes. Then $V$ has topological rank 1 if and only if some/any $V_{a\to}$ has topological rank 1.

Let $V$ be circularly ordered.  A subset $I\subseteq V$ is \emph{convex} if for any $a\neq b\in I$, one of the two intervals $a<x<b$ and $b<x<a$ is included in $I$. A convex set $I$ is \emph{bounded} if its complement is infinite. Note that if $V$ is dense, then any open interval is bounded. A bounded convex set $I\subseteq V$ has a well defined linear order induced by the circular order on $V$. If $I$ and $J$ are two bounded convex subsets of $V$ with no last element (in their induced linear orders), we say that $I$ and $J$ \emph{define the same cut} in $V$ if one is an end segment of the other. 

We define the completion $\overline V$ of $V$ as the set of definable bounded convex subsets of $V$ quotiented by the equivalence relation of defining the same cut. As for linear orders, this is naturally a countable union of interpretable sets (or rather a direct limit). In fact, given $a\in V$, $\overline V$ can be canonically identified with $\overline {V_{a\to}} \cup \{a\}$: the element $a\in V$ is identified with the class of an open interval $b<x<a$ and any cut of $V_{a\to}$ is a bounded convex subset of $V$ and is identified with its class in $\overline V$. As in the case of linear orders, $\overline V$ is naturally equipped with a circular order, and there is a canonical embedding of $V$ in $\overline V$ which sends $V$ to a dense subset of $\overline V$.

We say that the $A$-definable circular order $V$ is \emph{weakly transitive} over $A$ if it is densely ordered and no element in $\overline V$ is algebraic over $A$.

\begin{lemma}
	If, over some $A$, $(V,C)$ is definable and weakly transitive of topological rank 1, then any $A$-definable subset of $V$ is dense in $V$.
\end{lemma}
\begin{proof}
	By topological rank 1, any closed $A$-definable subset of $V$ is a finite union of convex sets. The cuts defining these convex sets are algebraic over $A$, but there can be no such cut by weak transitivity.
\end{proof}

If $V$ and $W$ are two $A$-definable weakly transitive circular orders, we say that they are \emph{intertwined} over $A$ if there is an $A$-definable order-preserving injective map $f\colon V\to \overline W$. As for linear orders, this is an equivalence relation. It is no longer true that such a map has to be unique, however, we will see that there can be at most finitely many.

\begin{defi}
A \emph{self-intertwining} of a circular order $(V,C)$ is an intertwining map $f\colon V\to \overline V$ which is not the identity.
\end{defi}

Let $(V,C)$ be a $0$-definable circular order of topological rank 1 and fix some $a\in V$. Then we can write $V=F \cup V_1 \cup \cdots \cup V_n$, where $F=\dcl(a)\cap V$ and the $V_i$'s are convex subsets of $V$, definable and weakly transitive over $a$, with  $V_1<V_2<\cdots <V_n$. (In other words, the $V_i$'s are the infinite convex subsets of $V$ defined as the set of points between two consecutive elements of $\dcl^{eq}(a)\cap \overline V$. Note that $F$ only contains those elements of $\dcl^{eq}(a)\cap \overline V$ that actually lie in $V$.) The $V_i$'s are then minimal over $A$ and by Lemma \ref{lem:unique intertwining}, for any $i, j\leq n$, there is at most one intertwining map $f_{ij}\colon V_i \to \overline {V_j}$. If it exists, $f_{ij}$ has dense image.

Let now $f\colon V\to \overline V$ be a self-intertwining map defined over some set $A$. Let $W\subseteq V$ be a minimal $Aa$-definable infinite convex subset of $V$. Say that $W$ is a subset of $V_i$ and assume that $f$ sends $W$ to some $\overline{V_j}$. Composing by $f_{ij}^{-1}$, we get an intertwining between two subsets of $V_i$. By Corollary \ref{cor:no self-intertwining}, this must be the identity. It follows that $f$ coincides with $f_{ij}$ on $W$. By continuity of $f$, $f$ must coincide with $f_{ij}$ on the whole of $V_i$. So $f$ sends $V_i$ to some $\overline{V_j}$ via $f_{ij}$. Assume that for some $i$, $f$ sends $V_i$ to $\overline{V_{i+k}}$. Then as $f$ preserves the order, it must send $V_{i+1}$ to $\overline{V_{i+k+1}}$ (addition modulo $n$) and iteratively send any $V_{j}$ to $\overline{V_{j+k}}$. The number $k$ completely determines $f$, as does therefore the image of $a$. The possibilities for $k$ form a subgroup in $\mathbb Z/n \mathbb Z$. Hence the set of self-intertwinings along with the identity map, equipped with composition, is isomorphic to $\mathbb Z/\delta \mathbb Z$ for some integer $\delta$.

\begin{definition}
	A circular order $V$ is \emph{minimal} if it is weakly transitive, of topological rank 1 and admits no self-intertwining.
\end{definition}

\begin{lemma}\label{lem:def completeness}
	Let $V$ be a circular order and let $X_a$, $a\in D$, be a uniformly definable family of non-empty subsets of $V$ which is directed: for any $a,a'\in D$ there is $a''\in D$ such that $X_{a''}\subseteq X_{a}\cap X_{a'}$. Then there is some $c\in \overline V$ such that for any $a\in D$ and any neighborhood $I$ of $c$ in $V$, $I \cap X_a\neq \emptyset$.
\end{lemma}
\begin{proof}
	We fix a  point $d\in V$ and work in the linear order $V_{d\to}$. Let $c\in \overline V$ be equal to $\inf_{a\in D}(\sup X_a)$. (If $c=\pm \infty$, then set $c=d$.) Then $c$ has the required property.
\end{proof}

\begin{prop}\label{prop:one circular order}
	Working over some $A$, let $V$ be a minimal definable circular order. Then for any type $p(x_1,\ldots,x_n)\vdash x_1<\cdots <x_n$ over $A$, and any open intervals $I_1<\cdots <I_n$ of $V$, we can find $a_i\in I_i$ with $(a_1,\ldots,a_n)\models p$.
\end{prop}
\begin{proof}
For simplicity, assume $A=\emptyset$. Fix $a<b$ in $V$ and let $q(x,y)=\tp(a,b)$. Call a convex subset $I$ of $V$ \emph{small} if there are no $a'< b'$ in $I$ with $\tp(a',b')=q$ (where the order $<$ is the canonical one on $I$). Assume that there is some small interval. Then by weak transitivity, for any point $c$ of $V$, there is a small open interval containing $c$. For any $c\in V$, let $s(c)$ be the maximal cut in $V_{c\to}$ so that $(c,s(c))$ is small. We have \[c<d<s(c) \Longrightarrow c<d<s(c)\leq s(d).\] Note that if $c<d<s(c)=s(d)$, then $s(c)=s(e)$ for any $e$, $c<e<d$. Hence the preimage of a cut by $s$ is a convex set. If the preimage of some cut is infinite, then this is true for infinitely many cuts in $\overline V$ by weak transitivity. But then the relation $s(x)=s(y)$ is a convex equivalence relation with infinitely many infinite classes, contradicting topological rank 1. It follows that $s$ is injective. Hence $s\colon V\to \overline V$ is a self-intertwining, which contradicts minimality. We have established that no interval is small. 

Given $a\in V$, let $m(a)$ denote the maximal cut in $V_{a\to}$ definable over $a$ (and $m(a)=a$ if there is no such cut). Chose $a_*\leq m(a_*)<b_*$ in $V$ and set $q=\tp(a_*,b_*)$. Now pick $I_0<I_1<\cdots <I_n$ open intervals of $V$. By the previous paragraph, we can find some pair $(a,b)\models q$ such that $a,b \in I_0$. The interval $x>m(a)$ in $V_{a\to}$ is a linear order which is weakly transitive over $a$. Let $p_{1n}(x_1,x_n)$ be the restriction of $p$ to the variables $(x_1,x_n)$. Applying the previous paragraph to $p_{1n}$, we see that there is a realization of $p_{1n}$ in $\{x\in V_{a\to}:x>m(a)\}$. Such a realization extends to a realization of $p$ lying in that same interval. By Lemma \ref{lem2}, we can find $a_1\in I_1,\ldots, a_n\in I_n$ with $\tp(a_1,\ldots,a_n)=p$.
\end{proof}

\begin{lemma}\label{lem:no circular self-intertwining}
	Let $V$ be a minimal circular order and $I,J\subseteq V$ two disjoint open intervals, then $I$ and $J$ are independent (as linear orders).
\end{lemma}
\begin{proof}
	Assume that some two disjoint intervals $I,J$ of $V$ are intertwined (over some set of parameters). Then by the previous proposition, we can find $I'\subset I$ and $J'\supset J$ disjoint such that the pair $(I',J')$ is a conjugate of $(I,J)$. In particular $I'$ and $J'$ are intertwined and we conclude as in Lemma \ref{lem:no self-intertwining}.
\end{proof}

\begin{definition}
	Two circular orders $V$ and $W$ are \emph{independent} if any open interval of $V$ is independent (as a linear order) from any open interval of $W$.
\end{definition}

\begin{lemma}\label{lem:linear and circular are independent}
	Working over some $A$, let $V$ be a weakly transitive definable circular order of topological rank 1 and $W$ a weakly transitive definable \underline{linear} order of topological rank 1. Then an open interval of $V$ is independent with any open interval of $W$.
\end{lemma}
\begin{proof}
	Assume that $I\subseteq V$ and $J\subseteq W$ are two open intervals definable and weakly transitive over some $B$, which are intertwined. Let $a\in I$. If $I_0$ is an open interval containing $a$ intertwined with some interval $J_0$ of $W$, then by uniqueness of intertwinings (and Lemma \ref{lem:no self-intertwining}), the intertwining maps $I_0\to \overline W$ and $I \to \overline W$ must coincide on $I_0\cap I$. It follows that the element of $\overline W$ to which $a$ is mapped lies in $\dcl(Aa)$: say it is equal to $g(a)$ for some $A$-definable function $g$. Then $g\colon V\to \overline W$ is locally increasing, which is impossible.
\end{proof}

\begin{prop}\label{prop:n circular orders}
	Working over some $A$, let $V_0,\ldots,V_{n-1}$ be minimal definable circular orders, pairwise independent. Then any $A$-definable closed set $X\subseteq V_0^{k_0}\times \cdots \times V_{n-1}^{k_{n-1}}$ is a finite union of products of the form $D_0\times \ldots \times D_{n-1}$, where each $D_i$ is an $A$-definable closed subset of $V_i^{k_i}$.
\end{prop}
\begin{proof}
	Assume $A=\emptyset$. We show the following two statements by induction on $n$. Note that $(B_n)$ implies what we want since by $\omega$-categoricity, any definable set is a finite union of types.

	\begin{itemize}
		\item [($A_n$)] Let $p(\bar x_i:i<n)$ be a type in some product $V_0^{l_0}\times \cdots \times V_{n-1}^{l_{n-1}}$, then given any intervals $I_i\subseteq V_i$, we can find $(\bar a_i:i<n)\models p$ with $\bar a_i\in I_i$ for each $i<n$.
		\item [($B_n$)]  For any type $p$ over $\emptyset$  on $V_0^{k_0}\times \cdots \times V_{n-1}^{k_{n-1}}$, the closure $X$ of the set of realizations of $p$ is equal to the product of its projections to each factor $V_i^{k_i}$.
	\end{itemize}
	
	($B_n$): Assume we know ($A_n$) and we show that ($B_n$) follows.
	
	Let $X$ be given as in ($B_n$) and for $i<n$, let $D_i$ be the projection of $X$ to $V_i^{k_i}$. For each $i<n$, let $T_i\subseteq V_i$ be an open interval and set $T=T_0^{k_0}\times \cdots \times T_{n-1}^{k_{n-1}}$. Since we can choose $T$ to contain any given finite set, it is enough to show the result for $X\cap T$ instead of $X$.
	
	Let $\bar e$ be any tuple of parameters containing at least two points from each $V_i$, $i<n$. For each $i<n$, let $a_i,b_i \in \dcl(\bar e)\cap \overline {V_i}$ be such that the complement of the convex set $a_i\leq x\leq b_i$ in $V_i$ is infinite and weakly transitive over $\bar e$.
	By ($A_n$), we may choose $\bar e$ so that each convex set $a_i\leq x \leq b_i$ is disjoint from $T_i$. Then over $\bar e$, the $T_i$'s are intervals in some weakly transitive $\bar e$-definable linear orders, which are pairwise independent. Therefore by Proposition \ref{lem6}, the restriction of $X$ to $T$ is the product of its projections to each factor, as required.
	
	\smallskip
	($A_n$): Assume that we know ($B_{n-1}$) and we prove ($A_n$).
	
	Let $V=V_0$ and $W=\prod_{0<i<n} V_i$. Given a point $d\in \prod_{0<i<n} V_i$, a neighborhood of $d$ will mean a product $\prod_{0<i<n} J_i$, where each $J_i$ is an open interval containing $d_i$.
	
	Let $c\in V$. Say that a subset $J=\prod_{0<i<n} J_i \subseteq W$ is \emph{good for $c$} if for any open interval $I\subset V$ containing $c$, there is $(\bar b_i)_{i<n}\models p$, with $\bar b_0 \in I$  and $\bar b_i\in J_i$, $i>0$. We claim that there are bounded convex sets $J_i\subset V_i$, $i<n$ such that $\prod_{0<i<n} J_i$ is good for $c$. To see this, take for each $i<n$, $K_{i,1},\ldots,K_{i,t}$ disjoint open intervals of $V_i$, with $t>|\bar b_i|$ and set $J_{i,s}=V_i\setminus K_{i,s}$: a bounded convex subset of $V_i$. By Proposition \ref{prop:one circular order}, for any neighborhood $I$ of $c$, there is $(\bar b_i)_{i<n}\models p$ with $\bar b_0\in I$ and $\bar b_i\in V_i$, $0<i$. For each $0<i<n$, there must be some $s(i)$ such that no coordinate of $\bar b_i$ lies in $K_{i,s(i)}$. As the family of possible $I$ is directed downwards, there is a choice of $s(i)$ which works for all $I$. Let $J_i=J_{i,s(i)}$, $i<n$. Then the set $\prod_{0<i<n} J_i$ is good for $c$.
	
For any $J\subseteq W$ a product of bounded convex sets, let $X(J)\subseteq V$ be the set of elements $c\in V$ for which $J$ is good. Note that $X(J)$ is closed in the order topology and hence is a finite union of closed intervals. For $d\in W$, the family $\{X(J): J$ neighborhood of $d\}$ is directed. By Lemma \ref{lem:def completeness}, there is some $c\in \overline V$ which lies in the closures of each such $X(J)$.
We then have the following property: for any neighborhoods $I$ of $c$ and $J$ of $d$, there is $(\bar b_i)_{i<n}\models p$, with $\bar b_0\in I$ and $\bar b_i\in J_i$, $i>0$. Take a set of parameters $\bar e$ containing two points from each $V_i$ and such that neither $c$ nor $d$ lies in $\acl(\bar e)$. Then, over $\bar e$, there are intervals $J_i\subseteq V_i$ that are weakly transitive and with $c\in J_0$ and $d_i\in J_i$. By assumption, the $J_i$'s are pairwise independent. Therefore by Lemma \ref{lem6}, given any subintervals $J'_i\subseteq J_i$, we can find a realization of $p$ in $\prod_{i<n} J'_i$.

Given $d\in W$, let $Z(d)$ be the set of points $c\in V$ such that any neighborhood $d$ is good for $c$. By the previous paragraph, there is $d$ such that $Z(d)$ has non-empty interior. Then by Proposition \ref{prop:one circular order}, for any open interval $I_*$ of $V$, there is $d_* \in W$ such that $Z(d_*)\supseteq I_*$. Let $Z_*(I_*)$ denote the set of such points $d_*$. For $i<n$ let $\pi_i\colon \prod_{0<j<n} V_j \to V_i$ be the canonical projection. Fix $c\in V$ and for each $0<i<n$ consider the family $\{\pi_i(Z_*(I)) : I$ open interval disjoint from $c\}$. Since the map $Z_*$ is decreasing, that family is directed and by Lemma \ref{lem:def completeness} there is $e_i\in \overline {V_i}$ in the closure of all of its elements. Now ($B_{n-1}$) implies that the closure of $Z_*(I)$ can be written as a union of definable subsets of the form $D_1\times \ldots \times D_{n-1}$, where each $D_i$ is a closed definable subset of $V_i$. Restricting it, we can assume that it is equal to one such set.
We then have that $e=(e_1,\ldots,e_{n-1})$ is in the closure of each $Z_*(I)$, $I$ open interval disjoint from $c$.

We have thus obtained the following property: for any neighborhood $J$ of $e$ in $W$ and any open interval $I$ of $V$ not containing $c$, there are $\bar a\in I$ and $\bar b\in J$ such that $(\bar a,\bar b)\models p$. Since any open interval contains a subinterval not containing $c$, we can remove the requirement that $I$ does not contain $c$. By ($A_{n-1}$), the locus of $e$ is dense in $\overline W$: any product of open intervals in $W$ contains a conjugate of $e$. This shows that for any open $I\subseteq V$ and $J_i\subseteq V_i$, we can find $(\bar b_i)_{i<n}\models p$, with $\bar b_0\in I$ and $\bar b_i\in J_i$, $0<i$, as required.
\end{proof}

\begin{lemma}\label{lem:cutting n orders}
	Working over some $A$, let $V_0,\ldots,V_{n-1}$ be pairwise independent, minimal definable circular orders. Let $V_n,\ldots V_{m-1}$ be pairwise independent minimal definable linear orders. Let $p(\bar x_i:i<m)$ be a type over $A$ in some product $V_0^{l_0}\times \cdots \times V_{m-1}^{l_{m-1}}$, then given any open intervals $I_i\subseteq V_i$, $i<n$ and initial segments $I_i\subseteq V_i$, $n\leq i<m$, we can find $(\bar a_i:i<m)\models p$ with $\bar a_i\in I_i^{l_i}$ for each $i<m$.
\end{lemma}
\begin{proof}
	We can assume that $A=\emptyset$ and that if $\bar a\models p$, then no two coordinates of the tuple $\bar a$ are equal. Let us first assume that all the orders are circular, that is $n=m$. Let $X\subseteq V_0^{l_0}\times \cdots \times V_{n-1}^{l_{n-1}}$ be the closure of the set of realizations of the type $p$. By Proposition \ref{prop:n circular orders}, $X= D_0\times \cdots \times D_{n-1}$, where $D_k \subseteq V_k^{l_k}$ is a closed 0-definable subset. Fix some intervals $I_i\subseteq V_i$, $i<n$. By Proposition \ref{prop:one circular order}, for each $k<n$, $D_k \cap I_k^{l_k}$ has non-empty interior. It follows that $p$ has a realization $(\bar a_i:i<n)$, where each $\bar a_k$ lies in $I_k$.

	Now for the general case, let $p_c(\bar x_i : i<n)$ and $p_l(\bar x_i : n\leq i<m)$ denote the restrictions of $p$ to the corresponding variables. Fix intervals $I_i\subseteq V_i$, $i<n$ in the circular orders. Then by the previous paragraph applied to the restriction of $p$ to the first $n$ tuples of variables, we can find $(\bar a_i:i<m)\models p$ with $\bar a_i\in I_i$ for each $i<n$ (find a realization of $p_c$, then extend it to a realization of $p$). 
	
	Let $B$ be a finite set of parameters over which the intervals $I_i$ are definable. For each $n\leq i<m$, let $J_i\subseteq V_i$ be the minimal $B$-definable initial segment of $V_i$. Hence $J_i$ is a minimal linear order over $B$. Restricting the $I_i$'s if necessary, we can assume that they are also minimal over $B$. If there is a realization $(\bar a_i :i<m)$ of $p$ such that each $\bar a_i\in I_i$ for $i<n$ and $\bar a_i \in J_i$ for $n\leq i<m$, then the result follows at once from Corollary \ref{cor:lem6+} applied to the linear orders $I_i$ and $J_i$.
	
	Assume now that there is no realization $(\bar a_i :i<m)$ of $p$ such that $\bar a_i\in I_i$ for $i<n$ and $\bar a_i \in J_i$ for $n\leq i<m$. Let $\tilde I =\prod_{i<n} I_i^{l_i}$. By Propositions \ref{prop:n circular orders} and \ref{prop:one circular order}, there are finitely many automorphisms $\sigma_1,\ldots,\sigma_k$ such that $\bigcup_{t\leq k} \sigma_t(\tilde I)$ covers $\prod_{i<n} V_i^{l_i}$. Let $J'_i = \bigcap_{t\leq k} \sigma_t(J_i)$, so that each $J'_i$ is a non-empty initial segment of $V_i$. Then we see that there is no realization $(\bar a_i:n\leq i<m)$ of $p_l$ with $\bar a_i \in J'_i$ for each $i$. This contradicts Corollary \ref{cor:lem6+} applied to the independent linear orders $V_i$, $n\leq i<m$.
	\end{proof}

\begin{thm}\label{th:all order types}
	Working over some $A$, let $V_0,\ldots,V_{n-1}$ be pairwise independent, minimal definable circular orders. Let $V_n,\ldots V_{m-1}$ be pairwise independent minimal definable linear orders. Then any $A$-definable closed subset $D\subseteq V_0^{k_0}\times \cdots \times V_{m-1}^{k_{m-1}}$ is a finite union of products of the form $D_0\times \cdots \times D_{m-1}$, where each $D_i$ is an $A$-definable closed subset of $V_i^{k_i}$.
\end{thm}
\begin{proof}
The proof is very similar to that of ($B_n$) in Proposition \ref{prop:n circular orders}, using Lemma \ref{lem:cutting n orders}. Assume $A=\emptyset$.

Let $D$ be given as in the statement and assume that it is the closure of a complete type. For $i<m$, let $D_i$ be the projection of $D$ to $V_i^{k_i}$. For each $i<m$, let $T_i\subseteq V_i$ be a bounded interval and set $T=T_0^{k_0}\times \cdots \times T_{m-1}^{k_{m-1}}$. It is enough to show the result for $D\cap T$ instead of $D$.
	
	Let $\bar e$ be any tuple of parameters containing at least two points from each $V_i$, $i<m$. For each $i<n$, let $a_i,b_i \in \dcl(\bar e)\cap \overline {V_i}$ be such that the complement of the interval $a_i\leq x\leq b_i$ in $V_i$ is infinite and weakly transitive over $\bar e$. For each $n\leq i<m$, let $d_i\in \dcl(\bar e)\cap \overline {V_i}$ such that the end-segment $x>d_i$ is weakly transitive over $\bar e$.

	By Lemma \ref{lem:cutting n orders}, we may choose $\bar e$ so that each interval $a_i\leq x \leq b_i$ is disjoint from $T_i$ for $i<n$, and for $n\leq i<m$, we have $d_i<T_i$. Then over $\bar e$, the $T_i$'s are intervals in some weakly transitive $\bar e$-definable linear orders, which are pairwise independent. Therefore by Proposition \ref{lem6}, the restriction of $X$ to $T$ is the product of its projections to each factor, as required.
\end{proof}

\begin{definition}
	Let $A\subseteq M^{eq}$ and let $(V,C)$ be an $A$-definable circular order. For $n<\omega$, an \emph{$A$-sector} of $V^n$ is a subset of $V^n$ defined by a formula $\phi(x_0,\ldots,x_{n-1})$ which is a finite boolean combination of relations of the form:
	
	\begin{itemize}
		\item $x_i = x_j$, for $i,j<n$;
		\item $C(x_i, x_j, x_k)$, for $i,j,k<n$;
		\item $C(a, x_i,x_j)$, for $i,j<n$ and $a\in \acl^{eq}(A)\cap \overline V$;
		\item $x_i \in W$, for $i<n$ and $W$ an $\acl^{eq}(A)$-definable convex subset of $V$.
	\end{itemize}
\end{definition}

\begin{definition}
	Let $A\subseteq M^{eq}$ and let $V_0,\ldots,V_{m-1}$ be pairwise independent $A$-definable linear or circular orders. For $n<\omega$, a \emph{sector} (resp. \emph{$A$-sector}) of $V_0^{k_0}\times \cdots \times V_{m-1}^{k_{m-1}}$ is a set of the form $D_0\times \cdots \times D_{m-1}$, where each $D_i$ is a sector (resp. $A$-sector) of $V_{i}^{k_i}$.
\end{definition}

\begin{cor}\label{cor:all order types with parameters}
	Let $V_0,\ldots,V_{n-1}$ be pairwise independent, minimal $0$-definable circular orders. Let $V_n,\ldots V_{m-1}$ be pairwise independent minimal $0$-definable linear orders. Let $X\subseteq V_0^{k_0}\times \cdots \times V_{m-1}^{k_{m-1}}$ be definable over some parameters $A$. Then the topological closure of $X$ is an $A$-sector of $V_0^{k_0}\times \cdots \times V_{m-1}^{k_{m-1}}$.
\end{cor}
\begin{proof}
	Each $V_i$ breaks over $A$ into finitely many $A$-definable points and $A$-definable convex subsets, each weakly transitive over $A$. Any two such convex subsets are independent by Lemmas \ref{lem:no self-intertwining}, \ref{lem:no circular self-intertwining} and \ref{lem:linear and circular are independent}. Applying Theorem \ref{th:all order types} to the closure of $X$, it is enough to prove the statement for one linear or circular order. The case of a linear order is Corollary \ref{lem2_cor} and the circular case follows similarly from Proposition \ref{prop:one circular order}.
\end{proof}


\begin{cor}\label{cor:op dimension in orders}
	Let $V_0,\ldots, V_{m-1}$ be pairwise independent, minimal 0-definable linear or circular orders. Let $\bar a$ be a finite tuple of $d$ pairwise distinct elements of $\bigcup_{i<m} \overline{V_i}$. Then $\dlr(\bar a)\geq d$.
\end{cor}
\begin{proof}
	Let $p(\bar x) = \tp(\bar a/\emptyset)$. Then by the previous corollary, $p(\bar x)$ is dense in a sector $S$ of $\prod_{i<m} V_i$. Since $p(\bar x)$ implies that all coordinates of $\bar x$ are distinct, the sector $S$ must have non-empty interior: it contains a product $\prod_{i<m} I_i$, where $I_i \subseteq V_i$ is a bounded open interval of $V_i$. Those intervals are definable over some set $A$.
	
	We can now construct an ird-pattern of length $d$ as follows: write $\bar a =(a_\alpha:\alpha<d)$. For each $\alpha<d$ let $i(\alpha)<m$ be such that $a_\alpha\in V_{i(\alpha)}$. Pick a sequence of points $(b_{\alpha,k} :k<\omega)$ of elements of $I_{i(\alpha)}$, increasing along the order on $I_{i(\alpha)}$. Let $\phi_\alpha(\bar x;b_{\alpha,k})$ be a formula wih extra parameters from $A$ saying that the $\alpha$-th element of $\bar x$ lies in the interval $I_i$ above the point $b_{\alpha,k}$. Then by density, for any $\eta:d\to \omega$, we can find a tuple $\bar a_\eta\models p$ such that for any $\alpha<d$ and $k<\omega$, we have \[ \phi_\alpha(a_\eta;b_{\alpha,k}) \iff \eta(\alpha)<k.\]
	
	This shows that $\dlr(p)\geq d$.
\end{proof}

With the same argument as for Proposition \ref{prop:iso types of n orders}, we can show the following classification result.

\begin{cor}\label{cor:iso types with all order types}
	Let $(M;C_1,\ldots,C_m,\leq_1,\ldots,\leq_n)$ be countable, $\omega$-categorical, transitive, equipped with $m$ circular orders and $n$ linear orders, each minimal. Write $M_i=(M,\leq_i)$. Then the isomorphism type of $M$ is completely determined up to automorphism by the following information:
	\begin{itemize}
		\item For any $i,j\leq m$, whether $C_i$ and $C_j$ are equal, equal up to reversal, intertwined, intertwined up to reversal, or independent.
		\item For any $i,j\leq n$, whether $\leq_i$ and $\leq_j$ are equal, equal up to reversal, intertwined, intertwined up to reversal, or independent.
		\item For any $i<j\leq n$ such that $\leq_i$ and $\leq_j$ are intertwined (possibly up to reversal) but not equal, if $f_{ij}\colon M_i \to \overline {M_j}$ is the intertwining map, whether we have $f_{ij}(x)<_j x$ or $x<_j f_{ij}(x)$ for some/any $x\in M$.
	\end{itemize}
\end{cor}

\section{Local equivalence relations and local formulas}\label{sec:local relations}

We now aim at describing a certain kind of definable sets on products of minimal orders, which we call \emph{local}. This will only be used at the end of the analysis to show the finiteness result of Theorem \ref{th:main1}. We advise the reader to skip this section at first and come back to it when it is called for.

\smallskip
We start by giving examples of local definable sets.

\begin{ex}\label{ex:finite covers}
All structures are assumed to be countable.
	\begin{enumerate}
		\item Let $(V,\leq)$ be a dense linear order without endpoints and let $E$ be an equivalence relation on $V$ with finitely many classes, each of which is dense co-dense. In the structure $(V;\leq,E)$, the order $(V,\leq)$ is weakly transitive and rank 1. The isomorphism type of this structure is determined by the number of classes. One could further expand this structure by adding any structure on the finite quotient $V/E$. We will see that those are the only weakly transitive, rank 1 and op-dimension 1 expansions of a linear order.
			
		\item Let $(V,C)$ be a dense circular order. We may similarly expand it by adding an equivalence relation $E$ with finitely many classes, each of which is dense co-dense. Again, the isomorphism type of the expansion is determined by the number of classes and one can expand the resulting structure by putting any structure on the quotient $V/E$.
		\item Take $(V,C)$ a dense (countable) circular order. Let $\pi\colon W\to V$ be a connected $k$-fold cover of $V$: that is $W$ is itself a circular order, the map $\pi$ is locally an isomorphism and is $k$-to-one. Up to isomorphism, there is a unique such structure. Now let $s\colon V\to W$ be a section of $\pi$ which is generic in the sense that on any small interval of $V$, $s$ takes values in the $k$ sheets of the cover above that interval. Again, those conditions determine the isomorphism type of $(W,V;\pi,s)$.
		
			The induced structure on $V$ can be described in various ways. If $k>1$, let $R(x,y)$ be the binary relation which holds for two points $a,b$ if $\pi$ is injective on the interval $s(a)<x<s(b)$. Note that the circular order on $V$ is definable from $R$ and in fact the whole structure is bi-interpretable with $(V;R)$. Those structures $(V;R)$ are sometimes named $S(k)$ in the literature. We will call them \emph{finite covers} of $V$ (in general a finite cover of a structure $M$ is a structure $N$ equipped with a finite-to-one projection map onto $M$).
			
			Another way to encode the structure on $V$ which will be more natural to us is as a \emph{local equivalence relation}. Define a 4-ary predicate \begin{eqnarray*}E(s,t;x,y) \equiv (s<x=y<t)\vee (s<x<y<t \wedge R(x,y)) \vee \\ \vee(s<y<x<t \wedge R(y,x)).\end{eqnarray*} Then for any $a\neq b$, the relation $E(a,b;x,y)$ is an equivalence relation on the interval $a<x<b$. It is in this form that those structures will appear in our analysis.
		\item We can combine examples (2) and (3). Fix some integers $(k_1,\ldots,k_m)$. Let $(V,C)$ be a dense circular order, equipped with an equivalence relation $E$ with $m$ dense co-dense classes. On the $i$-th class, we have a $k_i$-fold cover coded by a local equivalence relation $E_i$ as in (3). The isomorphism type of the structure $(V;C,E_1,\ldots,E_m)$ is determined  by the tuple $(k_1,\ldots,k_m)$.
		
		As we will see eventually, those are, up to bi-definability, the only minimal, rank 1, op-dimension 1, expansions of circular orders.
		
		\item Let $(V,C)$ be a dense circular order equipped with two equivalence relations $E$ and $F$ such that $F$ has two dense classes, each $E$-class consists of exactly one element from each, and the structure is generic such. Let $M$ be the quotient of $V$ by $E$. Then $M$ satisfies $(\star)$ and is a proper expansion of the last structure in Example \ref{ex:basic} (obtained from $M$ by forgetting about $F$). We then have an equivalence relation with two classes on the set $W_*$ of pairs $(a,b)\in V^2$, with $a\rel{E} b$ (given by the $F$-class of the first coordinate for instance). This is another example of a local equivalence relation. In this case it is a \emph{bona fide} equivalence relation, although not on the structure $M$ itself, but on the finite cover $W_*$.
	\end{enumerate}	
\end{ex}

Let $(V^*_k:k<m_*)$ be a finite family of $0$-definable minimal linear and circular orders so that any two are independent. Let $\bar c=(c_i)_{i<n_*}$ enumerate a relatively algebraically closed subset of $\bigcup V^*_k$ such that $\bar c \in \acl(c_i)$ for each $i$. For $i<n_*$, let $k(i)<m_*$ be such that $c_i\in V^*_{k(i)}$ and set $V_i=V^*_{k(i)}$. Reordering $\bar c$ if necessary, assume that for some $n_c\leq n_*$, $V_i$ is circular for $i<n_c$ and linear otherwise. Let $p_0 = \tp(\bar c)$ and let $W_*\subseteq \prod_{i<n_*} V_i$ be the set of realizations of $p_0$.

Note that as $p_0$ is a complete type, if $S\subseteq \prod_{i<n_*} V_i$ is a sector, then either $W_* \subseteq S$ or $W_* \cap S = \emptyset$.

For each $i<n_*$, let $W_i \subseteq V_i$ be the projection of $W_*$ on $V_i$: it is a dense subset of $V_i$ by minimality and is transitive since $p_0$ is a complete type. If $i\neq j$ and $V_i=V_j$ are linear, then $W_i\neq W_j$ since algebraic closure must be trivial on $W_i$ by Lemma \ref{lem0}. However, if $V_i=V_j$ is circular, then we could have either $W_i\neq W_j$ or $W_i=W_j$.


By the $L_0$-structure, we mean the structure having one sort for each $V^*_k$ equipped with its linear or circular order and a unary predicate for $W_*$ as a subset of $\prod_{i<n_*} V^*_{k(i)}$.
%
%
%

\subsection{Small cells, paths and simple connectedness}

A bounded interval of a linear or circular order is an interval of the form $a<x<b$, with $a< b$.

A \emph{small cell} of $W_*$ is a non-empty set of the form $W_* \cap \prod_{i<n_*} I_i$ such that each $I_i\subseteq V_i$ is a bounded interval, and such that for any $i\neq j$ such that $V_i = V_j$, $I_i$ and  $I_j$ are disjoint. Note that any sector of $\prod_{i<n_*} V_i$ intersecting $\prod_{i<n_*} I_i$ actually contains $\prod_{i<n_*} I_i$ since the disjointness condition ensures that order relations between the variables are the same for any two points in that product. It follows from Corollary \ref{cor:all order types with parameters} applied to $A=\emptyset$ that the intersection $W_* \cap \prod_{i<n_*} I_i$ is dense in $\prod_{i<n_*} I_i$ (as it is assumed to be non-empty).

A \emph{minimal small cell} of $W_*$ over $\bar a$ is a small cell $W_* \cap \prod_{i<n_*} I_i$ such that each $I_i$ is definable and minimal over $\bar a$.

In what follows, we equip $W_*$ with the induced topology coming from the product of the order topologies on each $V_i$.

\begin{lemma}\label{lem:small cells are everywhere}
  Let $X\subseteq W_*$ be a non-empty definable open set and let $C_{\bar a}\subseteq W_*$ be a small cell defined by a formula $\phi(\bar x;\bar a)$. Then there is $\bar a' \equiv \bar a$, such that the small cell $C_{\bar a'}$ defined by $\phi(\bar x;\bar a')$ is included in $X$.
\end{lemma}
\begin{proof}
  Let $d\in X$. Then as all elements of $W_*$ have the same type, there is $\bar b\equiv \bar a$ such that $d\in C_{\bar b}$. By Corollary \ref{cor:all order types with parameters}, the set of realizations of $\tp(\bar a)=\tp(\bar b)$ is dense in a sector, hence we can move the endpoints of the intervals defining $C_{\bar b}$ freely, as long as the order type of those endpoints is preserved. Therefore we can find $\bar b'\equiv \bar b$ such that $C_{\bar b'}$ contains $d$ and is small enough to be included in $X$.
\end{proof}

\begin{lemma}\label{lem:small cells cover a small cell}
  Let $C_{\bar a}$ and $D_{\bar b}$ be two small cells, defined by possibly different formulas. Then there is $\bar a' \equiv \bar a$ such that $C_{\bar a'} \supseteq D_{\bar b}$.
\end{lemma}
\begin{proof}
  By the previous lemma, there is $\bar b' \equiv \bar b$ such that $C_{\bar a} \supseteq D_{\bar b'}$. Now take $\bar a'$ such that $\tp(\bar a' \bar b) = \tp(\bar a \bar b')$.
\end{proof}

In the following definition, we slightly abuse terminology: when we write $E_{\bar a}$, we are assuming given not just the definable set $E_{\bar a}$, but also the formula $\phi(x,y;\bar z)$ so that $\phi(x,y;\bar a)$ defines $E_{\bar a}$. (In fact, the definition does not actually depend on the formula chosen, but we will not need this.)

\begin{definition}
	Let $C_{\bar a}$ be a minimal small cell defined over $\bar a$ and let $E_{\bar a}$ be an $\bar a$-definable equivalence relation on $C_{\bar a}$. We say that $E_{\bar a}$ is a \emph{local equivalence relation} if for any $\bar a'\equiv \bar a$ such that $C_{\bar a'}\subseteq C_{\bar a}$, $E_{\bar a'}$ and $E_{\bar a}$ coincide on $C_{\bar a'}$. 
\end{definition}

\begin{lemma}\label{lem:local equivalence rel are local}
	Let $E_{\bar a}$ be a local equivalence relation defined on the minimal small cell $C_{\bar a}$ and take $\bar a'\equiv \bar a$. Let $D_{\bar c} \subseteq C_{\bar a}\cap C_{\bar a'}$ be a small cell defined possibly by a different formula. Then $E_{\bar a}$ and $E_{\bar a'}$ coincide on $D_{\bar c}$.
\end{lemma}
\begin{proof}
	Let $\bar d$ be a finite tuple of points from $D_{\bar c}$. By Lemma \ref{lem:small cells are everywhere}, there is $\bar a''\equiv \bar a$ such that $C_{\bar a''}\subseteq D_{\bar c}$. Note that a minimal small cell is by definition a product of pairwise independent minimal linear orders. It then follows from Corollary \ref{cor:lem6+} that we can find a realization of $\tp(\bar d/\bar c)$ inside any open subset of $D_{\bar c}$ and in particular we can find such a realization $\bar d'$ inside $C_{\bar a''}$. Let $\bar a_*$ be such that $\bar d' \bar a'' \equiv_{\bar c} \bar d \bar a_*$. Then $C_{\bar a_*}$ is included in $D_{\bar c}$ and contains $\bar d$. Hence for any finite subset of $D_{\bar c}$ we have found a small cell included in $D_{\bar c}$ and containing that finite set. The result therefore follows by the definition of a local equivalence relation.
\end{proof}

Observe that a non-empty intersection of two small cells need not be a small cell: the intersection of two intervals in a circular order may be two disjoint intervals. This is where the topological complexity comes in and is the reason why a local equivalence relation does not always give a \emph{bona fide} equivalence relation on the whole of $W_*$.

Fix a local equivalence relation $E_{\bar a}$ and let $\mathcal E$ be the family $\{E_{\bar a'}:\bar a'\equiv \bar a\}$. We will also refer to $\mathcal E$ as a local equivalence relation. For any small cell $C$, we can find by Lemma \ref{lem:small cells cover a small cell} $E\in \mathcal E$ whose domain contains $C$. Then by the previous lemma, $E|_C$ does not depend on the choice of $E\in \mathcal E$. We will denote that equivalence relation by $\mathcal E(C)$ and its set of classes by $C/\mathcal E$.

We say that the local equivalence relation $E_{\bar a}$ (or $\mathcal E$) is \emph{finite} if $E_{\bar a}$ has finitely many classes on its domain $C_{\bar a}$.
%
%
%
%

\begin{lemma}\label{lem:density}
	Let $\mathcal E$ be a finite local equivalence relation. For any small cell $C$, any $\mathcal E(C)$-class is dense in $C$. Furthermore, $\mathcal E(C)$ has finitely many classes and that number does not depend on $C$.
\end{lemma}
\begin{proof}
	It follows at once from Corollary \ref{cor:lem6+} that any $E_{\bar a}$-class is dense in the minimal small cell $C_{\bar a}$. Since any small cell $C$ is included in a conjugate of $C_{\bar a}$, the result follows.
\end{proof}

In what follows, $\mathcal E$ is a finite local equivalence relation.

If $C_0, C_1$ are small cells such that $C_0\cap C_1$ is also a small cell, then we have a natural bijection $f\colon C_0/\mathcal E \to C_1/\mathcal E$ given by identifying both $C_0/\mathcal E$ and $C_1/\mathcal E$ with $C_0\cap C_1/\mathcal E$.

\begin{definition}\label{def:path}
	A \emph{path} is a family $\mathfrak p=(C_i)_{i<n}$ such that each $C_i$ is a small cell and each $C_i\cap C_{i+1}$ is a small cell.
\end{definition}

Given a path $\mathfrak p=(C_i)_{i<n}$, we can define a map $f_{\mathfrak p}\colon C_0/\mathcal E \to C_{n-1}/\mathcal E$ given by composing the natural bijections $f_i\colon C_i/\mathcal E \to C_{i+1}/\mathcal E$ defined above.

\begin{definition}\label{def:refines}
	Say that a path $\mathfrak p'=(C'_i)_{i<n'}$ refines a path $\mathfrak p=(C_i)_{i<n}$ if there exist indices \[0=i_0<\cdots <i_{n-1}<i_n = n'\] such that $i_k\leq i < i_{k+1}$ implies $C'_i\subseteq C_k$.
\end{definition}


\begin{prop}\label{prop:homotopy}
	\begin{enumerate}
		\item If a path $\mathfrak p=(C_i)_{i<n}$ satisfies that all the $C_i$'s lie in some given small cell $C$, then $f_{\mathfrak p}\colon C_0/\mathcal E\to C_{n-1}/\mathcal E$ is given by the identification of $C_0/\mathcal E$ and $C_{n-1}/\mathcal E$ to $C/\mathcal E$.
		\item If a path $\mathfrak p'$ refines $\mathfrak p$, then $f_{\mathfrak p'}$ is equal to $f_{\mathfrak p}$, modulo the canonical identifications of the domain and range given by inclusion maps.

	\end{enumerate}
\end{prop}
\begin{proof}
	The proof of (1) is immediate by induction on $n$.

To prove (2), let $0=i_0<\cdots <i_{n-1}<i_n = n'$ be as in Definition \ref{def:refines}. The map from $C'_0/\mathcal E$ to $C'_{i_1-1}/\mathcal E$ obtained following $\mathfrak p'$ is given by the identification of both to $C_0/\mathcal E$. Then since $C'_{i_1-1}\cap C'_{i_1}\subseteq C_0\cap C_1$, the map  $C'_{i_1-1}/\mathcal E \to C'_{i_1}/\mathcal E$ is the same---up to canonical identification of domain and range---as the one $C_0/\mathcal E \to C_1/\mathcal E$. Going on in this way proves the result.
\end{proof}

If $X\subseteq W_*$ then a \emph{path in $X$} is a path $\mathfrak p= (C_i:i<n)$, where each $C_i$ is included in $X$.

\begin{definition}
\begin{enumerate}
	\item An open definable set $X\subseteq W_*$ is \emph{path-connected} if for any two points $a,b\in X$, there is a path $\mathfrak p=(C_i:i<n)$ in $X$ with $a\in C_0$ and $b\in C_{n-1}$.
	\item An open set $X\subseteq W_*$ is \emph{simply connected} if it is path-connected and for any two paths $\mathfrak p=(C_i:i<n)$ and $\mathfrak p'=(C'_i:i<n')$ in $X$ with $C_0=C'_0$, $C_{n-1}=C'_{n'-1}$, the maps $f_{\mathfrak p}$ and $f_{\mathfrak p'}$ are equal.
	\end{enumerate}
\end{definition}


Let $X\subseteq W_*$ be a simply connected open set. Let $a,b\in X$ and take a path $\mathfrak p$ in $X$ from some small cell $C_a$ containing $a$ to a small cell $C_b$ containing $b$. This induces a bijection $f_{\mathfrak p}\colon C_a/\mathcal E \to C_b/\mathcal E$. Say that $a$ and $b$ are $\mathcal E(X)$-related if $f_{\mathfrak p}$ maps the $\mathcal E(C_a)$ class of $a$ to the $\mathcal E(C_b)$-class of $b$. This notion does not depend on the choice of $\mathfrak p$ by definition. It also does not depend on the choice of $C_a$ and $C_b$, since if we make a different choice, say $C'_a$ and $C'_b$, related by a path $\mathfrak p'$, then we can find $C''_a\subseteq C_a \cap C'_a$ and $C''_b\subseteq C_b\cap C'_b$ and any map $f_{\mathfrak p''}\colon C''_a/\mathcal E \to C''_b/\mathcal E$ coming from a path must coincide (modulo canonical identifications) with $f_{\mathfrak p}$ and $f_{\mathfrak p'}$.

 We therefore see that $\mathcal E(X)$ is an equivalence relation on $X$. Furthermore, it follows by construction that if $Y\subseteq X$ are both simply connected, then $\mathcal E(X)$ and $\mathcal E(Y)$ coincide on $Y$. Also if $C$ is a small cell, then by Proposition \ref{prop:homotopy} (1), this definition of $\mathcal E(C)$ coincides with the previous one. Finally, note that if $X$ is definable, then so is $\mathcal E(X)$ since it is automorphism-invariant.
 
\begin{lemma}
\begin{enumerate}
	\item If $X$ is simply connected, then any $\mathcal E(X)$-class is dense in X.
	\item If $X$ and $Y$ are simply connected, then the equivalence relations $\mathcal E(X)$ and $\mathcal E(Y)$ have the same number of classes.
\end{enumerate}
\end{lemma}
\begin{proof}
	1. Let $X$ be simply connected and let $C_0,C_1\subseteq X$ be small cells. Then there is a path $\mathfrak p$ in $X$ from some $C'_0\subseteq C_0$ to some $C'_1 \subseteq C_1$. This path induces a bijection $f_{\mathfrak p}\colon C'_0/\mathcal E \to C'_1/\mathcal E$ which in turns induces a bijection $C_0/\mathcal E \to C_1/\mathcal E$ via the canonical identifications induced by the inclusion maps. Therefore $C_0$ and $C_1$ intersect the same $\mathcal E(X)$-classes, hence every class is dense in $X$.
	
	2. By Lemma \ref{lem:density}, any two cells have the same number of $\mathcal E$-classes. Furthermore, each of $\mathcal E(X)$ and $\mathcal E(Y)$ has the same number of classes as $\mathcal E(C)$ for some/any small cell $C$ contained in them, since every class is dense. 
\end{proof}

\begin{lemma}\label{lem:criterion for simply connected}
	Let $X$ be an open subset of $W_*$. Assume that we have a family $\mathcal F$ of definable (over parameters) open subsets of $X$ such that:
	
	1. for any finite collection $\{C_1,\ldots, C_k\}$ of small cells included in $X$, there is a finite set $F\subseteq \mathcal F$ whose union contains all the $C_i$'s;
	
	2. for any non-empty finite set $F\subseteq \mathcal F$ the intersection of all the sets in $F$ is non-empty and simply connected.
	
	Then $X$ is simply connected.
\end{lemma}
\begin{proof}
	To see that $X$ is connected, let $a,b\in X$. We can find two sets $X_a,X_b\in \mathcal F$ that contain $a$ and $b$ respectively. By assumption $X_a\cap X_b$ is non-empty. Pick a point $c$ in it. Then since both $X_a$ and $X_b$ are connected, there are paths from $a$ to $c$ and from $c$ to $b$, which we can compose to obtain a path from $a$ to $b$.

	Let $\mathfrak p=(C_i:i<n)$ and $\mathfrak p'=(C'_i:i<n')$ be two paths with $C_0=C'_0$, $C_{n-1}=C'_{n'-1}$. Let $F$ be the finite set promised by condition 1 for the family $\{C_0,\ldots,C_{n-1},C'_0,\ldots,C'_{n'-1}\}$. Refining the two paths, we may assume that each $C_i$ and $C'_i$ lies in a unique member of the family. Let $F_\infty$ be the intersection of all the sets in $F$. By hypothesis $F_\infty$ is simply connected, so $\mathcal E(F_\infty)$ is well defined. Then we see that the transition maps from $C_i/\mathcal E\to C_{i+1}/\mathcal E$ coincide with the identification of both domain and range with $F_\infty/\mathcal E$, and same for the primed family. Hence the two maps $f_{\mathfrak p}$ and $f_{\mathfrak p'}$ are also defined in this way and therefore coincide.
\end{proof}

\begin{lemma}
	Assume that all the orders $V_i$ are linear and that the map $k$ is injective: no two coordinates of a $\bar c\in W_*$ lie in the same order. Then $W_*$ is simply connected.
\end{lemma}
\begin{proof}
	Any finite union of small cells of $W_*$ is included in one small cell (any finite union of bounded intervals of a linear order is included in one bounded interval and the same holds for products). Hence Proposition \ref{prop:homotopy} (1) directly implies that $W_*$ is simply connected.
\end{proof}

\begin{lemma}\label{lem:linear orders are simply connected}
	Assume that all the orders $V_i$ are linear. Then $W_*$ is simply connected.
\end{lemma}
\begin{proof}
	We prove the result by induction on the number of pairs $(i,j)$ for which $V_i=V_j$. If there is no such pair, then the previous lemma applies.

	Assume now that say $V_0=V_1= \ldots =V_{k-1}$ and $V_i \neq V_0$ for $i\geq k$. Without loss of generality, assume that $p_0(\bar x)\vdash x_0 < \cdots <x_{k-1}$. Consider the family $\mathcal F$ of non-empty sets of the form \[W_*\cap J_0 \times \prod_{0<i<k} J_1\times \prod_{k\leq i} V_i,\] where $J_0$ is an initial segment of $V_0$ and $J_1$ the complementary end segment. Any finite intersection of those sets is a non-empty set of the form \[W_*\cap K_0 \times \prod_{0<i<k} K_1 \times \prod_{k\leq i} V_i,\] where $K_0$ is an initial segment and $K_1$ some end segment of $V_0$. In such a set, the first coordinate lives in the linear order $K_0$, and all the others are in orders independent from it. By induction, that set is simply connected and we conclude by Lemma \ref{lem:criterion for simply connected}.
\end{proof}

\begin{prop}\label{prop:linear are simply connected}
	For each $i<n_*$, let $I_i\subseteq V_i$ be an open bounded interval of $V_i$, if $V_i$ is circular, or either an interval or the whole of $V_i$ if $V_i$ is linear. Then $X:= W_* \cap \prod_{i<n_*} I_i$ is empty or simply connected.
\end{prop}
\begin{proof}
	This follows at once from Lemma \ref{lem:linear orders are simply connected} applied to $W_* \cap \prod_{i<n_*} I_i$ instead of $W_*$.
\end{proof}

\subsection{Classification of finite local equivalence relations}

Let $\mathcal E$ be a finite local equivalence relation. Fix an $L_0$-formula $\psi(x;\bar y)$ and an $L_0$-type $q(\bar t)$ such that for any $\bar a\models q$, $C_{\bar a} := \psi(M;\bar a)$ is a small cell. Define the relation $E(\bar t;\bar x,\bar y)$ which holds for $\bar x,\bar y\in W_*$ and $\bar t \models q$ if $\bar x, \bar y$ are in $C_{\bar t}$ and are $\mathcal E(C_{\bar t})$-equivalent. Let $L_{\mathcal E}$ be the language $L_0\cup \{E\}$ and our goal now is to describe the possibilities for the isomorphism type of the expansion of the $L_0$ structure to $L_{\mathcal E}$.

\smallskip
Let $\mathcal{C}$ be the set of indices $k<m_*$ for which $V^*_k$ is circular.

For each $k\in \mathcal C$, let three distinct points $\alpha_k<\beta_k<\gamma_k \in V^*_k$ be given. Define three intervals $C_{k,0}:=\alpha_k<x<\gamma_k$, $C_{k,1}:=\beta_k<x<\alpha_k$ and $C_{k,2}:=\gamma_k<x<\beta_k$ of $V^*_k$.
The indices $0,1,2$ in $C_{k,0},...$ are considered as elements of the cyclic group $\mathbb Z_3$. Let also $A=\{\alpha_k,\beta_k,\gamma_k:k\in \mathcal C\}$.

Note that any two of $C_{k,0}, C_{k,1}, C_{k,2}$ intersect in a non-empty bounded interval of $V^*_k$.

Recall that $n_c\leq n_*$ was defined so that for $i<n_*$, $V_i$ is circular if and only if $i<n_c$. Given a tuple $\bar t=(t_k:k<n_c)$ of elements of $\mathbb Z_3$, let \[C_{\bar t}= W_* \cap \prod_{i<n_c} C_{k(i),t_k} \times \prod_{n_c\leq i<n_*} V_i.\]

A \emph{big cell} of $W_*$ is a set of the form $C_{\bar t}$, with $\bar t\in \mathbb Z_3 ^{n_c}$ as above. By Lemma \ref{lem:linear orders are simply connected}, each big cell is simply connected. Furthermore, the intersection \[C(\bar t,\bar s):= C_{\bar t}\cap C_{\bar s}\] of two big cells is a non-empty product of intervals and linear orders and hence is also simply connected. It follows that $\mathcal E(C_{\bar t})$ is a well defined equivalence relation on each big cell $C_{\bar t}$ and $\mathcal E(C(\bar t,\bar s))$ is a well defined equivalence relation on each $C(\bar t,\bar s)$. The latter induces a bijection between $C_{\bar t}/\mathcal E$ and $C_{\bar s}/\mathcal E$, which we will denote by $f_{\bar t,\bar s}$.

Let $M$ and $M'$ be two $L_{\mathcal E}$ structures with isomorphic $L_0$-reducts. Fix an isomorphism $\sigma:M \to M'$ between the $L_0$-reducts. Let $\alpha_k,\beta_k,\gamma_k$ in $M$ be points defining big cells and let $\alpha'_k,\beta'_k,\gamma'_k$ be their images under $\sigma$. Assume that the number of $\mathcal E$-classes are the same in $M$ and $M'$ and that for each $\bar t\in \mathbb Z_3 ^ {n_c}$, we have an identification of the classes in $C_{\bar t}$ and $C'_{\bar t}$ so that the maps $f_{\bar t,\bar s}$, for $\bar t, \bar s \in \mathbb Z_3 ^ {n_c}$ are the same in $M$ and $M'$ (modulo this identification). Then $M$ and $M'$ are isomorphic as $L_{\mathcal E}$-structures. Indeed, we can construct an isomorphism by a straightforward back-and-forth: $\mathcal E$-classes on each $C_{\bar t}$ are dense subsets and the identification ensures that the local equivalence relations coincide.


\subsection{Local formulas}

Say that two small cells $C_0, C_1$ of $W_*$ are \emph{strongly disjoint} if for any $i,j<n_*$ so that $V_i=V_j$, the projections $\pi_i(C_0)$ and $\pi_j(C_1)$ to $V_i$ and $V_j$ are disjoint.

\begin{definition}
  A (parameter-)definable set $R(x_1,\ldots,x_k)\subseteq W_*^k$ is \emph{local} if there is a finite local equivalence relation $\mathcal E_R$ on $W_*$ such that given strongly disjoint small cells $C_1,\ldots,C_k$ and two tuples $(a_1,\ldots,a_k),(a'_1,\ldots,a'_k)\in C_1\times \cdots \times C_k$, \[\bigwedge (a_i,a'_i)\in \mathcal E_R(C_i) \Longrightarrow (R(a_1,\ldots,a_k) \leftrightarrow R(a'_1,\ldots,a'_k)).\]
\end{definition}

We say that a formula is \emph{local} if it defines a local definable set.

\begin{rem}
  There is a slight clash of terminology with \emph{local equivalence relation}. A local equivalence relation is not the same thing as a local formula defining an equivalence relation, but we will never consider such objects so hopefully this should not lead to  confusion.
\end{rem}

\begin{prop}\label{prop:local relations}
	Let $R(x_1,\ldots,x_k)$ be a local definable set. Let $\bar a=(a_1,\ldots,a_k), \bar b=(b_1,\ldots,b_k)\in W_*^k$ be two tuples of pairwise distinct elements. Assume that $\bar a$ and $\bar b$ have the same $L_0$-type and that for each $i\leq k$, there is a big cell $C$ of $W_*$ containing both $a_i$ and $b_i$ with $(a_i,b_i)\in \mathcal E_R(C)$. Then we have 
	\[R(a_1,\ldots,a_k) \leftrightarrow R(b_1,\ldots,b_k).\]
\end{prop}
\begin{proof}(Sketch)
	For any two $k$-tuples $\bar c$ and $\bar d$ of elements of $W_*$, write $\bar c\to \bar d$ if for each $i\leq k$, there is a big cell $C_i$ of $W_*$ and a small cell $C_i'\subseteq C_i$ that contains $c_i$ and $d_i$ and such that $(c_i,d_i)\in \mathcal E_R(C'_i)$ and the $C'_i$'s are strongly disjoint. To prove the proposition, it is sufficient to find a sequence $\bar a=\bar a^ 0 \to \bar a ^ 1 \to \cdots \to \bar a ^m =\bar b$.	The fact that the $L_0$-types of $\bar a$ and $\bar b$ are the same implies that the relative order of the elements in the tuple are the same. Thus we can always find such a path from $\bar a$ to $\bar b$ by moving the points one by one.
%
%
%
\end{proof}

It follows that a local definable set $R$ is definable over the parameters $A$ used to define the big cells along with parameters defining the equivalence relations $\mathcal E_R$ on each big cell and a name for each $\mathcal E_R$-equivalence class inside each big cell. Also, for a fixed $L_{\mathcal E}$-structure, there are only finitely many local definable sets of each arity.

\begin{prop}\label{prop:finitely many local}
  For a given $L_0$-structures and some $n<\omega$ there are finitely many possibilities for the $L_{\mathcal E}$-structure, where $\mathcal E$ has $n$ classes (in some/any small cell). Furthermore, for a given $L_{\mathcal E}$-structure, there are finitely many local definable sets of a given arity.
\end{prop}

\subsection{Monodromy}

The previous discussion gives us all we need to prove the main theorem of this paper. However, it is natural to push the analysis a little bit further and show that the data contained in the set of maps $f_{\bar t,\bar s}$ can be encoded by an action of the fundamental group of the space on a finite set. We explain this here. This subsection will not be used in the rest of the paper.

\begin{lemma}\label{lem:simply connected products}
	For each $i<n_*$, let $I_i\subseteq V_i$ be either an open bounded interval of $V_i$ or the whole of $V_i$. Assume that for each $k<m_*$ such that $V^*_k$ is circular, there is exactly one value of $i$ for which $V_i=V^*_k$ and $I_i\neq V_i$. Then $X:=W_*\cap \prod_{i<n_*} I_i$ is empty or simply connected.
\end{lemma}
\begin{proof}
	We first explain what this corresponds to in a standard topological framework. Let $\tilde V^*_k$, $k<m_*$, be 1-dimensional manifolds, which are thus homeomorphic to either $\mathbb R$ or the circle $\mathbb S_1$. Let $\tilde V_i$, $i<n_*$ be each equal to one of the $\tilde V^*_k$ and let $\tilde U\subseteq \prod_{i<n_*} \tilde V_i$ be the set of tuples with distinct coordinates. Let $\tilde W_*$ be a connected component of $\tilde U$. Choose open intervals $\tilde I_i \subseteq \tilde V_i$ satisfying the same condition as in the statement of the lemma. Then the set $\tilde X= \tilde W_* \cap \prod_{i<n_*} \tilde I_i$ is simply connected. In fact this space is contractible. This is not hard to see: First, we can assume that $m_*=1$, since the space decomposes as a product of spaces each involving one $\tilde V^*_k$ and a product of contractible spaces is contractible. Let us assume for example that $\tilde V^*_0$ is circular. At least one coordinate, say $i=0$ is constrained inside a proper interval $\tilde I_0$. Fix any element $\bar a\in \tilde X$. Then we can send any other element $\bar a'$ to $\bar a$, by sending $a'_0$ to $a_0$ via a shortest path (and moving the other coordinates with it so that no two cross). We then move only the other coordinates in the circle minus $\{a_0\}$, and this reduces to the linear case which is clear.
	
	Now, we just have to translate this topological intuition into an argument in our context. The reader who is already convinced will not lose anything by skipping the rest of this proof. Assume that $X$ is not empty. As above, we can assume that $m_*=1$: all points live in the same order $V^*_0$, since coordinates in different $V^*_k$ are completely independent of each other. If $n_*=1$, then this follows from Proposition \ref{prop:homotopy} (1): any finite set of bounded intervals is included in one bounded interval, so any two paths are included in one common bounded interval and thus define the same functions $f_{\mathfrak p}$.
	
	Assume that $V^*_0$ is linear, and we prove the result by induction on $n_*$. Without loss $p_0(\bar x)\vdash x_0 < \cdots <x_{n_*-1}$. Consider the family $\mathcal F$ of non-empty sets of the form $X\cap J_0 \times \prod_{i<n_*} J_1$, where $J_0$ is an initial segment of $V^*_0$ and $J_1$ the complementary end segment. Any finite intersection of those sets is a non-empty set of the form $X\cap L_0 \times \prod_{i<n_*} L_1$, where $L_0$ is an initial segment and $L_1$ some end segment of $V^*_0$. In such a set, the first coordinate lives in the linear order $L_0$ and the others in $L_1$ which is independent from it. By induction, that set is simply connected and we conclude by Lemma \ref{lem:criterion for simply connected}.
	
	Assume next that $V^*_0$ is circular. Without loss, $I_0$ is a proper interval and $I_i=V_i$ for $i>0$. We may also assume that $p_0(\bar x)\vdash x_0<x_1<\cdots <x_{n_*-1}$. Fix some $I_* \subset I_0$ a proper subinterval that has no endpoint in common with $I_0$ and let $J_*$ be the complement of $I_*$. Define $F$ to be $W_* \cap I_* \times \prod_{0<i<n_*} J_* \subseteq X$. By the linear case, $F$ is simply connected.

	
	Identify $\{0,\ldots,n_*-1\}$ with $\mathbb Z/n_*\mathbb Z$. Let $\mathcal S$ be the set of pairs $(t,k)\in \mathbb Z/n_*\mathbb Z^2$ such that the sequence $(t,t+1,\ldots,t+k)$ contains 0. For $(t,k)\in \mathcal S$, let $G_{t,k} \subseteq X$ be the set of tuples $\bar a\in X$ for which $a_{t},\ldots,a_{t+k}$ lie in $I_0$ in that order and no other $a_i$ is in $I_0$. Again using the linear case, any such set is simply connected. Note also that two distinct $G_{t,k}$ are disjoint. For $(t,k)\in \mathcal S$, $G_{t,k}\cap F$ has the form $\prod_{i<n_*} I_i$, where the $I_i$'s are intervals, any two of which are either equal or disjoint. From the linear case, it follows that $G_{t,k}\cap F$ is simply connected. Enumerate the elements of $\mathcal S$ arbitrarily as $s_1,\ldots,s_v$. For $r\leq v$, let $F_r = F \cup \bigcup_{i<r} G_{s_i}$. By induction using the remarks above and Lemma \ref{lem:criterion for simply connected} with the two element family $\{F_{r-1}, G_{s_{r}}\}$, we see that each $F_r$ is simply connected. Since $F_v=X$, we are done.	
\end{proof}

%

Let $\bar t\in \mathbb Z_3^{\mathcal C}$ and take $\bar \epsilon_0,\bar \epsilon_1\in \mathbb Z_3^{\mathcal C}$ having each exactly one non-zero coordinate, with $\bar \epsilon_0 \neq \pm \bar \epsilon_1$.
Then the big cells $C_{\bar t}$, $C_{\bar t +\bar \epsilon_0}$, $C_{\bar t+\bar \epsilon_1}$, $C_{\bar t+\bar \epsilon_0+\bar \epsilon_1}$ are included in a common simply connected set. It follows that we have the commutation relation:

\begin{itemize}
	\item [$(\square)$] $f_{\bar t+\bar \epsilon_0,\bar t+\bar \epsilon_0+\bar \epsilon_1} \circ f_{\bar t,\bar t+\bar \epsilon_0} = f_{\bar t+\bar \epsilon_1,\bar t+\bar \epsilon_0+\bar \epsilon_1} \circ f_{\bar t,\bar t+\bar \epsilon_1}.$
\end{itemize}

Denote by $\bar 0\in \mathbb Z_3^{\mathcal C}$ the tuple all of whose coordinates are 0 and let $X=C_{\bar 0}/\mathcal E$. We may identify each $C_{\bar t}/\mathcal E$ with $X$ by following a path of bijections between $C_{\bar 0}$ and $C_{\bar t}$ that never \emph{wraps around}. More formally, order $\mathbb Z_3$ by identifying it with $\{0,1,2\}$. If $C_{\bar t_0}, \ldots, C_{\bar t_n}$ and $C_{\bar s_0},\ldots ,C_{\bar s_n}$ are two sequences of cells with
\[ \bar t_0\leq \bar t_1 \leq \ldots \leq \bar t_n \text{, }\bar s_0\leq \bar s_1 \leq \ldots \leq \bar s_n\text{, and }\bar t_0=\bar s_0, \bar t_n=\bar s_n\]
and both
\[ f_{\bar t_{n-1},\bar t_n}\circ \cdots \circ f_{\bar t_0,\bar t_1} \text{ and } f_{\bar s_{n-1},\bar s_n}\circ \cdots \circ f_{\bar s_0,\bar s_1}\]
well defined, then those two compositions are equal by iterations of $(\square)$. We identify $C_{\bar t}/\mathcal E$ with $X=C_{\bar 0}/\mathcal E$ by following any sequence of adjacent big cells from $C_{\bar 0}$ to $C_{\bar t}$ as above.

For any $i\in \mathcal C$, let $\bar \epsilon_i\in \mathbb Z_3^{\mathcal C}$ be the element with coordinates 0 everywhere except for $2$ at the $i$-th place. Now to describe $\mathcal E$, it is enough to describe the maps $f_{\bar t,\bar t+\bar \epsilon_i}$ when the $i$-th coordinate of $\bar t$ is equal to 0. (All other maps $f_{\bar t,\bar s}$ are the identity on $X$ by our identification.) In fact, we can further simplify by noticing that such an $f_{\bar t,\bar t+\bar \epsilon_i}$ is equal to $f_{\bar 0,\bar \epsilon_i}$: let $g$ be a composition of maps $f_{\bar t,\bar s}$, which do not wrap around (that is change a coordinate from 2 to 0 or vise-versa), such that $t_i=s_i=2$ so that the $i$-th coordinate is not changed and $g\circ f_{\bar t,\bar t+\bar \epsilon_i}$ maps $C_{\bar t}/\mathcal E$ to $C_{\bar \epsilon_i}/\mathcal E$. Let $h$ be the same composition as $g$, but with all $i$-th coordinate being equal to 0 instead of 2. Then $h$ sends $C_{\bar t}/\mathcal E$ to $C_{\bar 0}/\mathcal E$ and $f_{\bar 0,\bar \epsilon_i}\circ h$ also sends $C_{\bar t}/\mathcal E$ to $C_{\bar \epsilon_i}/\mathcal E$. As neither $g$ nor $h$ wraps around, $g$ and $h$ induce the identity map on $X$. Furthermore, by successive applications of $(\square)$, one sees that \[g\circ f_{\bar t,\bar t+\bar \epsilon_i} = f_{\bar 0,\bar \epsilon_i}\circ h.\]
Hence, seen as maps from $X$ to $X$, we have $f_{\bar t,\bar t+\bar \epsilon_i}=f_{\bar 0,\bar \epsilon_i}$

For each index $i$, set $h_i=f_{\bar 0,\bar \epsilon_i}$, seen as a map from $X$ to $X$. Using $(\square)$ and following the standard argument that the fundamental group of a torus is $\mathbb Z^2$, one obtains that $h_i$ and $h_j$ commute for all $i,j$. (Deform the path corresponding to $h_i\circ h_j$ to that corresponding to $h_j \circ h_i$ by successive applications of $(\square)$.)

We have thus associated to the local equivalence relation $\mathcal E$ a family of pairwise commuting maps $h_i\colon X\to X$, or equivalently, an action of $\mathbb Z^{\mathcal C}$ on $X$. We will call this the \emph{monodromy action} of $\mathcal E$. Given a decomposition of $W_*$ into big cells, this action is well defined only up to conjugation by a permutation of $X$. Furthermore, it follows from the analysis above that another choice of big cells would lead to the same family of maps up to conjugation. The monodromy action determines the $L_{\mathcal E}$-structure up to isomorphism (given the $L_0$-reduct) since we can from it reconstruct a set of maps $f_{\bar t,\bar s}$.

\section{Classification of rank 1 structures}\label{sec:classification}

In this section, we prove our main theorems. In Subsection \ref{sec:no binary} we prove some important technical statements that follow from having a ranked structure: first a kind of \emph{geometric triviality} property saying that any binary function into a minimal order is essentially unary, and second we show that minimal linear orders can intersect (up to intertwining) only in very restricted ways. For instance it is impossible for two minimal orders to have proper initial segments that are intertwined, while the remaining final segments are independent. (Think of two branches in a tree for instance: they start out equal and then diverge. This situation will be ruled out by showing that if it happens then we can increase this pair of orders to a whole tree and trees cannot exist in ranked structures).

In Subsection \ref{sec:gluing} we show how starting with a definable family of minimal linear orders, we can glue them together to construct one (or in fact finitely many) linear or circular orders that are algebraic over $\emptyset$, hence canonical. Intuitively, we start with an order in the family and extend it as much as possible by adjoining other orders in the family that have a convex subset in common. This can never lead to orders \emph{diverging} by the result mentioned above. Hence either we can keep going obtaining a longer and longer linear order, or the construction wraps up on itself, and we obtain a circle. This picture is complicated by the fact that the orders might not actually intersect, but be intertwined. We solve this by first \emph{thickening} every order in our family by adding to it all (convex subsets of) orders that are intertwined with it. Having done that, we do not have to worry about intertwinings any more and the rest of the argument is completely elementary, although rather tedious.

Subsection \ref{sec:analysis} is in some sense the core of the paper. Here we start with an $\omega$-categorical rank 1 primitive unstable NIP structure $M$ and construct a canonical family $W$ composed of finitely many linear and circular orders. Two orders in $W$ are either in order-reversing bijection or independent. Those orders are obtained by applying the procedure described above to all definable families of minimal linear orders. We show that this does indeed lead to only finitely many orders by rank arguments. The structure $W$ is definable in $M^{eq}$ and admits a definable finite-to-one map onto $M$. From now on, we switch our focus from $M$ to $W$ and see $M$ as a quotient of $W$. It remains to classify the possible structures on $W$. In Subsection \ref{sec:skeletal} we introduce the basic structure on $W$ that is given by its construction which we call the \emph{skeletal structure}. In Subsection \ref{sec:stable structure} we show that any additional structure on $W$ must come from local definable sets. Finally, in Subsection \ref{sec:homogeneous} we use this analysis to prove our main theorems.

\subsection{Preliminary statements}\label{sec:no binary}

The following proposition shows that certain binary functions are essentially unary.

\begin{prop}\label{prop:no binary functions} Assume that $M$ has finite rank and is NIP. Let $a,b$ be two finite tuples and set $p(x,y)=\tp(a,b)$. Assume that either $a \ind b$ or $\rk(a)=1$. Let also $V$ be a 0-definable linear or circular order of topological rank 1 and let $f\colon p(M)\to \overline V$ be a 0-definable function. Then  $f(a,b)\in \acl^{eq}(a)\cup \acl^{eq}(b)$.
\end{prop}
\begin{proof} If $a\in \acl(b)$, there is nothing to show. If $\rk(a)=1$ and $a\notin \acl(b)$, then $\rk(a/b)\geq 1=\rk(a)$, so by Lemma \ref{lem:independence_acl} $a\ind b$. Hence we can assume $a\ind b$. Let $a_1,\ldots,a_n \in M$ be realizations of $p(x,b)$ so that for each $k$, $\rk(a_i/ba_1\ldots a_{i-1})= \rk(a) (=\rk(a_i))$ (this exists by Lemma \ref{lem:exist maximal rank}). For $i\leq n$, set $c_i=f(a_i,b)$. If $c_i$ is algebraic either over $b$ or over $a_i$, then since $\tp(a_ib) = \tp(ab)$, it follows that $f(a,b)$ is algebraic either over $b$ or over $a$ and we are done. Assume that this is not the case.

\smallskip
\underline{Claim:} $c_i \notin \acl^{eq}(ba_1\ldots a_{i-1})$.

\smallskip
\emph{Proof:} Assume that $c_i \in \acl^{eq}(ba_1\ldots a_{i-1})$. Then $c_i \in \acl^{eq}(ba_i) \cap \acl^{eq}(ba_1\ldots a_{i-1})$. Since $\rk(a_i/ba_1\ldots a_{i-1}) = \rk(a_i) = \rk(a_i/b)$, by Lemma \ref{lem:independence_acl}, we have \[ a_i \ind_b a_1\ldots a_{i-1},\]

and Lemma \ref{lem:independence_acl} implies $c_i \in \acl^{eq}(b)$. Contradiction.

\smallskip
As a consequence of the claim, the $c_i$'s are pairwise distinct (since for $j<i$, $c_j \in \acl^{eq}(ba_1\ldots a_{i-1})$).

Set $\bar a=(a_1,\ldots,a_n)$. By Proposition \ref{prop:basic rank} (5), \[\rk(\bar ab) = \rk(b) + \rk(a_1 / b) + \cdots + \rk(a_n/ba_1\ldots a_{n-1}) = \rk(b) + n \rk(a).\]

Hence for any $i$,  \begin{align*}\rk(\bar a/a_i b) = \rk(\bar ab) - \rk(a_i b) &= \rk(b) + n\rk(a) - (\rk(a_i/b) + \rk(b))\\ &= (n-1) \rk(a).\end{align*}
We also have $\rk(\bar a/a_i) = \rk(\bar a) -\rk(a_i) =  (n-1) \rk(a)$ so that we have $\bar a \ind_{a_i} b$. As $c_i \in \acl^{eq}(a_i b) \setminus \acl^{eq}(a_i)$ we deduce from Lemma \ref{lem:independence_acl} again that $c_i \notin \acl^{eq}(\bar a)$.


Let $Z^0_i$ be the subset of $V$ defined by $\tp(c_i/\acl^{eq}(\bar a))$ and let $Z_i$ be the topological closure of $Z^0_i$ in $V$. Note that $Z^0_i$, and hence $Z_i$, is infinite as $c_i$ is not algebraic over $\bar a$. Then by Lemma \ref{lem1} and the remark after it, $Z_i$ is a convex subset of $V$. The set $Z_i$ with the induced order is minimal over $\bar a$ as it has topological rank 1 and is the closure of a transitive set. For $i,j\leq n$, the subsets $Z_i$ and $Z_j$ are either equal or disjoint (since each is the closure of a complete type over $\bar a$). If they are disjoint, then they are independent by Lemma \ref{lem:no self-intertwining}. Taking a larger value of $n$ and restricting to a subtuple of $(c_1,\ldots,c_n)$, we may assume that they are either all equal or pairwise disjoint.

Assume that the $Z_i$'s are pairwise disjoint and that $V$ is linear. Let $X$ be the topological closure of the set of realizations of $q:=\tp(c_1,\ldots,c_n/\bar a)$ in $Z_1 \times \cdots \times Z_n$. By Proposition \ref{lem6}, $X$ is a finite union of products of closed subsets of each $V_i$. Since $X$ is $\bar a$-definable and each $Z_i$ is minimal over $\bar a$, the only non-empty $\bar a$-definable closed subset of $Z_i$ is $Z_i$ itself, hence it must be that $X = \prod_{i\leq n} Z_i$. It follows that for any subset $I\subseteq \omega$, we can find $(c'_1,\ldots,c'_n)\models q$ with \[c'_i < c_i \iff i\in I.\] Take $b'$ so that \[\tp(c'_1,\ldots,c'_n,b'/\bar a)=\tp(c_1,\ldots,c_n,b/\bar a).\] We then have \[f(a_i,b')<c_i \iff i\in I.\] As $n$ was arbitrary, the formula \[\phi(xx' ;y) \equiv f(x,y) < x'\] has the independence property. The argument in the circular case is similar using Proposition \ref{prop:n circular orders} instead of Proposition \ref{lem6}.

Assume finally that the $Z_i$'s are all equal to some $Z$. In the linear case, the argument is exactly the same as in the previous paragraph, using Proposition \ref{lem2} to find the $c'_i$'s. In the circular case we use Proposition \ref{prop:one circular order} instead.
\end{proof}

Recall from the beginning of Section \ref{sec:intertwinings} that we say that $\overline V$ is definable if there is a finite set of formulas $\Phi$ such that every cut of $V$ can be defined by an instance of some $\phi(x;y)\in \Phi$.

\begin{cor}\label{cor:trivial geometry}
Assume that $M$ has rank 1 and is NIP. Let $(V,\leq)$ be a minimal 0-definable linear order. Let $\overline V(a)$ denote $\acl^{eq}(a)\cap \overline V=\dcl^{eq}(a)\cap \overline V$\footnote{Algebraic closure and definable closure are equal on a linear order.}. Then:

\begin{enumerate}
\item for any $a_0,\ldots, a_{n-1}\in M$, we have $\overline V(a_0,\ldots,a_{n-1})=\bigcup_{i<n} \overline V(a_i)$;

\item $\overline V$ is definable, minimal and has rank 1.

\end{enumerate}
\end{cor}
\begin{proof}
(1) Let $c\in \overline V(a_0,\ldots,a_{n-1})$ and set $p=\tp(a_0,a_1\ldots a_{n-1})$. Then by definition of $ \overline V(a_0,\ldots,a_{n-1})$, there is some 0-definable function $f$ defined on realizations of $p$, such that $f(a_0,a_1\ldots a_{n-1})=c$. By Proposition \ref{prop:no binary functions}, $c\in \dcl^{eq}(a_0) \cup \dcl^{eq}(a_1\ldots a_{n-1})$. We conclude by induction on $n$.

  (2) Let $a$ be a singleton, then $\overline V(a)$ is finite, by Lemma \ref{lem:finite dcl}. Hence there is a finite set of formulas $\Phi$ such that every element $c$ of $\overline V(a)$ is definable by $\phi(x;a)$ for some $\phi(x;y)\in \Phi$. Since $M$ has rank 1 and $c$ is definable over $a$, $\rk(c)\leq \rk(a) \leq 1$. By (1), every element of $\overline V$ is of this form, hence $\overline V$ is definable and has rank 1. Finally, $\overline V$ is minimal since $V\subseteq \overline V$ is minimal and dense inside it (Lemma \ref{lem:preservation under dense subsets} (3)).
\end{proof}

In what follows, we will consider a definable family $(V_a,\leq_a)$, $a\in D$ of linear orders, by which we mean that $D$ is a definable set and there are formulas $\phi(x;t)$ and $\psi(x,y;t)$ such that for any $a\in D$, the formula $\psi(x,y;a)$ defines a linear order denoted $\leq_a$ on $V_a := \phi(M;a)$.

Recall from Section \ref{sec:weakly minimal} that $I \wmc V$ means that $I$ is a weakly minimal definable convex subset of $V$.

\begin{prop}\label{prop:unique continuation of orders} Let $D$ be a 0-definable set and let $(V_u,\leq_u)$, $u\in D$, be a definable family of linearly ordered sets, with $V_u$ minimal over $u$. Assume that $D$ is ranked and let $a,b\in D$. Let $I\wmc V_a$ be intertwined with some $J\wmc V_b$. Let $h:I \to \overline{J}$ be the intertwining map and take $t\in I$. Then the following two statements hold:
\begin{itemize}
\item Either $\{x\in V_a : x<_{a} t\}$ is intertwined with a convex subset of $V_b$ or $\{x\in V_b : x<_b h(t)\}$ is intertwined with a convex subset of $V_a$.

\item Either $\{x\in V_a : x>_{a} t\}$ is intertwined with a convex subset of $V_b$ or $\{x\in V_b : x>_b h(t)\}$ is intertwined with a convex subset of $V_a$.
\end{itemize}
\end{prop}
\begin{proof}
By reversing the orders, it is enough to prove the second statement. We will drop the indices in the linear orders $\leq_a,\leq_b,...$ when they are implied by the context.

For any $c,d\in D$ and $u\in \overline{V_c}$, consider the definable set $C[c,d,u]\subseteq V_{c}$ defined as:
\[ v\in C[c,d,u] \text{ if } \{x\in V_c : u < x < v\}\text{ is intertwined with some }W\wmc V_d.\]

If $C[c,d,u]$ is non-empty then it is an initial segment of $\{x\in V_c:u<x\}$. In that case define \[f_{c,u}(d) = \sup C[c,d,u] \in \overline{V_c}\cup\{+\infty\}.\] If $C[c,d,u]$ is empty, then $f_{c,u}(d)$ is undefined.

\usetikzlibrary{backgrounds}
\vspace*{5pt}
\begin{tikzpicture}[scale = 1]
	\draw[dotted] (1,1) -- (1.7,1) ;
	\draw (1.7,1)  -- (10,1) node [above] {$V_c$} ;
	\draw  (5.5,1.1) node [above] {$f_{c,u}(d)$} -- (5.5,0.9);
	\draw (1.7,1.1) node [above]{$u$} -- (1.7,0.9); 
	
	\draw[dotted] (1,0.8) -- (1.7,0.8) ;
	\draw (1.7,0.8) -- (5.5,0.8) -- (10,0) node [above] {$V_d$};

\end{tikzpicture}

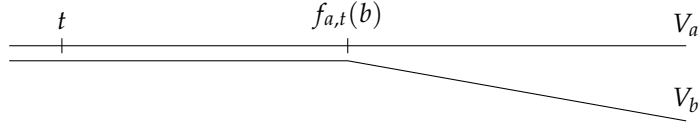
\captionof{figure}{Definition of $f_{c,u}(d)$. Intertwined intervals are represented as parallel lines.}

\smallskip
\underline{Claim A}: If $f_{c,u}(d)$ is defined and different from $+\infty$, then $f_{c,u}(d)\in \dcl^{eq}(cd)$.

\emph{Proof}: Working over $cd$, the two linear orders $V_c$ and $V_d$ are weakly minimal and can be each written as a disjoint union of $cd$-definable points and $cd$-definable minimal convex subsets (as explained at the beginning of Section \ref{sec:weakly minimal}). Any two of those $cd$-definable minimal orders are either intertwined or independent. Hence $f_{c,u}(d)$, if defined and different from $+\infty$, is the supremum of one of those minimal orders. Therefore $f_{c,u}(d)\in \dcl^{eq}(cd)$.

\smallskip
Note also:
\[\odot \qquad \text{ If }u<u'<f_{c,u}(d)\text{, then }f_{c,u}(d)=f_{c,u'}(d).\]

\underline{Claim B}: Let $c,d,e\in D$ and $u\in V_c$ such that $C[c,d,u]\subseteq V_c$ is non-empty and bounded above. Let $I \wmc V_d$ be intertwined with $C[c,d,u]$. Take $v,w\in V_d$ so that \[v< \inf I < \sup I < w.\]
Assume that $\{x\in V_d : v < x< w\}$ is intertwined with a convex subset of $V_e$. Then $C[c,e,u] = C[c,d,u]$.

\smallskip
\emph{Proof}: Write $C = C[c,d,u]$, hence $C$ is intertwined with $I \wmc V_d$. By assumption $C[d,e,v] \supset I$. Let $J\wmc V_e$ be intertwined with $C[d,e,v]$. By transitivity of intertwining, $C$ is also intertwined with a convex subset $J_0$ of $J$. Hence $C\subseteq C[c,e,u]$ by definition of $C[c,e,u]$. Assume that equality does not hold. Then for some $u' > f_{c,u}(d)$ in $V_c$ the subset $C' := \{x\in V_c : u <x <u'\}$ is intertwined with a convex subset $J'$ of $V_e$. We have $J_0 \subseteq J'$ and $J'$ extends $J_0$ to the right. Taking a smaller $u'$ if necessary, we may assume that $J' \subseteq J$. But then by transitivity of intertwining, $C'$ is intertwined with a convex subset of $V_d$. Therefore $f_{c,u}(d) \geq u'$; contradiction.

\vspace*{5pt}
\begin{tikzpicture}[scale = 1]
	\draw[dotted] (1,4) -- (2.3,4);
	\draw (2.3,4)  -- (10,4) node [above] {$V_{c}$} ;
	\draw[dotted] (1,3.8) -- (1.8,3.8);
	\draw (1.8,3.8) -- (4,3.8) -- (10,2.3) node [above] {$V_{d}$} ;
	\draw[dotted] (1,3.6) -- (1.8,3.6);
	\draw (1.8,3.6) -- (4,3.6) -- (6,3.1) -- (10, 1) node [above] {$V_{e}$} ;

	\draw (4,4.1) node [above] {$f_{c,u}(d)$} -- (4,3.9);
	
	\draw (2.3,4.1) node [above]{$u$} -- (2.3,3.9); 
	\draw (1.8,3.9) node [above]{$v$} -- (1.8,3.7); 
	\draw (6, 3.2) -- (6, 3.4) node [above] {$w$};
\end{tikzpicture}
\captionof{figure}{Claim B.}

\medskip
Coming back to the original $a,b$ and $t$ given by the statement of the proposition, the set $C[a,b,t]$ contains all elements of $I$ that are greater than $t$, hence it is non-empty. If $f_{a,t}(b) = +\infty$, then $\{x\in V_a : x>t\}$ is intertwined with a convex subset of $V_b$ and we are done. Assume this is not the case, so $f_{a,t}(b)\in \overline{V_a}$. Similarly, if $f_{b,h(t)}(a)=+\infty$, then we are done, so we may assume that $f_{b,h(t)}(a) \in \overline{V_b}$.

By Claim A, $f_{a,t}(b)$ is equal to one of the finitely many $ab$-definable elements of $\overline{V_a}$. List those elements as $s_1 < \ldots < s_n$ and say that $f_{a,t}(b)=s_k$. It follows that the convex subset $s_{k-1}<x<s_k$ of $V_a$ is intertwined with a convex subset of $V_b$ (if $k=1$, set $s_{k-1}=-\infty$). By minimality of $V_a$ over $a$ and Proposition \ref{lem2} (along with Remark \ref{rem:lem2 in completion}), for any $u\in V_a$, $u > t$, we can find $b' \equiv_{a} b$ such that the $k$-th $ab'$-definable element of $\overline{V_a}$ is greater than $u$ and, if $k>1$, the $k-1$-th is smaller than $t$. Since $\tp(b'/a)=\tp(b/a)$, the convex subset of $V_a$ between $t$ and $u$ is intertwined with a convex subset of $V_{b'}$. Hence $f_{a,t}(b') > u$. This shows that the image of $f_{a,t}$ is dense in the final segment $x>t$ of $V_a$.

We now construct inductively $(b_i:i<\omega)$ in $D$ and points $(u_i:i<\omega)$, $u_i\in V_{b_i}$. We start by setting $(b_{-1},u_{-1})=(a,t)$. Set also $b_0 = b$ and take $u_0\in J$ to be smaller than $h(t)$. Note that $f_{b_0,u_0}(b_{-1}) > b_0$ since $b_0$ lies in $J$ which is intertwined with an interval of $V_{b_{-1}}$. We will use the notation $V_k := V_{b_k}$.

Having constructed $(b_k,u_k)$, let $(u'_k,w_k)\in V_{k}\times \overline{V_{k}}$ be such that \[ (1) \qquad u'_k < u_k < f_{b_k,u_k}(b_{k-1}) < w_k\] and \[ \tp(u'_k,w_k,b_k) = \tp(u_k,f_{b_k,u_k}(b_{k-1}),b_k).\]
This is possible by Lemma \ref{lem2}. Next, let $b_{k+1}$ be such that \[ \tp(u'_k,w_k,b_k, b_{k+1}) = \tp(u_k,f_{b_k,u_k}(b_{k-1}),b_k, b_{k-1}).\]

We then have \[ f_{b_k,u'_{k}}(b_{k+1}) = w_k,\]

and then by (1) and $\odot$, \[ f_{b_k,u_k}(b_{k+1}) = w_k > f_{b_k,u_k}(b_{k-1}).\]

The interval $\{x\in V_{k} : u'_k < x < w_k\}$ is intertwined with a convex subset $J_{k+1}$ of $V_{k+1}$ via a function $h_k$. Pick $u_{k+1} \in J_{k+1}$ smaller than $h(u_k)$.

\vspace*{5pt}
\begin{tikzpicture}[scale = 1]
	\draw[dotted] (1,4) -- (2.9,4);
	\draw (2.9,4)  -- (10,4) node [above] {$V_{b_0}$} ;
	\draw[dotted] (1,3.8) -- (2.5,3.8);
	\draw (2.5,3.8) -- (4,3.8) -- (10,2.3) node [above] {$V_{b_1}$} ;
	\draw[dotted] (1,3.6) -- (2.1,3.6);
	\draw (2.1,3.6) -- (4,3.6) -- (6,3.1) -- (10, 1) node [above] {$V_{b_2}$} ;
	\draw[dotted] (1,3.4) -- (1.7,3.4);
	\draw (1.7,3.4) -- (4,3.4) -- (6,2.9) -- (8,1.85) -- (10, 0) node [above] {$V_{b_3}$} ;

	\draw (4,4.1) node [above] {$w_0$} -- (4,3.9);
	\draw (6,3.4) node [above] {$w_1$} -- (6,3.2);
	\draw (8.1,2.1) node [above] {$w_2$} -- (8.1,1.9);
	
	\draw (2.9,4.1) node [above]{$u_0$} -- (2.9,3.9); 
	\draw (2.5,3.9) node [above]{$u_1$} -- (2.5,3.7); 
	\draw (2.1,3.7) node [above]{$u_2$} -- (2.1,3.5); 
	\draw (1.7,3.5) node [above]{$u_3$} -- (1.7,3.3); 

\end{tikzpicture}
\captionof{figure}{Construction of $(b_k,u_k)$.}

\smallskip
\underline{Claim C}: For $l<k+1$, we have $f_{b_l,u_l}(b_{k+1}) = w_l$.

\smallskip
\emph{Proof}: We show this by decreasing induction on $l$. We already know this for $l=k$. Assume we know it for $l+1$. By construction, $C[b_l,b_{l+1},u_l]$ has supremum $w_l$ and is intertwined with some $J$ in $V_{l+1}$ such that $u_{l+1} < \inf J < \sup J < w_{l+1}$. Furthermore, by induction hypothesis the interval $\{x\in V_{l+1} : u_{l+1} < x < w_{l+1}\}$ is intertwined with a convex subset of $V_{k+1}$. Hence by Claim B, $C[b_l,b_{k+1}, u_l] = C[b_l,b_{l+1},u_l]$ hence $f_{b_l,u_l}(b_{k+1}) = f_{b_l,u_l}(b_{l+1}) = w_l$.

\medskip
For $c\in D$ and $u\in V_c$, define the equivalence relation $E_{c,u}(x,y)$ on $D$ by $f_{c,u}(x)=f_{c,u}(y)$. From Claim C, we deduce that for $k,k'>l$, $b_k$ and $b_{k'}$ are $E_{b_l,u_l}$-equivalent. In particular the $E_{b_l,u_l}$-class of $b_{l+1}$ is infinite. At each stage $l$ of the construction, we have infinitely many choices for $w_l$ (since the only condition on it is that it is larger than $f_{b_k,u_k}(b_{k-1})$), hence for the $E_{b_l,u_l}$-class of $b_{l+1}$. We can make a choice that is not algebraic over all the elements considered so far. Let $D_1$ be the $E_{b_0,u_0}$-class of $b_1$. As that class is not algebraic over $b_0u_0$, we have $\rk(D_1)<\rk(D)$ by Lemma \ref{lem:equivalence relation rank}. Next, $D_1$ is split into infinitely many $E_{b_1,u_1}$-classes. The $E_{b_1,u_1}$-class of $b_2$, say $D_2$, is not algebraic over $b_0u_0b_1u_1$, hence $\rk(D_2)<\rk(D_1)$. Continuing in this way, we obtain an infinite sequence of definable sets of decreasing ranks, which is absurd.
\end{proof}

\subsection{Gluing definable orders}\label{sec:gluing}

Let $(V_a,\leq_a)$, $a\in D$, be a $0$-definable family of linearly ordered sets, with $V_a$ minimal over $a$. Assume that $D$ is ranked. Our goal in this section is to glue the orders $V_a$ together as much as possible along definable intertwinings between subintervals so as to construct a 0-definable family of pairwise independent orders.

Here is a useful example the reader might want to keep in mind while reading this section.

\begin{ex}\label{ex:gluing}
Let the structure $M$ be equipped with a linear order $\leq$, an equivalence relation $E_0$ and no additional structure (a similar construction would work for a circular order). There are infinitely many $E_0$-classes and each one is dense. Let $D$ be the set of pairs $(a_0,a_1)\in M^2$ such that $a_0<a_1$. For $\bar a = (a_0,a_1)\in D$, define \[V_{\bar a}=\{x\in M \colon a_0 <x <a_1 \wedge x\rel{E_0}a_0\}.\] The set $V_{\bar a}$ equipped with the induced order is minimal over $\bar a$. It is a dense subset of a convex subset of the original order $(M,\leq)$. In this situation, the construction presented in this section will essentially reconstruct the order $(M,\leq)$ from its pieces $(V_{\bar a})_{\bar a\in D}$.

For a slightly more complicated example which shows the need for the equivalence relation $E$ in the theorem below, let $N$ be a structure equipped with an equivalence relation $E$ and such that each $E$-class is a copy of the structure $M$ above and there are no extra relations between the classes. Here the family $(V_{a})_{a\in D}$ given by the theorem is the union of the families constructed as above in each class. The theorem will recover $E$ and the linear order on each $E$-class.
\end{ex}

\begin{thm}\label{th:gluing}
	Let $D$ be a 0-definable subset of $M^{eq}$ which is ranked and let $(V_a,\leq_a)_{a\in D}$ be a definable family of linearly ordered sets, with $V_a$ minimal over $a$. Then there is a 0-definable set $W$ in $M^{eq}$ and a 0-definable equivalence relation $E$ on $W$ such that:
	
	\begin{itemize}
		\item for $e\in W/E$, let $W[e]\subseteq W$ be the $E$-class corresponding to $e$, then $W[e]$ admits either a linear or circular $e$-definable order;
		\item for $e\in W/E$, $W[e]$ equipped with that order is minimal over $e$;
		\item for any $e\neq e' \in W/E$, the orders $W[e]$ and $W[e']$ are either independent or in definable order-reversing bijection;
		\item for any $a\in D$, there is (a necessarily unique) $e\in W/E$ such that $V_a$ admits an order-preserving injection into $W[e]$.
	\end{itemize}
	\end{thm}


\smallskip
Note that the conclusion becomes stronger if we replace the given family $(V_a,\leq_a)_{a\in D}$ by a larger family $(V'_a,\leq_a)_{a\in D'}$ for some $D' \supseteq D$, with $V'_a = V_a$ for $a\in D$. Having noticed this, we start by increasing the family so that the following property holds:

\begin{itemize}

\item[$(\triangle)$] For any $a\in D$, there is $\tilde a\in D$ such that $V_a$ is in definable order-reversing bijection with $V_{\tilde a}$.
\end{itemize}


To achieve this, we replace $D$ by $D' = D \times \{*_1,*_2\}$, where $*_1, *_2$ are two elements of $\dcl^{eq}(\emptyset)$\footnote{Whereas $\dcl(\emptyset)$ can very well be empty, $\dcl^{eq}(\emptyset)$ is always infinite as it contains the quotient of each $M^k$ by the equivalence relation with a unique class.
} and let $V_{(a,*_1)}$ be $V_a$ and $V_{(a,*_2)}$ be the reverse of $V_a$ (which is also minimal over $a$, hence over $(a,*_2)$).

\medskip
The next stage of the construction involves \emph{thickening} the $V_a$'s by replacing each one by a large enough definable subset of $\overline{V_a}$, which will be called $W_a$. For instance in Example \ref{ex:gluing}, $W_{\bar a}$ would be the convex hull of $V_{\bar a}$ inside $M$.

To this end, we first define an equivalence relation $\sim$ on pairs $(a,t)$, $a\in D, t\in {V_a}$. The intuition is that if $a,b\in D$, $t\in V_a$ and $u\in V_b$, then $(a,t)$ and $(b,u)$ are equivalent if they should be glued together. In Example \ref{ex:gluing}, we will have $(\bar a,t)\sim (\bar b,u)$ if and only if $t=u$.

Given $a,b\in D$, $t\in V_a$ and $u\in V_b$, we say that $(a,t)\sim (b,u)$ if there exist parameter-definable subsets $t\in U_a\wmc V_a$ and $u\in U_b \wmc V_b$ which are intertwined and such that the unique intertwining $f:\overline{U_a} \to \overline{U_b}$ sends $t$ to $u$.

\begin{lemme}\label{lem:beginning of proof}
The relation $\sim$ is an equivalence relation on the set of pairs $(a,t)$, $a\in D, t\in V_a$.
\end{lemme}
\begin{proof}
Reflexivity and symmetry are clear. Let us show transitivity. Let $a,b,c\in D$, $t\in V_a$, $u\in V_b$, $w\in V_c$ and assume $(a,t)\sim (b,u)$ and $(b,u)\sim (c,w)$. Let $A$ be a set of parameters large enough to define all the relevant convex subsets of $V_a$, $V_b$ and $V_c$. Working over $A$, there are $U_a\wmc V_a$, $U_b \wmc V_b$ containing $t$ and $u$ respectively which are intertwined by $f:\overline{U_a} \to \overline{U_b}$, with $f(t)=u$. Similarly, there are $U'_b \wmc V_b$ and $U_c \wmc V_c$ intertwined by $g: \overline{U'_b}\to \overline{U_c}$ with $g(u)=w$. Let $U''_b = U_b \cap U'_b$. Then $U''_b$ is an $A$-definable convex subset of $V_b$ which is weakly minimal and contains $u$. Let $U''_c$ be the convex subset of $V_c$ intertwined with $U''_b$ via $g$, that is:
$$U''_c = g(\overline{U''_b})\cap V_c.$$

And define in the same way:
$$U''_a = f^{-1}(\overline {U''_b})\cap V_a.$$

Then $U''_a$ and $U''_c$ are weakly minimal and both are intertwined with $U''_b$. By transitivity of intertwining, $U''_a$ and $U''_c$ are intertwined: in fact the intertwining map is $(g|_{U''_b})\circ (f|_{U''_a})$. That maps sends $t$ to $w$, hence we have $(a,t)\sim (c,w)$.
\end{proof}

Note that for $t,u\in V_a$ distinct, we have $(a,t)\nsim (a,u)$ since by Lemma \ref{lem:no self-intertwining} there can be no intertwining between two disjoint convex subsets of $V_a$.

Let $[a,t]_{\sim}$ denote the $\sim$-class of $(a,t)$ and let $W$ be the set of $\sim$-classes. For $a\in D$, define $W_a$ as: \begin{align*}W_a &= \{[b,t]_{\sim}\colon \text{some weakly minimal convex subset of }V_{b} \text{ containing }t\\ & \qquad \text{ is intertwined with a weakly minimal convex subset of }V_a\}.\end{align*}

In this definition, the weakly minimal convex subsets of $V_a$ and $V_b$ are allowed to be defined over any set of parameters. In fact, if they exist, they can always be taken to be defined over $ab$, but we will not use that fact.

\medskip
Some observations:
\begin{itemize}
\item[$\boxplus_0$] $W_a$ is well defined: the definition does not depend on the representative of $[b,t]_{\sim}$.
\end{itemize}

To see this, assume that $(b,t)\sim (c,u)$ and that some weakly minimal convex subset $t\in U_b \subseteq V_b$ is intertwined with some $U_a\wmc V_a$. By definition of $\sim$, there are weakly minimal subsets $t\in U'_b \wmc V_b$, $u\in U'_c \wmc V_c$ which are intertwined via $f:\overline{U'_b}\to \overline{U'_c}$ sending $t$ to $u$. Define $U''_b = U_b \cap U'_b$ and let $U''_c$ be the convex subset of $U'_c$ corresponding to $U''_b$, that is $U''_c = f(\overline{U''_b})\cap V_c$. Then $U''_c$ is weakly minimal, contains $u$ and is intertwined with a convex subset of $V_a$.

Note that $W_a$ is invariant under automorphisms fixing $a$ and hence since $M$ is $\omega$-categorical, it is definable over $a$.

\begin{itemize}
\item[$\boxplus_1$] For $a\in D$ and $t\in V_a$, $[a,t]_{\sim} \in W_a$.
\end{itemize}
Indeed, in the definition of $W_a$ we can take $V_a$ itself as weakly minimal convex subset of $V_a$. Note that $W_a$ is in general larger than $\{[a,t]_{\sim} : t\in V_a\}$ since we are not assuming in the definition that the intertwining sends $t$ to a point in $V_a$ (it could be sent to a point in $\overline{V_a}\setminus V_a$).

\begin{itemize}
\item[$\boxplus_2$] There is an $a$-definable injective map $\iota_a \colon W_a \to \overline{V_a}$.
\end{itemize} 

To construct this map, let $[b,t]_{\sim} = [c,u]_{\sim} \in W_a$. By definition of $\sim$, there are $t\in U_b \wmc V_b$ and $u\in U_c \wmc V_c$ which are intertwined by some $f:\overline{U_b} \to \overline{U_c}$ sending $t$ to $u$. By definition of $W_a$ and $\boxplus_0$, up to replacing $U_b$ and $U_c$ by smaller neighborhoods if necessary, $U_b$ and $U_c$ are each intertwined with convex subsets of $V_a$, via maps $g:\overline{U_b} \to \overline{T_b}$ and $h:\overline{U_c} \to \overline{T_c}$, where $T_b$ and $T_c$ are weakly minimal convex subsets of $V_a$. Then $T_b$ and $T_c$ are intertwined by $h\circ f \circ g^{-1}$. By Corollary \ref{cor:no self-intertwining}, we must have $T_b = T_c$ and $h\circ f \circ g^{-1}$ is the identity map. Since $f$ maps $t$ to $u$, $g(t) = h(u)$. We then define $\iota_a ([b,t]_{\sim}) = g(t)\in \overline{V_a}$. We have proved that $\iota_a$ is well defined. It is $a$-definable since it is invariant under automorphisms fixing $a$ and we are working in an $\omega$-categorical structure.

To see that $\iota_a$ is injective, assume that $\iota_a([b,t]_{\sim})= \iota_a([d,v]_{\sim})$. Then there are $t\in U_b \wmc V_b$ and $v\in U_d \wmc V_d$ intertwined with convex subsets of $V_a$ via maps $f: \overline{U_b} \to \overline{V_a}$ and $g:\overline{U_d} \to \overline{V_a}$ with $f(t)=g(v)$. Restricting $U_b$ and $U_d$, we may assume that $f$ and $g$ have the same image $\overline{U_a}$, where $U_a\wmc V_a$. But then $g^{-1}\circ f$ intertwines $U_b$ and $U_d$ and sends $t$ to $v$. Hence $[b,t]_{\sim} = [d,v]_{\sim}$.

\medskip
It follows from $\boxplus_2$ and $\boxplus_1$ that $W_a$ is in definable bijection with an $a$-definable dense subset of $\overline {V_a}$. We equip $W_a$ with the order $\leq_a$ inherited from this bijection. The linear order $(W_a,\leq_a)$ is then minimal over $a$ (by Lemma \ref{lem:preservation under dense subsets} applied to a suitable definable subset of $\overline{V_a}$).

\medskip
Recall that $W$ is the set of $\sim$ classes and that each $W_a$ is a subset of $W$.

\begin{itemize}
\item[$\boxplus_3$] Let $a,b\in D$ and let $U_a \wmc W_a$ and $U_b \wmc W_b$. If $U_a$ and $U_b$ are intertwined, then $U_a = U_b$ and the orders induced on that set by $W_a$ and $W_b$ coincide.
\end{itemize} 

Let $[c,t]_{\sim} \in U_a$. Then by definition of $W_a$, there is $t\in U_c \wmc V_c$ such that $U_c$ is intertwined with a convex subset $U'_a\subseteq \overline{V_a}$. By construction, $\iota_a$ sends $[c,t]_{\sim} \in U_a$ to a point in $U'_a$. Hence $U_c$ is intertwined with a convex subset of $U_a$. By transitivity of intertwining, $U_c$ is intertwined with a convex subset of $U_b$ (hence with a convex subset of $V_b$). So $[c,t]_{\sim} \in U_b$ by construction. Let $h$ be the intertwining map from $U_a$ to $\overline{U_b}$. There are convex neighborhoods of $[c,t]_{\sim}$ and $h([c,t]_{\sim})$ in $\overline{U_b}$ that are intertwined with neighborhoods of $t$ in $U_c$ (by definition of $[c,t]_{\sim}$). Hence some convex neighborhoods of  $[c,t]_{\sim}$ and $h([c,t]_{\sim})$ in $\overline{U_b}$ are intertwined. By Corollary \ref{cor:no self-intertwining} applied to $W_a$, that intertwining must be the identity and thus $h$ sends $[c,t]_{\sim}$ to itself.

By symmetry of the roles of $a$ and $b$, this shows that $U_a = U_b$ and the intertwining map is the identity.


\medskip
We know that for every $x\in W$, there is a (non-unique) subset $x\in W_a \subseteq W$ which is linearly ordered. We can think of $W$ as a kind of ``{1-dimensional manifold}" and $W_a$ as a neighborhood of $x$ ``{homeomorphic}" to a line. We have to check now that the $W_a$'s do indeed function like neighborhoods, in particular, that their intersections are well behaved and that the orders coincide on them. This will be achieved by Lemma \ref{lem:good position}. We will then be able to mimic the classical proof that a 1-dimension manifold is a disjoint union of lines and circles.

\begin{defi}
Let $V_0,V_1 \subseteq W$ be two parameter-definable subsets equipped with parameter-definable linear orders $\leq_0$ and $\leq_1$ respectively. We say that $(V_0,\leq_0)$ and $(V_1,\leq_1)$ are \emph{in good position} (or that the pair $(V_0,V_1)$ is in good position) if the following two conditions hold:
\begin{itemize}
\item for any $t\in V_0 \cap V_1$, there is $\eta\in \{0,1\}$ such that $\{x\in V_\eta : x\leq_{\eta} t\}$ is a convex subset of $V_{1-\eta}$ and the two orders $\leq_0$ and $\leq_{1}$ coincide on it.

\item for any $t\in V_0 \cap V_1$, there is $\eta\in \{0,1\}$ such that $\{x\in V_\eta : x\geq_{\eta} t\}$ is a convex subset of $V_{1-\eta}$ and the two orders $\leq_0$ and $\leq_{1}$ coincide on it.
\end{itemize}
\end{defi}

\begin{lemme}\label{good position cases}
Two parameter-definable $V_0,V_1\subseteq W$ equipped with linear orders as in the previous definition are in good position if and only if one of the following holds:

\begin{enumerate}
	\item $V_0\cap V_1=\emptyset$;
	\item for some $\eta\in \{0,1\}$, $V_\eta$ is a convex subset of $V_{1-\eta}$ and the two orders coincide on $V_\eta$;
	\item for some $\eta\in \{0,1\}$, $V_0 \cap V_1$ is an initial segment of $V_\eta$ and a final segment of $V_{1-\eta}$ and the two orders coincide on it;
	\item $V_0\cap V_1$ can be written as a disjoint union $I \sqcup J$, where $I$ is an initial segment of $V_0$ and a final segment of $V_1$, $J$ is a final segment of $V_0$ and an initial segment of $V_1$ and the two orders coincide on each of $I$ and $J$.
\end{enumerate}
\end{lemme}
\begin{proof}
It is clear that any one of those conditions imply that $V_0$ and $V_1$ are in good position.

Assume $V_0$ and $V_1$ are in good position. If $V_0 \cap V_1$ is empty, then item (1) holds. Assume it is not and let $t\in V_0\cap V_1$. Without loss of generality, the initial segment $\{x\in V_0: x\leq t\}$ is a convex subset of $V_1$. Let $I\subseteq V_0$ be the maximal initial segment of $V_0$ which is a convex subset of $V_1$ and on which the two orders $\leq_0$ and $\leq_1$ coincide. This is a definable subset of $V_0$ which contains $t$. If $I=V_0$, the item (2) holds and we are done. If $\{x\in V_0 : x>_0 t\}$ is a convex subset of $V_1$ with the same induced order, then the whole of $V_0$ is such and $I=V_0$. This has been ruled out so it must be that $\{x\in V_1 : x>_1 t\}$ is a convex subset of $V_0$ with the same induced order. This set is therefore equal to $\{x\in I : x>_0 t\}$ by definition of $I$, and we have that $I$ is an initial segment of $V_0$ and a final segment of $V_1$. If $I=V_0 \cap V_1$, then item (3) holds and we are done. If this is not the case, we show that (4) holds. To this end, take $s\in V_0 \cap V_1 \setminus I$. Then $\{x \in V_1 : x\geq_1 s\}$ contains $I$ as a proper subset, hence cannot be a convex subset of $V_0$ with the same induced order since $I$ is an initial subset of $V_0$. Therefore it must be that $\{x\in V_0 : x\geq_0 s\}$ is a convex subset of $V_1$ with the same induced order. Define $J$ as the maximal final segment of $V_0$ which is a convex subset of $V_1$ and on which the two orders coincide. By the same reasoning as above, $J$ is an initial segment of $V_1$.

It remains to show that $V_0\cap V_1 = I \cup J$. If not, let $u\in V_0\cap V_1 \setminus (I\cup J)$. Then by the same reasoning as for $s$, we cannot have that $\{x \in V_1 : x\geq_1 u\}$ is a convex subset of $V_0$ (since $I$ is an initial segment of $V_0$ which is properly contained in it), but by the same argument with $J$ instead of $I$, $\{x \in V_0 : x\geq_0 u\}$ cannot be a convex subset of $V_1$. This contradicts the definition of good position.
\end{proof}

Note that cases (2) and (3) are not mutually exclusive: if for instance $V_0$ is a proper initial segment of $V_1$, then they both hold.

If item (3) above holds, we will say that $V_\eta$ is a \emph{simple continuation} of $V_{1-\eta}$. In this case there is a natural definable linear order on $V_0\cup V_1$ which coincides with $\leq_0$ and $\leq_1$ on $V_0$ and $V_1$ respectively and for which every point of $V_\eta\setminus V_{1-\eta}$ is above every point of $V_{1-\eta}$.

If item (4) holds, we will say that $V_0$ and $V_1$ are in \emph{circular position}. Note that in this case, there is a natural definable circular order $C$ on $V_0\cup V_1$ for which $V_0$ and $V_1$ are convex subsets and $\leq_0$ and $\leq_1$ are the linear orders induced by $C$ (this completely characterizes $C$).

At this point we strongly encourage readers to take a moment to convince themselves that from a family of linear orders pairwise in good position, one can glue those orders together and obtain a family of pairwise disjoint linear and circular orders. To achieve this, start say from a linear order $V$ from the family. If some $W$ in the family intersects $V$ non-trivially, add it to it to obtain a longer linear order, or a circular order if case (4) holds. Keep going in this way adding all orders that have non-empty intersection with what you have so far. From the definition of good position, one can see that at any stage of this process, we have either a linear or a circular order (we prove this is details in what follows).

\begin{lemme}\label{lem:circular position 2}
  Let $V_0,V_1,V_2\subseteq W$ be parameter-definable and linearly ordered. Assume that any two of them are in good position and that $(V_0,V_1)$ is in circular position. Then either $V_2$ is disjoint from $V_0\cup V_1$, or it is a convex subset of $V_0\cup V_1$ when the latter is equipped with its canonical circular order.
\end{lemme}
\begin{proof}
  This is rather straightforward by going through all the cases. The reader is  encouraged to make drawings to follow the arguments. Assume that $V_2$ is not disjoint from $V_0\cup V_1$.  Write $V_0\cap V_1 = I \cup J$ as in the definition of circular position, so that $I$ is an initial segment of $V_0$ and $J$ a final segment of it.

  Consider first the case where $V_2$ and $V_0$ are in circular position. Then we can write $V_0\cup V_2 = I' \cup J'$ as in the definition of circular position, so that $I'$ is an initial segment of $V_0$ and $J'$ a final segment of it. As $I$ and $I'$ are two initial segments of $V_0$, one is included in the other. Assume that $I \subseteq I'$; the other case is similar. Write $I' = I \cup K$. It follows that $J = K \cup J'$. Then $V_0 = I \cup K \cup J'$, $V_1 = K \cup J' \cup I$ and $V_2 = J' \cup I \cup K$, where we are writing each union in the order in which the sets appear in the linear order in question. Then $V_1, V_2$ are in circular position as witnessed by $K$ and $J'\cup I$ and $V_2$ is a convex subset of the circular order $V_0 \cup V_1$ (in fact it covers it entirely). Note that in this case the three pairs $(V_0, V_1)$, $(V_1, V_2)$ and $(V_2, V_0)$ are in circular position and define the same circular order on the union.

  The case where $V_2$ and $V_1$ are in circular position is similar, exchanging the roles of $V_0$ and $V_1$ in the arguments above.

  Assume that $V_0$ is included in $V_2$ and is a convex subset of it. Then we have $I<J$ in $V_2$, since this holds in $V_0$. However, we know that $J<I$ holds in $V_1$. Given that $V_1$ and $V_2$ are in good position, this can only happen if $V_1$ and $V_2$ are in circular position, since in all other cases in the definition of good position, the intersection of the two orders is ordered in the same way in both orders. We have already covered this case.

  Assume that $V_2$ is included in $V_0$ and is a convex subset of it. If $V_2$ is included in either $I$ or $J$, then $V_2$ is a convex subset of the circular order $V_0 \cup V_1$ as required (since $I$ and $J$ are such). Otherwise, since $V_2$ is convex in $V_0$, it contains a final segment $I'$ of $I$ and an initial segment $J'$ of $J$. We then have $I' < J'$ in $V_2$, but $J' < I'$ in $V_1$. We conclude as in the previous case.

  The cases where one of $V_1$ or $V_2$ is included in the other as a convex subset are similar by exchanging the roles of $V_0$ and $V_1$ in the two paragraphs above.

  Assume that $V_2$ is a simple continuation of $V_0$. Let $J' = V_0 \cap V_2$, then $J'$ is a final segment of $V_0$ and an initial segment of $V_2$. As $J$ and $J'$ are final segments of $V_0$, one is included in the other. Assume that $J\subsetneq J'$. Write $J' = K \cup J$. Hence $K\subseteq I$. We then have $K<J$ in $V_2$, but $J<K$ in $V_1$ and we conclude as we have done twice before. If $J' = J$, then it is an initial segment of both $V_2$ and $V_1$. As $V_1$ and $V_2$ are in good position, this can only happen if one is included in the other and those cases have been covered. Assume next that $J' \subsetneq J$ and write $J = K\cup J'$. Then $K\cup J'$ is an initial segment of $V_1$ and $J'$ is an initial segment of $V_2$. It thus cannot be the case that $V_1$ is a simple continuation of $V_2$. Given that we have already covered the cases where those two sets are in circular position, or one is included in the other, we can assume that $V_2$ is a simple continuation of $V_1$. So $V_2\cap V_1$ is an end segment of $V_1$. As this intersection contains $J'$ and $J'<I$ in $V_1$, it also contains $I$. But then $I\subseteq V_0 \cap V_2$ and we have $I<J'$ in $V_0$ and $J'<I$ in $V_2$, contradiction to $V_2$ being a simple continuation of $V_0$.

  The case when $V_0$ is a simple continuation of $V_2$ is handled in the same way. With it, we have covered all cases.
\end{proof}

For further reference, we note the following which follows from the proof above.

\begin{lemme}\label{lem:circular position}
Let $V_0,V_1,V_2\subseteq W$ be parameter-definable and linearly ordered. Assume that any two of them are in good position, that $(V_0,V_1)$ is in circular position and that $V_1\subseteq V_2$, with the two orders coinciding on $V_1$. Then $V_0$ and $V_2$ are in circular position and $V_0 \cup V_1 = V_0 \cup V_2$.
\end{lemme}



\begin{lemme}\label{lem:good position}
Let $a,b\in D$. Then $W_a$ and $W_b$ are in good position.
\end{lemme}
\begin{proof}
  This follows from Proposition \ref{prop:unique continuation of orders} along with $\boxplus_3$: If $x \in W_a \cap W_b$, say $x = [c,t]_{\sim}$, then taking a small enough convex subset $t\in K \subseteq V_c$,  there are two convex subsets $x\in I \subseteq W_a$ and $x\in J \subseteq W_b$ which are each intertwined with $K$. Hence $I$ and $J$ are intertwined by transitivity of intertwining. We now apply Proposition \ref{prop:unique continuation of orders}. By $\boxplus_3$, we can replace in the conclusion ``intertwined'' by ``equal'', which precisely gives us that $W_a$ and $W_b$ are in good position.
\end{proof}

We now move to the construction of the $E$-classes given by the theorem. We start by defining inductively families $(W^k_t)_{t\in D_k}$ of linear orders obtained by gluing finitely many of the $W_a$'s together.

Start by letting $D_1 = D$ and $W^1_a = W_a$ for all $a\in D$.

Having defined $(W^k_a)_{a\in D_k}$, let  $D_{k+1} = \{ (a,b) : W^k_b $ is a simple continuation of $W^k_a\}$ and for $(a,b)\in D_{k+1}$, $W^{k+1}_{(a,b)} = W^{k}_a \cup W^k_b$. This set is equipped with its natural linear order extending those of $W^k_a$ and $W^k_b$.

\begin{lemme}\label{lem:exists reverse order}
  For any $k$ and any $a \in D_k$, there is $\tilde a\in D_k$ such that $W^k_{\tilde a}$ is in order-reversing bijection with $W^k_a$.
\end{lemme}
\begin{proof}
  We prove the result by induction on $k$. For $k=1$, take $a\in D$. Let $\tilde a$ be given by property $(\triangle)$. Then $\boxplus_3$ implies that $W_a$ and $W_{\tilde a}$ are in order-reversing bijection. Assume we know the result for $k$ and let $d=(a,b) \in D_{k+1}$. Take $\tilde a, \tilde b$ given by the induction hypothesis. Since $W^k_b$ is a simple continuation of $W^k_a$, $W^k_{\tilde a}$ is a simple continuation of $W^k_{\tilde b}$ and we can take $\tilde d = (\tilde b, \tilde a)$.
\end{proof}

\begin{lemme}\label{lem:good position induction}
For any $k$, if $(a,b), (c,d)\in D_{k+1}$, then $W^{k}_{(a,b)}$ and $W^{k}_{(c,d)}$ are in good position.
\end{lemme}
\begin{proof} We prove the result by induction on $k$. For $k=1$, it follows from Lemma \ref{lem:good position}.

  Assume that $s\in W^{k+1}_{(a,b)}\cap W^{k+1}_{(c,d)}$. We will only prove the first half of the definition of good position since the second half is proved in exactly the same way reversing the orders. Up to exchanging the roles of $(a,b)$ and $(c,d)$, one of the three following cases applies:

\begin{itemize}
\item[(a)] $s\in W^k_a \cap W^k_c$.

By induction, $W^k_a$ and $W^k_c$ are in good position. Without loss of generality, the initial segment $I$ defined by $x\leq s$ in $W^k_a$ is a convex subset of $W^k_c$. But $I$ is also the initial segment defined by $x\leq s$ in $W^{k+1}_{(a,b)}$ and hence is a convex subset of $W^{k+1}_{(c,d)}$.

\item[(b)] $s\in W^k_a \cap W^k_d$.

If the initial segment defined by $x\leq s$ in $W^k_a$ is a convex subset of $W^k_d$, we are done. Otherwise, as $W^k_a$ and $W^k_d$ are in good position, it must be that the initial segment $x\leq s$ in $W^k_d$ is a convex subset of $W^k_a$. But this initial segment contains a point $t$ in $W^k_c$. The interval $[t,s]$ is the same in $W^k_a$ and in $W^k_d$, hence also in $W^k_{(a,b)}$ and in $W^k_{(c,d)}$ and we conclude by applying Case (a) to $t$ instead of $s$.

\item[(c)] $s\in W^k_b \cap W^k_d$.

We do a similar reduction as in the previous case. Assume without loss of generality that the initial segment $x\leq s$ in $W^k_b$ is a convex subset of $W^k_d$. That subset contains a point $t$ in $W^k_a$ and the interval $[t,s]$ is the same in $W^k_{(a,b)}$ and in $W^k_{(c,d)}$. We then conclude by applying case $(b)$ to $t$ instead of $s$.
\end{itemize}
\end{proof}

Note that $W^k_a = W^{k+1}_{(a,a)}$ for any $a\in D_k$.


For $s,t\in W$, write $s\to_k t$ if there is $a\in D_k$ such that $s,t\in W^k_a$ and satisfy $s<t$ there. This relation is irreflexive. By the previous observation, if $s\to_k t$ holds, then so does $s\to_{k+1} t$. By $\omega$-categoricity, there are only finitely many types of elements of $W^2$, hence for only finitely many $k$ do there exist $s$ and $t$ such that $s\to_{k+1} t$ holds, but we do not have $s\to_{k} t$. Therefore there is $n_*<\omega$ such that for any $s,t\in W$, if $s\to_k t$ holds for some $k<\omega$, then it holds for $k=n_*$. We will write $\to$ for $\to_{n_*}$. Define also $D' = D_{n_*}$ and for $a\in D'$, $W'_a = W^{n_*}_a$.

Let \[R_+(s) = \{t\in W : s\to t\},\] \[R_- (s)=  \{t\in W : t\to s\}\] and \[R(s)=R_+(s)\cup R_- (s) \cup \{s\}.\] Note that $s\in R(t)$ if and only if $t\in R(s)$. Define also \[R^2(s) = \{t\in W: (\exists u\in W) u\in R(s) \wedge t\in R(u)\}\] and note that $R^2(s) \subseteq R(s)$ as witnessed by taking $u=s$.

Say that $s\in W$ is of \emph{circular type} if there exists $a,b \in D'$ such that $s\in W'_a$ and such that $W'_a$ and $W'_b$ are in circular position. If $s\in W$ is not of circular type, say that it is of \emph{linear type}.

\begin{lemme}\label{lem:circular type}
If $s\in W$ is of circular type as witnessed by $W'_a$ and $W'_b$, then $R^2(s) = W'_a \cup W'_b$. If $t\in R^2(s)$, then $t$ is of circular type and $R^2(t)=R^2(s)$.
\end{lemme}
\begin{proof}
We first prove that $R^2(s) = W'_a \cup W'_b$. If $t\in W'_a$, $t<s$ in $W'_a$, then $t \to s$ holds by definition, hence $t\in R(s)$, and similarly if $s<t$ or $t=s$. Hence $W'_a \subseteq R(s) \subseteq R^2(s)$. Now take $u \in W'_a \cap W'_b$, then $W'_b \subseteq R(u)$ and $u \in R(s)$, so $W'_b \subseteq R^2(s)$. Thus $W'_a \cup  W'_b \subseteq R^2(s)$.

We then show $R(s)\subseteq W'_a \cup W'_b$. Take $t\in R(s)$ and take $c\in D'$ be such that $s$ and $t$ lie in $W'_c$ (this exists by definition whether $s\to t$ or $t\to s$ holds). By Lemma \ref{lem:circular position 2}, $W'_c$ is a convex subset of the circular order $W'_a \cup W'_b$. It follows that $t \in W'_a \cup W'_b$ as required.


Now if $t\in R(s)$, then one of the pairs $(W'_a,W'_b)$ or $(W'_b,W'_a)$ witnesses that $t$ is of circular type and by the previous paragraph $R(t)\subseteq W'_a \cup W'_b$. It follows that $R^2(s)\subseteq W'_a \cup W'_b$, hence $R^2(s)=W'_a \cup W'_b = R^2(t)$. Finally, if $t\in R^2(s)$, there is $u\in R(s)$ such that $t\in R(u)$. But then $u$ is of circular type, then so is $t$ and $R^2(s)= R^2(u) = R^2(t)$.
\end{proof}

\begin{lemme}\label{lem:linear type}
Assume that $s$ is of linear type. Then the relation $\to$ defines a strict linear order on $R(s)$. For any $t\in R(s)$, $t$ is of linear type and $R(t)=R^2(t)=R(s)=R^2(s)$.
\end{lemme}
\begin{proof}
Let $s\to t\to u$. We wish to show $s\to u$. Let $a\in D'$ be such that $s<t$ holds in $W'_a$ and let $b\in D'$ such that $t<u$ holds in $W'_b$. Then $t\in W'_a \cap W'_b$. As $s$ is not of circular type, $W'_a$ and $W'_b$ are not in circular position. Hence either one is included in the other as a convex subset, or one is a simple continuation of the other. Let us consider all cases.

\begin{itemize}
\item[(a)] $W'_a \subseteq W'_b$.

We then have $s<t$ in $W'_b$, hence $s<u$ holds in $W'_b$ and $s\to u$.

\item[(b)] $W'_b\subseteq W'_a$.

In this case, we have $t<u$ in $W'_a$ hence $s<u$ in $W'_a$ and again $s\to u$ holds.

\item[(c)] $W'_b$ is a simple continuation of $W'_a$.

Under this assumption, $W^{k+1}_{(a,b)}$ is well defined and $s<t<u$ holds there. Hence $s\to_{k+1} u$ and by hypothesis, we also have $s\to u$.

\item[(d)] $W'_a$ is a simple continuation of $W'_b$.

It must be that $u\in W'_a$ and $s<t<u$ holds in $W'_a$ so again $s\to u$ is true.
\end{itemize}
Let now $t\in R(s)$ and assume that $t\to u \to v$. Since $t\in R(s)$ and $s$ is not of circular type, also $t$ is not of circular type. Hence the previous result applies to $t$ and we deduce $t\to v$. Therefore $\to$ is transitive on $R(s)$. Since it is irreflexive by construction, $\to$ defines a strict linear order in $R(s)$. It remains to show that $R(t)=R(s)$. Without loss of generality, we have $s\to t$. Take $u\in R(t)$. If $t\to u$, then $s\to u$ and $u\in R(s)$. Assume $u\to t$. Let $b\in D'$ be such that $u<t$ holds in $W'_b$. Then $W'_a$ and $W'_b$ are in good position (by Lemma \ref{lem:good position induction}) and contain $t$. Assume that the initial segment $x\leq t$ of $W'_b$ is a convex subset of $W'_a$. Then $u\in W'_a$, so $u\in R(s)$. Otherwise, the initial segment $x\leq t$ of $W'_a$ is included in $W'_b$, which gives us $s\in W'_b$ and again $u\in R(s)$. We have shown $R(t)\subseteq R(s)$. A similar argument gives us $R(s)\subseteq R(t)$ so $R(s)=R(t)$. Since $t$ was an arbitrary element of $R(s)$, this implies that $R^2(s)=R(s)$ and then also $R^2(t)=R(t)$.
\end{proof}

We can now finish the proof of Theorem \ref{th:gluing}.

Let $E = R^2$. By lemmas \ref{lem:circular type} and \ref{lem:linear type}, $E$ is an equivalence relation on $W$. It is also 0-definable. Furthermore, in each $E$-class, either all elements are of linear type, or all elements are of circular type. Let $E_0$ be the quotient $W/E$, seen as a 0-definable subset of $M^{eq}$. Fix some $e\in E_0$ and let $W[e]\subseteq W$ be the $E$-class corresponding to $e$ (we will also say \emph{coded by} $e$). We will say that $e$, or $W[e]$, is of linear/circular type if all elements in $W[e]$ are such.

Assume first that $e$ is of circular type and let $s\in W[e]$. By Lemma \ref{lem:circular type}, we can write $W[e] = W'_a\cup W'_b$, where $a,b\in D'$ are such that $s\in W'_a$ and $(W'_a,W'_b)$ is in circular position. The union $W'_a \cup W'_b$ admits a canonical circular order which induces each of the linear orders on $W'_a$ and $W'_b$. If now $c,d\in D'$ are such that $W'_c$ and $W'_d$ are in circular position and $W[e] = W'_c \cup W'_d$, then the circular order on $W[e]$ induced by this decomposition is the same as the previous one. This follows from Lemma \ref{lem:circular position 2} which implies that both $W'_c$ and $W'_d$ are convex with respect to the circular order on $W'_a \cup W'_b$ induced by $W'_a$ and $W'_b$; and conversely $W'_a$ and $W'_b$ are convex with respect to the circular order given by $W'_c$ and $W'_d$. Hence those circular orders coincide. It follows that the circular order on $W[e]$ given by any such decomposition is definable over $e$.

Assume now that $e$ is of linear type. Then by Lemma \ref{lem:linear type} $\to$ defines a strict linear order on $W[e]$ which, on each $W'_a$ included in $W[e]$, coincides with its canonical order.

To summarize: we have an equivalence relation $E$ on $W$; every class $e$ of $E$ is equipped with either an $e$-definable linear order, or an $e$-definable circular order. Each $W'_a$, $a\in D'$ is a convex subset of one of those orders. This is also the case for each $W_a$, $a\in D$. Thus an element $V_a$, $a\in D$ of the original family is intertwined with a convex subset of one of those classes. This shows the first and last points of Theorem \ref{th:gluing}. It remains to show the other two, namely that the classes are minimal and any two are independent or in order-reversing bijection.

If $e\in W/E$ is circular, write $W[e] = W'_a \cup W'_b$ as above. Then $W[e]$ equipped with the circular order defined above has topological rank 1 since it is the union of two convex subsets $W'_a$ and $W'_b$ which each have topological rank 1. Assume that there is a cut $c\in \overline{W[e]}$ that is $e$-definable. By construction, the circular order $W[e]$ is covered by elements of the family $(W_a)_{a\in D}$. Hence there is $a\in D$ such that $W_a$ is a convex subset of $W[e]$ and $c$ falls inside $W_a$, so can be naturally identified with an element of $\overline{W_a}$. However $e\in W/E$ is the unique class containing $W_a$, so $e\in \dcl^{eq}(a)$ and $c\in \dcl^{eq}(a)$. This contradicts the fact that $W_a$ is minimal over $a$.

The argument is similar for $e\in W/E$ linear. If $s<t\in W[e]$, then $s\to t$, hence the interval $s<x<t$ is included in $W'_a$ for some $a\in D'$. Assume that there is a parameter-definable convex equivalence relation on $W[e]$ with infinitely many infinite classes. By $\omega$-categoricity, there is $n<\omega$ such that every interval of $W[e]$ intersects either infinitely many infinite classes, or at most $n$ many. Therefore, there is an interval of $W[e]$ which contains infinitely many classes. But then some $W'_a$, $a\in D'$ would contain infinitely many infinite classes which is impossible as it is weakly minimal. The argument that $W[e]$ is minimal is now the same as in the circular case: an $e$-definable cut of $W[e]$ would induce an $a$-definable cut of $W_a$ for some $a\in D$ with $W_a\subseteq W[e]$. This proves the second point in Theorem \ref{th:gluing}.

Finally, if $e,e' \in W/E$, $e\neq e'$, then $W[e]$ and $W[e']$ are disjoint, hence by $\boxplus_3$ there can be no intertwining between a convex subset of $W[e]$ and a convex subset of $W[e']$.

\smallskip
\underline{Claim:} Assume that there is an intertwining between a convex subset $C$ of $W[e]$ and the reverse of a convex subset $C'$ of $W[e']$. Then there is a definable order-reversing bijection between $W[e]$ and $W[e']$.

\smallskip
\emph{Proof}: For $a\in D'$, let $\tilde a\in D'$ be such that $W'_{\tilde a}$ is in definable order-reversing bijection with $W'_a$, as given by Lemma \ref{lem:exists reverse order}.

Restricting $C$ if necessary, $C$ is a convex subset of some $W'_a$, $a\in D'$. Assume first that $e$ is of circular type. Then $W[e] = W'_{a}\cup W'_{b}$ for some $a,b\in D'$ with $W'_a, W'_b$ is circular position. Restricting $C$ further, we may assume that it is included in either $W'_a$ or $W'_b$, say it is the former. It follows that $C'$ is in order-preserving bijection with a convex subset of $W'_{\tilde a}$, but then $C'\subseteq W'_{\tilde a}$. Therefore $W'_{\tilde a}\subseteq W[e']$. Now $W'_{\tilde a}$ and $W'_{\tilde b}$ are also in circular position so $W[e'] = W'_{\tilde a} \cup W'_{\tilde b}$ and there is a definable order-reversing bijection between $W[e]$ and $W[e']$.

The case where $e$ is linear is treated in a similar way. As shown above, two paragraphs before the claim, every interval of $W[e]$ is included in some $W'_a$, $a\in D'$. If $C\subseteq W'_a$, then $C' \subseteq W'_{\tilde a}$. It follows that $W'_{\tilde a}\subseteq W[e']$. Now if $I$ is an arbitrary interval of $W[e]$, let $J$ be an interval containing both $I$ and $C$ and say that $J$ is included in $W'_b$. Then $C'$ is a convex subset of $W'_{\tilde b}$ and as above $W'_{\tilde b}$ is a convex subset of $W[e']$. It follows in particular that $I$ admits a (necessarily unique) order-reversing bijection with a convex subset of $W[e']$ (compose the order-reversing bijection into $W'_{\tilde b}$ with the embedding of this order into $W[e']$). By gluing together those bijection for larger and larger $I$ (which are compatible by uniqueness), we obtain a definable order-reversing bijection between $W[e]$ and $W[e']$.

\smallskip
It follows form the claim and the paragraph before it that any two $E$-classes are either independent or in order-reversing bijection. This therefore settles the third point of Theorem \ref{th:gluing} and finishes the proof.

\subsection{Analysis of a rank 1 structure}\label{sec:analysis}

In this section we assume:

\smallskip
$(\star)$ $M$ is an $\omega$-categorical, rank 1, primitive, unstable NIP structure.

\smallskip
Under those assumptions, we will construct an interpretable set $W$ equipped with a finite-to-one surjective map $W\to M$. The set $W$ will be a disjoint union of finitely many linear and circular orders.

\smallskip
To begin, note that since $M$ is primitive, it is transitive, that is has a unique 1-type. It follows that no element is algebraic over $\emptyset$ and $\rk(a) = 1$ for every $a\in M$. The relation $x\in \acl(y)$ is an equivalence relation on elements of $M$. This follows from the fact the $\rk$ defines a pregeometry on rank 1 structures as explained in Section \ref{sec:rank}, but can also be seen directly: transitivity holds in any structure and symmetry comes from the equivalence: \[x \in \acl(y) \iff \rk(xy) = 1 \iff \rk(xy)<2.\]
As $M$ is assumed to be primitive, $\acl(a) =\{a\}$ for every singleton $a\in M$.

Assume that $\dlr(M)\geq n$, then Fact \ref{fact:linear orders from op-dimension} provides us with a finite tuple of parameters $d$ (named $A$ there) and a $d$-definable subset $X_d$ transitive over $d$ along with $d$-definable linear quasi-orders $\leq_{d,1},\ldots,\leq_{d,n}$ on $X_d$. By definitions of quasi-orders ther are $d$-definable equivalence relations $E_{d,1},\ldots,E_{d,n}$ on $X_d$ such that each $\leq_{d,i}$ induces a linear order (still denoted $\leq_{d,i}$) on the quotient $V_{d,i} := X_d/E_{d,i}$. Note that each $E_{d,i}$ has infinitely many classes by the universality property stated at the end of Fact \ref{fact:linear orders from op-dimension}. Since $M$ has rank 1 and $X_d$ is transitive over $d$, all $E_{d,i}$-classes are finite. The quotients $(V_{d,i},\leq_{d,i})$ also have rank 1 and are transitive, hence minimal over $d$.

Let $E_0(x,y)$ be the equivalence relation on $X_d$ defined by $\acl(dx)=\acl(dy)$. Let $a\in X_d$ and for $i\leq n$, let $a_i$ be the projection of $a$ on $V_{d,i}$. Since $V_{d,i}$ is minimal over $d$, we have $\acl^{eq}(a_i d)\cap V_{d,i} = \{a_i\}$. Hence $\acl^{eq}(ad)\cap V_{d,i}  = \{a_i\}$. Now all elements $E_0$-equivalent to $a$ are in $\acl(ad)$. Therefore they must all project to $a_i$ in $V_{d,i}$. Therefore $E_0$ is finer than $E_{d,i}$. Conversely, if $a$ and $b$ are $E_{d,i}$-equivalent, then $a\in \acl(db)$ and $b\in \acl(da)$, hence $\acl(da)=\acl(db)$ and $a,b$ are $E_0$-equivalent. We thus see that $E_{d,i}= E_0$ for all $i$.

Define $V_d:=X_d/E_0$ and let $\pi_d \colon X_d \to V_d$ be the canonical projection. Then $V_d$ is equipped with $n$ minimal $d$-definable linear orders. No two of those orders are equal, or reverse of each other (as ensured by Fact \ref{fact:linear orders from op-dimension}). We apply Corollary \ref{cor:n independent orders} to $V_d$ equipped with its induced structure and deduce that there exists some finite tuple of parameters $d'$ and a $d'$-definable subset $V'_{d'} \subseteq V_d$ such that $V'_{d'}$ is infinite and transitive over $d'$ and any two orders $\leq_i$ and $\leq_j$, $i\neq j$ are independent when restricted to $V'_{d'}$. Replacing $d$ by $d'$ and $X_{d}$ by $\pi_d^{-1}(V'_{d'})$, we can assume that the $n$ orders are pairwise independent on $V_d$.

To place ourselves in the context of the previous section, pick arbitrarily $n$ distinct elements $c_1,\ldots, c_n$ in $\dcl^{eq}(\emptyset)$. Let \[D = \{(d',c_i) : \tp(d')=\tp(d), i\leq n\}.\]

We will write an element of $D$ as $(d',i)$ instead of $(d',c_i)$. For $(d',i)\in D$, we have a $(d',i)$-interpretable set $V_{d',i}$ defined from the formula defining $V_{d,i}$ by replacing the parameter $d$ by $d'$ and that set is equipped with a minimal order $\leq_{d',i}$ obtained in the same way. Hence we have a uniformly definable family $(V_a)_{a\in D}$ of minimal linear orders. By Theorem \ref{th:gluing}, we can glue those orders together and obtain an interpretable set $W$ equipped with an equivalence relation $E$ satisfying the conclusion of that theorem. We will also make use of the equivalence relation $\sim$ defined before Lemma \ref{lem:beginning of proof}, along with the notation $[a,t]_\sim$ for the $\sim$-class of $(a,t)$, for $a\in D$ and $t\in V_a$.

\medskip
\underline{Claim 1}: Let $e\in W/E$. Take any $a\in D$, say $a=(d,i)$, and $t\in X_d$ so that $[a,\pi_d(t)]_{\sim}$ belongs to $W[e]$. Then $[a,\pi_d(t)]_{\sim}$ is algebraic over $(e,t)$.

\emph{Proof}: Working over $e$, let $p(x,y) = \tp(a,t/e)$. Let $f$ be the $e$-definable function on $p$ which maps a pair $(a', t')\models p$ to $[a',\pi_{d'}(t)]_{\sim}$ (where $a'=(d',i')$). Proposition \ref{prop:no binary functions} gives $f(a,t)\in \acl^{eq}(e,t) \cup \acl^{eq}(e,a)$. Now we have $e\in \acl^{eq}(a)$ since $e$ is the class of $V_{d,i}$ and is therefore definable over $d$. As $X_d$ is transitive over $d$, hence over $a$, $t$ is not algebraic over $a$.

Assume that $f(a,t)\in \acl^{eq}(e,a)$. Note that $t$ is algebraic over $(a,\pi_d(t))$ since $\pi_d$ has finite fibers. The canonical map from $V_{d,i} = \pi_d(X_d)$ to $W$ is injective and maps $\pi_d(t)$ to $[a,\pi_d(t)]_{\sim}$. Hence $t$ is algebraic over $(a,[a,\pi_d(t)]_{\sim})$. Our assumption then implies that $t\in \acl(e,a)$. Next, we have $e\in \acl^{eq}(a)$ since $e$ is the class to which $V_{d,i}$ maps in $W$ and is therefore definable over $d$. So $t\in \acl(a)$, but this is absurd as $X_d$ is transitive over $a$ and infinite. So we must have $f(a,t)\in \acl^{eq}(e,t)$.

\medskip
\underline{Claim 2}: Let $a=(d,i)$ in $D$ and $t\in X_d$, then $t$ is algebraic over $[a,\pi_d(t)]_{\sim}$.

\emph{Proof}: Let $u = [a,\pi_d(t)]_{\sim}$ and let $e\in W/E$ be the $E$-class of $u$. By the previous claim, $u\in \acl^{eq}(e,t)$. Also $u\notin \acl^{eq}(e)$ as $W[e]$ is minimal over $e$, hence also $t\notin \acl^{eq}(e)$. It follows that $\rk(u,e,t) = \rk(e,t) = \rk(e)+1$. And $\rk(u,e)>\rk(e)$ so $\rk(u,e)=\rk(e)+1 =\rk(u,e,t)$. Therefore by Proposition \ref{prop:basic rank}(7), $t\in \acl^{eq}(u,e)$ which equals $\acl^{eq}(u)$ since $e$ is algebraic over $u$ (indeed $e$ is definable over $u$ as it is the $E$-class of $u$).

\medskip
\underline{Claim 3}: There are finitely many $E$-classes.

\emph{Proof}: Let $a=(d,i)$ in $D$ and $t\in X_d$. Write $u = [a,\pi_d(t)]_{\sim}$. By Claim 1, $u$ is algebraic over $(e,t)$. Let $M(e)\subseteq M$ be the set of realizations of $\tp(t/e)$. Since $t$ is not algebraic over $e$ (as $u$ is not), $M(e)$ is infinite.

Assume that there are infinitely many $E$-classes. As $M$ has rank 1, there is an infinite subset $A\subseteq W/E$ such that the intersection $\bigcap_{e\in A} M(e)$ is infinite. Let $A_0 = \{e_1,\ldots,e_n\}$ be a finite subset of $A$ such that $W[e_i]$ and $W[e_j]$ are independent for $i\neq j$. Let $M(A_0) = \bigcap_{e\in A_0} M(e)$. Assume first that all the orders $W[e]$, $e\in A_0$ are linear. In this case, $u\in \dcl(e,t)$. Let $f_e$ be the $e$-definable function defined on $M(e)$ and sending $t$ to $u$ and let \[ X = \{(f_{e_1}(t),\ldots,f_{e_n}(t)) : t\in M(A_0)\}.\]

Then $X$ is an infinite $A_0$-definable subset of the product of the independent orders $\prod_{i\neq n} W[e_i]$. Furthermore, since $t$ is algebraic over $f_{e_i}(t)$ by Claim 2, each projection of $X$ to $W[e_i]$ is finite-to-one. Let $t\in M(A_0)$ be non-algebraic over $A_0$. Then also each $f_{e_i}(t)\in W[e_i]$ is non-algebraic over $A_0$. For each $i$, let $W_i \subseteq W[e_i]$ be an $A_0$-definable convex subset containing $f_{e_i}(t)$ and minimal over $A_0$. Let $X' = X\cap \prod_{i\neq n} W_i$. Then $X'$ is $A_0$-definable and non-empty. By Proposition \ref{lem6}, $X'$ is dense in a boolean combination of closed definable subsets of each $W_i$. Since $W_i$ is minimal over $A_0$, $X'$ is dense in $\prod_{i\neq n}W_i$. It follows that for any $I \subseteq \{1,\ldots,n\}$, we can find $t,t'\in M(A_0)$ such that for $i<n$: \[f_{e_i}(t) \leq_i f_{e_i}(t') \iff i\in I,\] where $\leq_i$ is the order on $W[e_i]$.

Let $\phi(xx';z)$, where $x,x'$ range over $M$ and $z$ ranges over $W/E$ and which expresses the fact that $f_z(x) \leq f_z(x')$ in $W[z]$. Then we see that we have at least $2^n$ $\phi$-types over $A_0$. As $n$ can be taken arbitrarily large, this shows that $\phi$ has the independence property and hence contradicts the NIP assumption.

The argument in the case where the $W[e]$'s are circular is similar. The main difference is that $u$ need not be definable over $(e,t)$ so we no longer have the function $f_e$. Instead define $f_e$ as the function sending some $t$ to the finite set of conjugates of $u$ over $e$. If this set has one element, then the proof goes through, using the circular order at the end instead of the linear order. If $f_e(t)$ has more than one element, then we can for instance replace the formula $\phi$ by the formula $\phi'(xx';z)$ expressing that the elements of $f_e(x)$ lie in a convex subset of $W[e]$ that does not contain any element in $f_e(x')$.

\medskip
\underline{Claim 4}: There is a $0$-definable map $\pi\colon W \to M$ with finite fibers which maps each $E$-class surjectively on $M$.

\emph{Proof}:   It follows from claims 1 and 2, that if $t\in X_d$ and $a=(d,i)$, then $([a,\pi_d(t)]_{\sim},e)$ is inter-algebraic with $t$. Since $e\in \acl^{eq}(\emptyset)$ by Claim 3, we have that $[a,\pi_d(t)]_{\sim}$ and $t$ are inter-algebraic. For $t\in M$, we have seen that $\acl(t) = \{t\}$. We deduce that $(a,\pi_d(t))\sim (b,\pi_d(u))$ implies $t=u$. Thus we have a 0-definable map $\pi\colon  W\to M$ sending each $(a,\pi_d(t))$ to $t$. As $M$ is primitive, each $E$-class maps surjectively onto $M$. Furthermore any $x\in W$ is algebraic over $\pi(x)$ by Claim 1, hence the map $\pi$ has finite fibers.

\medskip
In the proof of the claim, we saw that each map $\pi_d$ is injective. Therefore the equivalence relation $E_0$ defined at the beginning of this subsection is trivial on each $X_d$ and $V_d=X_d$. The orders $V_a$, $a\in D$ therefore have as universe definable subsets of $M$ itself.

Given a point $a\in M$, define $W(a)=\pi^{-1}(a)$. Note that we also have $W(a)=\acl^{eq}(a)\cap W$ as $\acl^{eq}(a)\cap M = M$. For any $V\subseteq W$, define also $V(a)=\pi^{-1}(a)\cap V=\acl^{eq}(a)\cap V$.
 
\medskip
Recall that the $E$-classes come in pairs that are in order-reversing bijection. Let $F\subseteq W/E$ be a set containing exactly one class from each such pair, so that $|F| = \frac 1 2 |W/E|$ and the classes coded by elements of $F$ are pairwise independent. For $V= W[e]$ an $E$-class, the cardinality of $V(t)$ does not depend on the choice of $t\in M$ by primitivity of $M$ (consider the 0-definable equivalence relation on $M$ for which $t,t'\in M$ are equivalent if $|V(t)|=|V(t')|$ for each $E$-class $V$). Let $n(e)$ be the cardinality of $V(t)$ for $t\in M$ and $V=W[e]$.

\medskip
\underline{Claim 5}: With notations as above, for each $e\in F$ choose $n(e)$ many disjoint open intervals of $W[e]$. Then there is $t\in M$ such that each interval selected contains a point of $\pi^{-1}(t)$.

\emph{Proof}: Enumerate $F = \{e_0,\ldots,e_{m-1}\}$ and for $i<m$,  let $V_i=W[e_i]$ and $k_i = n(e_i)$. Then the $V_i$'s are pairwise independent and minimal over$ F$.

If $m=1$, then the result follows at once from either Proposition \ref{lem2} or Proposition \ref{prop:one circular order}, taking $p$ to be the type of a tuple enumerating $V_0(t)$ for some $t\in M$. The general case follows from Theorem \ref{th:all order types}: Let $D \subseteq V_0 ^{k_0} \times \cdots \times V_{m-1}^{k_{m-1}}$ be the set of tuples that are an enumeration of $\pi^{-1}(t) \cap (V_0\cup \ldots \cup V_{m-1})$ for some $t$. Then $D$ is $F$-definable. Let $D_0\subseteq D$ be transitive over $F$, then by Theorem \ref{th:all order types} its closure is a product of its projection on each $V_i^{k_i}$. As in the case $m=1$, we can find in each such projection a tuple which has one point in each of the selected intervals of $V_i$. The result follows.

\medskip
\underline{Claim 6}: The op-dimension of $M$ is at least equal to $\frac 1 2 |\pi^{-1}(t)|$ for some/any $t\in M$.

\emph{Proof}: 
Fix a set $F$ as above and let $W_F  = \bigcup_{e\in F} W[e]$. Let $t\in M$ and let $m= |\pi^{-1}(t)\cap F| = \frac 1 2 |\pi^{-1}(t)|$. Let $\bar a$ enumerate $\pi^{-1}(t)\cap F$. Then as $\bar a \subseteq \acl^{eq}(t)$, Fact \ref{fact:opdim} gives $\dlr(\bar a)\leq \dlr(\bar a t) = \dlr(t)$. Now $W_F$ is a union of pairwise independent orders. By Corollary \ref{cor:op dimension in orders}, $\dlr(\bar a) \geq m$. It follows that $M$ has op-dimension at least $m$.



\medskip
\underline{Claim 7}: For $t\in M$, the fiber $\pi^{-1}(t)$ has size at least $2n$, where $n$ is the same $n$ as at the beginning of this subsection, namely the number of independent orders found on the the set $X_d$.

\emph{Proof}: The set $X_d$ is transitive over $d$ and by construction is equipped with $n$ many independent $d$-definable orders $\leq_1,\ldots,\leq_n$. For each $i\leq n$, the order $(X_d,\leq_i)$ admits a unique definable increasing embedding into one of the $E$-classes. The same is true of the reverse orders $(X_d,\geq_i)$. As those orders are independent, the convex hull of the images of those $2n$ embeddings are disjoint. It follows that a point $t\in X_d$ is sent by those embeddings to $2n$ different points, each of which lies in $\pi^{-1}(t)$.

\begin{lemme}\label{lem:trivial acl}
	Algebraic closure is trivial on $M$ in the sense that $\acl(A)= A$ for each $A\subseteq M$.
\end{lemme}
\begin{proof}
	We already know by the first paragraphs of Section \ref{sec:analysis} that $\acl(a)=a$ for any single element $a\in M$. Let $V$ be any $E$-class of $W$. Then $V$ is definable over $\acl^{eq}(\emptyset)$. By Proposition \ref{prop:no binary functions}, for any $A=\{a_1,\ldots,a_n\}\subseteq M$, we have $\acl^{eq}(A)\cap V = \bigcup_{i\leq n} \acl^{eq}(a_i)\cap V$. Any point in $M$ is inter-algebraic with at least one point in $V$. It follows that $\acl^{eq}(A)\cap M = \bigcup_{i\leq n} \acl^{eq}(a_i)\cap M$, that is $\acl(A) = \bigcup_{i\leq n} \acl(a_i)=A$, since $\acl(a)=a$ for every $a\in M$.
\end{proof}

\begin{corollary}\label{cor:stable reduct}
	The only stable reduct of $M$ is the trivial reduct to pure equality.
\end{corollary}
\begin{proof}
	By Proposition \ref{prop:stable main}, any stable reduct of $M$ is strongly minimal. By the previous lemma, algebraic closure is trivial on such a reduct (since the algebraic closure of a set cannot increase in a reduct). The only strongly minimal set with trivial algebraic closure is pure equality: in a strongly minimal set any two $n$-tuples of algebraically independent points have the same type, hence if algebraic closure is trivial, any two $n$-tuples of distinct points have the same type.
\end{proof}

\begin{lemme}\label{lem:three points}
There are three points $a,b,c\in M$ such that:
\begin{itemize}
\item there is a set $W_{\text{or}}\subseteq W$, definable over $abc$, which is a union of $E$-classes and contains exactly one class in each pair of classes in order-reversing bijection;
\item for each $e\in W/E$, $\dcl^{eq}(abce)$ intersects $W[e]$ in at least 3 points;
\item there is a formula $\theta(x,y;u)$ with parameters in $abc$ such that for each $e\in W/E$, $\theta(x,y;e)$ defines a linear order on $W[e]$.
\end{itemize}
\end{lemme}
\begin{proof}
Choose three points $a,b,c\in M$ so that for every class $V$, we have either $V(a)<V(b)<V(c)$ or $V(c)<V(b)<V(a)$ (meaning that those inequalities holds for any choice of one element in each tuple). This is possible by Theorem \ref{th:all order types}. If $V$ and $V'$ are two classes with an order-reversing definable bijection, then for exactly one of $V$ or $V'$ do we have $V(a)<V(b)<V(c)$. Take $W_{\text{or}}\subseteq W$ to be the union of classes $V$ for which $V(a)<V(b)<V(c)$.

If $V$ is a linear class of code $e$, then $\dcl^{eq}(abce) \cap V = \acl^{eq}(abce)\cap V$, so $\dcl^{eq}(abce)\cap V$ has at least 3 elements. Let now $V$ be a circular class in $W$ of code $e\in W/E$. If we follow the circular order from any point in $V(a)$ to any point in $V(c)$, we encounter the points in $V(b)$ in the same order regardless of which points of $V(a)$ and $V(c)$ we choose. Therefore all the points of $V(b)$ are in $\dcl^{eq}(abce)$. The same is true for $V(a)$ and $V(c)$ by circularly permuting the roles of  $a,b,c$.
\end{proof}

\begin{prop}\label{prop:bounding op-dimension}
The structure $M$ has finite op-dimension, bounded by the number of 4-types of elements of $M$.
\end{prop}
\begin{proof}
Assume that $\dlr(M)\geq n$. Then applying the construction in this subsection starting with that value of $n$, we obtain a set $W$ equivalence relation $E$ and projection $\pi$ such that $\pi^{-1}(t)$ has size at least $2n$ for each $t\in M$ (by Claim 7). It is therefore enough to show that the number of 4-types of elements of $M$ is at least $\frac 1 2 |\pi^{-1}(t)|$.

Let $a,b,c\in M$, $W_{\text{or}}\subseteq W$ and $\theta(x,y;u)$ be as in the previous lemma. Recall the notation $n(e)$ which denotes the number of points in $\pi^{-1}(t)\cap W[e]$. Let $F= W_{\text{or}}/E$. Let $\Sigma$ be the set of functions $\sigma \colon F \to \omega$ satisfying $\sigma(e)\leq n(e)$ for each $e\in F$. For each $\sigma \in \Sigma$, we can find by Claim 5 a point $d_{\sigma}\in M$ such that $\pi^{-1}(d_{\sigma})$ has exactly $\sigma(e)$ points in $W[e]$ that are larger than all the points in $\pi^{-1}(b)$ according to the linear order $\theta(x,y;e)$. (By construction, no point in $\pi^{-1}(b)$ is an extreme point according to that order.) The elements of $F$ are not necessarily definable over $abc$, hence it could be that $d_{\sigma}$ and $d_{\sigma'}$ have the same type over $abc$ for $\sigma \neq \sigma'$. However, if those two points do have the same type, there must be a permutation $\iota$ of $F$ such that $\sigma' = \sigma \circ \iota$, since this is something we can express with a first order formula over $abc$. Say that $\sigma$ and $\sigma'$ are conjugate if $\sigma'=\sigma\circ \iota$ for some permutation $\iota$.

\smallskip
\underline{Claim}: There are at least $\sum_{e\in F} n(e)$ conjugation classes of elements of $\Sigma$.

\emph{Proof}: We prove this by induction on $\sum_{e\in F} n(e)$. First assume that $n(e)=1$ for each $e\in F$, then there are $|F|+1$ many conjugation classes: a conjugation class is given by how many 0s are in the image and this could be any number between 0 and $|F|$. Now assume for some $e_0\in F$ we have $n(e_0)>1$. By induction, there are at least $(\sum_{e\in F} n(e))-1$ conjugation classes of elements $\sigma$ for which $\sigma(e_0) < n(e_0)$. Let $\sigma_0\in \Sigma$ so that the maximal number of values of $e\in F$ are sent to $n(e_0)$. Then $\sigma_0$ cannot be conjugated to a $\sigma$ for which $\sigma(e_0)<n(e_0)$, hence we have at least one extra class.

\smallskip
We conclude that there are at least $\sum_{e\in F} n(e)$ many 1-types over $abc$. As  $\sum_{e\in F} n(e) = \frac 1 2 |\pi^{-1}(t)|$, there are at least $ \frac 1 2 |\pi^{-1}(t)|$ many 4-types of elements of $M$ as required.
\end{proof}

\subsection{The skeletal structure}\label{sec:skeletal}

We know by Proposition \ref{prop:bounding op-dimension} that the op-dimension of $M$ is finite. Let $n=\dlr(M)$. We can then apply the construction of the previous subsection starting with $n$ many independent orders on the set $X_d$ and hence obtain a set $W$, equivalence relation $E$ and projection $\pi:W \to M$ such that $\pi^{-1}(t)$ has exactly $2n$ elements (by Claim 7 it has at least $2n$ elements, and at most $2n$ by Claim 6). To summarize the situation:

\begin{itemize}
\item We have a 0-interpretable set $W$ equipped with a 0-definable equivalence relation $E$ and a 0-definable surjective map $\pi:W\to M$.
\item There are finitely many $E$-classes. 
\item For each $e\in W/E$, the class $W[e]\subseteq W$ coded by $e$ is equipped with either an $e$-definable linear order or an $e$-definable circular order and as such is minimal over $e$.
\item For any $e\in W/E$, there is a unique $\tilde e$ such that there is an $e$-definable order-reversing bijection between $W[e]$ and $W[\tilde e]$.
\item For any $e\in W/E$ and any $e'\in W/E \setminus \{e,\tilde e\}$, the classes $W[e]$ and $W[e']$ are independent.
\item For any $e\in W/E$ there is $0<n(e)<\omega$ such that for any $t\in M$, the set $\pi^{-1}(t) \cap W[e]$ has size $n(e)$.
\item For any $t\in M$, $|\pi^{-1}(t)| = 2 \dlr(M)$.
\end{itemize}

We think of $W$ as a structure in its own right, when equipped with its full induced structure inherited from $M$ (so that the notion of a 0-definable set is the same whether we think of $W$ as a structure or as a definable subset of $M^{eq}$). The original structure $M$ is a quotient of $W$ and we can shift our focus from $M$ to $W$: if we classify the possibilities for $W$ (up to bi-definability), we will also have classified the possibilities for $M$.

A few words about the finite quotient $W/E$. This set is 0-definable in $M^{eq}$ and hence the automorphism group $\aut(M)$ acts on it. We are already aware of some relations on it that have to be preserved by this action (equivalently, that are 0-definable):

-- the unary relation naming the classes that are linear (the complement being the circular classes);

-- for each $k<n$, the unary relation naming the classes which contain exactly $k$ points from $\pi^{-1}(a)$ for some/any $a\in M$;

-- the equivalence relation on $W/E$ with classes of size 2 composed of $E$-classes in order-reversing bijection.

It could be that this is the only structure on that set. At the other extreme, it could also be that $W/E$ is rigid: $\aut(M)$ acts trivially on it, equivalently $W/E\subseteq \dcl^{eq}(\emptyset)$. Any action between those two extremes is also possible. In the definition of the skeletal structure below, we simply include all the structure on $W/E$ without analyzing it further.

\medskip

Consider the reduct of $W$ to the language consisting of:

\begin{itemize}
	\item the equivalence relation $E$;
	\item for every $d\leq |W/E|$ and every 0-definable subset of $(W/E)^d$, a predicate naming the pullback of that set to $W^d$;
	\item a binary relation $\leq (x,y)$ which holds of a tuple $(x,y)$ if and only if $x,y$ are in the same linear $E$-class and $x$ is less than $y$ in that class;
	\item a ternary relation $C(x,y,z)$ which holds of a tuple $(x,y,z)$ if and only if $x,y,z$ are in the same circular $E$-class and are circularly ordered in this order;
	\item a relation $F(x,y)$ which holds if $x$ is in some $E$-class $e$, $y$ is in the $E$-class $\tilde e$ in order-reversing bijection with $e$ and that bijection sends $x$ to $y$;
	\item an equivalence relation $E_{\pi}$ whose classes are the fibers of $\pi$.
\end{itemize}

We will call this structure the \emph{skeletal structure} on $W$ and denote the reduct to it by $W_{Sk}$. Note that the actual language of $W_{Sk}$ is not precisely defined because of the second bullet point above and can depend on $M$. By an isomorphism between two skeletal structures $W_{Sk}$ and $W'_{Sk}$ we mean an isomorphism up to possibly renaming the predicates obtained from that second bullet point.


	
	
	Let $M$ and $M'$ be two structures satisfying $(\star)$. Construct $W$ and $W'$ as above for each of them, along with equivalence relations $E$ and $E'$. For $e\in W/E$ we use again the notation $n(e)$ to denote the number of elements in $\pi^{-1}(t)\cap W[e]$ for some/any $t\in M$, and define similarly $n(e')$ for $e'\in W'/E'$.

\begin{lemme}
With notations as above, assume that we have a bijection $f: W/E \to W'/E'$ which sends circular classes to circular classes, sends linear classes to linear classes, preserving pairs in order-reversing bijection, such that $n(e)=n(f(e))$ for all $e\in W/E$ and such that 0-definable subsets of $(W'/E')^d$ are precisely the images of 0-definable subsets of $(W/E)^d$. Then there is an isomorphism $g$ from $W_{Sk}$ to $W'_{Sk}$ whose image in the quotient by $E$ is $f$.
\end{lemme}
\begin{proof}
This is a straightforward back-and-forth argument using Claim 5: Assume that we have a partial isomorphism $g_0 : W_0\subseteq W \to W'$, where $W_0$ is finite and $g_0$ respects $f$ in the sense that it sends an element of a class $e$ to an element of $f(e)$. Pick an element $a\in W$ and let $a_*$ denote an enumeration of the $E_\pi$-class of $a$. We want to extend $f$ so that its domain contains $a$. For that, it is enough to find some $b_*$ enumerating an $E_\pi$-class in $W'$ and satisfying certain inequalities between the coordinates and elements of $g_0(W_0)$. Let $d=|a_*|$. We can select $d$ many open intervals $I_1,\ldots,I_d$ of the various $E$-classes in such a way that any $b_*$ enumerating an $E_\pi$ class and having its $i$-th coordinate in $I_i$ will satisfy the required inequalities. By Claim 5, we can always find a $b_*$ satisfying those constraints.
\end{proof}


\begin{prop}\label{prop:finitely many skeletal}
	For a given number $n$, there are, up to isomorphism, finitely many possible skeletal structures $W_{Sk}$ associated to structures $M$ with at most $n$ 4-types.
\end{prop}
\begin{proof}
	Given a number $n$ of 4-types, there are finitely many possibilities for the op-dimension of $M$ by Proposition \ref{prop:bounding op-dimension}, hence also finitely many possibilities for the size of a fiber of $\pi$ by Claim 6. The quotient $W/E$ has size at most that of a fiber of $\pi$, hence its size is bounded in terms of $n$ (in fact it is bounded by $2n$). Having fixed the size of $W/E$, there are finitely many possibilities for the choice of which classes are linear or circular, and for each class $V$, how many elements each fiber of $\pi$ has in $V$. Given this data, there are finitely many possibilities for the remaining structure on $W/E$. By the previous lemma, this completely determines the skeletal structure.
\end{proof}

\subsection{The additional local structure}\label{sec:stable structure}

In this subsection, we complete the description of $W$ by showing that the structure not accounted for by the skeletal structure  comes from finite local equivalence relations. In particular, it will follow that if all classes are linear, then there is no additional structure.

Let $F\subseteq W/E$ be a set containing exactly one point in every pair of classes in definable order-reversing bijection. Then any two different classes of $F$ are independent (for instance by the fourth and fifth bullets at the start of Section \ref{sec:skeletal}). From now on, we work over $F$. Take some $m\in M$ and let $m_*$ enumerate the intersection of $\pi^{-1}(m)$ with the classes in $F$.

\medskip
Throughout this subsection, we work over $\acl^{eq}(\emptyset)$, which in particular contains $W/E$ and hence the set $F$. Let $W_*$ be the set in $M^{eq}$ defined by $\tp(m_*/\acl^{eq}(\emptyset))$.  We are now in the context of Section \ref{sec:local relations} and we use the terminology from there.

\begin{lemme}
  There is a finest finite local equivalence relation on $W_*$.
\end{lemme}
\begin{proof}
   Let $\mathcal E$ be a local equivalence relation and let $C_{\bar t}$ be a big cell of $W_*$ defined as in Section \ref{sec:local relations}. Then the relation $\mathcal E(C_{\bar t})$ is definable form $\bar t$ ($\mathcal E(C_{\bar t})$ is invariant under automorphisms fixing $\bar t$, since no other parameters were used in its definition, and hence definable over $\bar t$ by $\omega$-categoricity). By $\omega$-categoricity, there is a finest $\bar t$-definable equivalence relation on $C_{\bar t}$ with finitely many classes. Since $W_*$ is partitioned in finitely many big cells and any local equivalence relation is determined by its restrictions to each of those, there are only finitely many local equivalence relations on $W_*$. Their intersection is the finest one.
\end{proof}

\begin{prop}\label{prop:finitely many classes}
  Let $\mathcal E$ be the finest finite local equivalence relation on $W_*$, whose existence is guaranteed by the previous lemma. Then the number of classes of $\mathcal E$ is bounded by the number of 4-types of elements of $M$.	
\end{prop}
\begin{proof}
  First note that the structure $M$ is still primitive over $\acl^{eq}(\emptyset)$: If say $E(x,y;a)$ is an equivalence relation relation with finitely many classes defined over $a\in\acl^{eq}(\emptyset)$, then the intersection $\bigcap E(x,y;a')$, where $a'$ ranges over the finitely many conjugates of $a$ under $\aut(M)$ is a 0-definable equivalence relation with finitely many classes.

  Take $x\in M$ and let $X\subseteq W_*$ be the union of the $\mathcal E$-classes that one can reach by starting with a point in $\acl^{eq}(x)\cap W_*$ and following a path of small cells. This set $X$ is definable over $\acl^{eq}(\emptyset)x$ and given another $x'\in M$, the corresponding $X'$ is either equal to $X$ or disjoint from it. By primitivity of $M$, they have to be equal and we have $X = W_*$.

  Let $a,b,c\in M$ be given by Lemma \ref{lem:three points}. Then the set $F$ we used to define $W_*$ is definable over $abc$ (it is $W_{or}/E$ in Lemma \ref{lem:three points}). Furthermore, each class $V$ has three points $\alpha_V,\beta_V,\gamma_V$ definable over $abc$ and hence admits a linear order definable over $abc$ (for linear classes it is clear and for circular ones, use any one of the points as the minimal element for instance). It follows that $\acl^{eq}(Fabc)\cap W_* = \dcl^{eq}(abc)\cap W_*$ since algebraic and definable closures coincide on linear orders.
		
	
Define big cells $C_{\bar t}$ as in Section \ref{sec:local relations} using $\alpha_V,\beta_V,\gamma_V$. Let $e$ be any $\mathcal E(C_{\bar t})$-class. Then any class in $\mathcal E(C_{\bar s})$ that one can reach from $e$ following a sequence of transition maps $f_{\bar t',\bar s'}$ is definable from $e$ (along with $abc$). By the previous two paragraphs, every class is reachable by such a sequence of transition maps starting with some point in $\dcl^{eq}(abc)$. Hence each class in each big cell $C_{\bar t}$ is definable over $abc$.

Let $n$ be the number of $\mathcal E$-classes in some/any big cell. Two points in $W_*$ that lie in the same big cell but in different $\mathcal E$-classes have different types over $abc$, since each $\mathcal E$-class is definable over $abc$. It follows that $n$ is bounded by the number of types of elements of $W_*$ over $abc$.
\end{proof}

Let $\phi(\bar x;y)=\phi(x_1,\ldots,x_k,y)$ be a formula over $\acl^{eq}(\emptyset)$, where $y$, as well as each $x_i$ ranges over $W_*$. Fix $\bar a\in W_*^{k}$ and $b\in W_*\setminus \acl^{eq}(\bar a)$. Let $U\subseteq W_*^{k}$ be a product of small cells containing $\bar a$ and $V\subseteq W_*$ a small cell containing $b$. Assume that $U$ and $V$ are small enough so that $V$ is strongly disjoint from any small cell appearing in the product defining $U$.

\smallskip
\underline{Claim}: The formula $\phi_{UV}(\bar x;y)\equiv\phi(\bar x;y)\wedge \bar x\in U \wedge y\in V$ is stable.

\emph{Proof}: Assume not, then we can find sequences $(\bar a_i)_{i<\omega}$ in $U$  and $(b_i)_{i<\omega}$ in $V$ such that $\phi(\bar a_i;b_j)$ holds if and only if $i> j$. For every $j,n<\omega$, the set of realizations of $\tp(b_j/\bar a_{<n})$ is dense in an $\bar a_{<n}$-sector by Corollary \ref{cor:all order types with parameters}. Since that sector has a point in $V$ and $U$ and $V$ are strongly disjoint, it must contain $V$ entirely. Hence the set of realizations of $\tp(b_j/\bar a_{<n})$ is dense in $V$. It follows that the set of realizations of the full type $\tp(b_j/\bar a_{<\omega})$ is dense in $V$. 

For each coordinate $i$ of $W_*$, let the formulas $(\zeta_{i}(y;\bar d_{i,k}):k<\omega)$ define an increasing sequence of intervals on the $i$-th coordinate of $y\in W_*$ (where the tuples $\bar d_{i,k}$ are parameters from $W_*$). Recall that we let $n_*$ denote the number of coordinates of $W_*$. Then the family \[(\zeta_{i}(y;\bar d_{i,k}):k<\omega, i<n_*)\] forms an ird-pattern of size $n_*$ inside $V$. Add to it the line $(\phi(y;\bar a_i):i<\omega)$. This gives an ird-pattern of size $n_*+1$: given $\eta \colon n_*+1 \to \omega$ take $b_\eta$ to be a realization of $\tp(b_{\eta(n_*)}/\bar a_{<\omega})$ such that for all $i<n_*$ and $k<\omega$, we have
\[ \models \zeta_{i}(b_{\eta};\bar d_{i,k}) \iff \eta(i)<k.  \]
This is possible by density of $\tp(b_j/\bar a_{<\omega})$. We conclude that $\dlr(W_*)\geq n_*+1$. Now any element of $W_*$ is inter-algebraic with an element of $M$. Thus by Fact \ref{fact:opdim}, $\dlr(W_*) = \dlr(M) = n_*$: contradiction.

\medskip

%
%
\underline{Claim}: The formula $\phi(x_1,\ldots,x_k,y)$ is local.

\emph{Proof}:
Let $\bar c_U$ (resp. $\bar c_V$) be the tuple of end-points of the intervals in each $E$-class defining $U$ (resp. $V$) and set $\bar c=\bar c_U\bar c_V$.

Let $E_{UV}$ be the equivalence relation on $V$ defined over $\bar c$ by: \[b \rel{E_{UV}} b' \iff (\forall \bar a'\in U)(\phi(\bar a',b)\leftrightarrow \phi(\bar a',b')).\]

Note that for the tuples that we consider $\phi$ is the same thing as $\phi_{UV}$. We claim that $E_{UV}$ has finitely many classes. To see this first note that if $b\in V$ and $\bar a \in U$, then we have $b\ind_{\bar c} \bar c \bar a$. This follows form Lemma \ref{lem:independence_acl}: by construction of $U$ and $V$, $b$ cannot be algebraic over $\bar a\bar c$, hence $\rk(b/\bar a \bar c)\geq 1=\rk(b/\bar c)$.

Now, for $b\in V$, there are finitely many possibilities for $\tp(b/\bar c)$. Fix such a type $p=\tp(b/\bar c)$. By Fact \ref{fact:finitely many non-forking extensions} and the previous paragraph, the set
\[\{\tp_{\phi_{UV}}(b/U) : b\models p\}\] is finite. Hence, there are finitely many possibilities for $\tp_{\phi_{UV}}(b/U)$, $b\in V$.

We now show that $E_{UV}$ actually only depends on $V$ and not on $U$. This will follow from similar argument as in Section \ref{sec:local relations}. To simplify notation, we write e.g. $U\equiv_V U'$ to mean $\bar c_{U} \equiv_{\bar c_V} \bar c_{U'}$, where $U'$ is defines from $c_{\bar U'}$ in the same way that $U$ is defined from $c_{\bar U}$. If $U\equiv_{V} U'$ and $U'\subseteq U$, then $E_{UV}$ and $E_{U'V}$ coincide, since they must have the same number of classes. Next, if $U \equiv_V U'$, are such that $U\cap U' \neq \emptyset$, then there is $U''\subseteq U\cap U'$ such that $U'' \equiv_V U$ and we conclude that $E_{UV}$ and $E_{U'V}$ coincide. Finally, any $U'\equiv_V U$ can be linked to $U$ by a finite chain $U'=U_0,\ldots,U_m=U$, with $U_i\equiv_V U$, $U_i \cap U_{i+1} \neq \emptyset$.

It follows that the relation $E_{UV}$ is definable over $\bar c_V$ and depends only on $V$ and $\tp(U/V)$. We can therefore write it as $E_{UV} = E_{\bar c_V}$. If $V'\subseteq V$, then $E_{\bar c_V}$ and $E_{\bar c_{V'}}$ coincide on $V'$, hence $E_{\bar c_V}$ is a local equivalence relation (by definition of local equivalence relations). This relation depends only on $\tp(\bar a,b/\acl^{eq}(\emptyset))$.

Now, do the same starting with any type of tuple $(\bar a,b)$ and any permutation of the variables of $\phi$. Let $\mathcal E_\phi$ be the intersection of all the local equivalence relations obtained. Then $\mathcal E_\phi$ is a finite local equivalence relation definable over $\acl^{eq}(\emptyset)$. Take now strongly disjoint small cells $C_1,\ldots,C_{k+1}$ and two tuples $(a_1,\ldots ,a_{k+1}), (a'_1,\ldots, a'_{k+1}) \in C_1 \times \cdots C_{k+1}$ such that $(a_i, a'_i) \in \mathcal E_\phi(C_i)$ for all $i\leq k+1$. Then by construction of $\mathcal E_\phi$, we can replace the $a_i$'s by the $a'_i$'s one by one in the formula $\phi(a_1,\ldots,a_{k+1})$ without changing its truth value. Therefore by definition $\phi$ is a local formula.

\subsection{Proof of the main theorems}\label{sec:homogeneous}

We keep the same notation $M, W,\ldots$ as in the previous section. Fix a finite set $A\subseteq M$ so that all elements of $W/E$ are definable over $A$ and each $E$-class has at least three points definable over $A$. Let $F\subseteq W/E$ be a set containing exactly one point from each pair of classes in order reversing bijection and let $W_F\subseteq W$ be the union of the classes in $F$. Then, over $A$, $W_F$ is definable and $W$ is just two copies of $W_F$ (up to a definable bijection), so we can forget about $W$ and focus on $W_F$ instead.

Having named $A$, the induced structure on $W_F$ is extremely simple: Each $E$-class splits into finitely many $A$-definable points and $A$-definable minimal convex subsets. Any such convex subset $C$ is equipped with a definable linear order induced by the linear or circular order of the $E$-class to which it belongs. For each such $C$, let $\mathcal E(C)$ be the finest $A$-definable equivalence relation on $C$. By minimality, it has finitely many classes, each of which is dense. Furthermore, by the argument in the proof of Proposition \ref{prop:finitely many classes}, each one of these classes is definable over $A$. Finally, if $R(x_1,\ldots,x_m)$ is a local definable set on $W_F$, then if $(a_1,\ldots,a_m), (b_1,\ldots,b_m)$ are such that for each $i\leq m$, $a_i$ and $b_i$ lie in the same minimal convex set $C_i$ and are $\mathcal E(C_i)$-equivalent, then we have \[R(a_1,\ldots,a_n)\iff R(b_1,\ldots,b_m).\]

Consider the language $L^ 0_A$ on $W_F$ consisting of:

\begin{itemize}
	\item a constant to name each $A$-definable element of $W_F$;
	\item for each minimal $A$-definable convex subset $C$ of an $E$-class of $W_F$ and each $\mathcal E(C)$-class $e$, a unary predicate $P_e(x)$ naming that class;
	\item for each such $C$, a binary predicate  $\leq_C$ naming the linear order on $C$.
 \end{itemize}
 
 Since we know that apart from the $E$-classes and the orders, any additional structure on $W_F$ is given by local formulas, we see that $W_F$ admits elimination of quantifiers in $L^0_A$.

It remains to transfer that structure to $M$ itself. For any $a\in M$, $\acl^{eq}(a)\cap W_F$ has size $n$ equal to the op-dimension of $M$. Furthermore, every element of it lies in $\dcl^{eq}(Aa)$ (indeed, every such element is the $k$-th element of $\acl^{eq}(a)$ in some $A$-definable convex set $C$, for some $k$ and $C$, and this can be expressed over $Aa$). For each $a\in A$, fix an enumeration $\acl^{eq}(a)\cap W_F = \{a_1,\ldots,a_n\}$. We can further fix such an enumeration so that if $a\equiv_A b$, then $(a_1,\ldots,a_n)\equiv_{A} (b_1,\ldots,b_n)$. 

Consider now the two-sorted structure with sorts $M$ and $W_F$ and the language $L^ 1_A$ consisting of:

\begin{itemize}
	\item the language $L^ 0_A$ on the sort $W_F$;
	\item the projection $\pi\colon W \to M$;
	\item for each $i\leq m$, a function symbol $f_i$ from $M$ to $W_F$ interpreted as $f_i(a) = a_i$, where $a_i$ is as above.
\end{itemize}

\begin{prop}\label{prop:qe for two-sorted structure}
	The two-sorted structure $(M,W)$ admits elimination of quantifiers in $L^ 1_A$. Furthermore the induced structure on $M$ is inter-definable with the original structure on $M$ with elements of $A$ named.
\end{prop}
\begin{proof}
	First notice that the skeletal structure on $W$ as defined in Section \ref{sec:skeletal} is quantifier-free definable from the $L^1_A$-structure: the equivalence relation $E$ is in our language; every $E$-class is a union of sets named by predicates $P_C$, hence any pre-image of a subset of $(W/E)^d$ is $L^1_A$-quantifier-free definable; the orders on $E$-classes can be defined from the orders $\leq_C$; the order-reversing bijections are in the language and finally the equivalence relation $E_{\pi}$ is quantifier-free definable from $\pi$. By Section \ref{sec:stable structure}, the additional structure on $W$ is given by local formulas. Let $R(x_1,\ldots,x_n)$ be such a relation. Let $C_1,\ldots,C_n$ be $E$-classes. Then by Proposition \ref{prop:local relations} and the fact that linear orders are simply connected, $R\cap C_1\times \cdots \times C_n$ is a union of products of the form $X_1\times \cdots X_n$, where $X_i \subseteq C_i$ is a $\mathcal E(C_i)$-class. All those classes are quantifier-free definable in $L^1_A$, hence so is $R$. This shows that the original structure on $W$ after naming $A$ is a reduct of the $L^1_A$-structure. Since the $L^1_A$-structure is itself a reduct of that structure, they are inter-definable. Hence the same is true for the structure on $M$ since it is interpretable as the quotient of $W$ by $E_{\pi}$.
	
	Now quantifier-elimination is shown by a straightforward back-and-forth argument using Claim 5 and the density of each $\mathcal E(C)$-class.
\end{proof}

We can now easily obtain a language on the one-sorted structure $M$ in which we have quantifier elimination: for each unary relation $R(x)$ in $L^1_A$ and each $i\leq m$, let $R_i(x)$ be a unary relation over $M$ interpreted in $M$ as \[M\models R_i(x) \iff (M,W) \models R(f_i(x)).\]

Introduce similarly a binary relation $R_{i,j}(x,y)$ for each binary relation $R(x,y)\in L^1_A$ and each $i,j\leq m$. This gives a binary language in which $M$ admits elimination of quantifiers.

%
%
%

\begin{lemma}\label{lem:finite homogeneity of reducts}
	Let $M$ be an $\omega$-categorical structure. Assume that for some integer $r$, for any set $A\subseteq M$ of size $r$, the expansion of $M$ naming every $\acl^{eq}(A)$-definable set is homogeneous in a finite relational language of arity at most $m$. Then $M$ is finitely homogenizable.
\end{lemma}
\begin{proof}
	We need to show that for some integer $k$, any $n$-type $p(x_1,\ldots,x_n)$ is implied by the conjunction of its restrictions to sets of $k$ variables. Fix an $r$-type $q$ and $\bar a\models q$. Let $L_{\bar a}=\{\phi_1(\bar x_1),\ldots,\phi_l(\bar x_l)\}$ be a set of formulas with parameters in $\acl^{eq}(\bar a)$ such that $M$ has quantifier elimination in a language with a predicate for each of those formulas. Assume that $L_{\bar a}$ is closed under $\aut(\acl^{eq}(\bar a))_{\bar a}$ (automorphisms fixing $\bar a$ pointwise) and that the maximal arity of those formulas is $m$. For any finite set $C\subset M$, define an equivalence relation $E^{\bar a}_C$ on $L_{\bar a}$ by saying that two formulas $\phi(\bar x)$ and $\phi'(\bar x)$ are $E^{\bar a}_C$-equivalent if they are conjugated over $\bar a$ and for any tuple $\bar c$ of elements of $C$, we have \[M\models \phi(\bar c)\leftrightarrow \phi'(\bar c).\]
	If a pair $(\phi,\phi')$ is not in $E^{\bar a}_C$, then there is a subset $C_0\subseteq C$ of size at most $m$ such that $(\phi,\phi')$ is not in $E^{\bar a}_{C_0}$. It follows that for any $C$, there is $C_*\subseteq C$ of size at most $l^2 m$ such that $E^{\bar a}_C=E^{\bar a}_{C_*}$. Since the numbers $l$ and $m$ depend only on $q=\tp(\bar a)$, we can define $N(q)=l^2 m$.
	
	Let $p=\tp(a_1,\ldots,a_n)$ be any type in finitely many variables and we want to show that for some $k$ not depending on $p$, $p$ is implied by its restrictions to subsets of $k$ variables. We can assume that all the $a_i$'s are distinct. Set $\bar a=(a_1,\ldots,a_r)$ and $q=\tp(\bar a)$. Let $C=\{a_1,\ldots,a_n\}$ and take $C_*\subseteq \{a_1,\ldots,a_n\}$ of size at most $N(q)$ so that $E^{\bar a}_{C_*}=E^{\bar a}_{C}$. By construction of $E^{\bar a}_C$, for any subtuple $\bar d$ of $(a_1,\ldots,a_n)$, the type $\tp(\bar d/\bar aC_*)$ implies the quantifier-free $L_{\bar a}$-type of $\bar d$. Since $L_{\bar a}$ is composed of relations of arity at most $m$ the quantifier-free $L_{\bar a}$-type of $(a_1,\ldots,a_n)$ is implied by the conjunction of  the quantifier-free $L_{\bar a}$-types of $(a_{i_1},\ldots,a_{i_m})$ for $i_1,\ldots,i_m\leq n$. This quantifier-free $L_{\bar a}$-type is itself implied by $\tp(a_{i_1},\ldots,a_{i_m}/\bar aC_*)$. Therefore the full type of the tuple $(a_1,\ldots,a_n)$ is implied by the conjunction of $\tp(a_{i_1},\ldots,a_{i_m},\bar a,\bar c_*)$ for $i_1,\ldots,i_m\leq n$, where $\bar c_*$ enumerates $C_*$. Thus $k:=m+r+\max_q N(q)$ has the required property.
\end{proof}

\begin{question}
	In the previous lemma, can we replace ``for any set $A\subseteq M$" by ``for some set $A\subseteq M$"?
\end{question}

\begin{lemma}\label{lem:finitely many reducts}
	Let $M$ be an $\omega$-categorical structure and fix some $r<\omega$. Then there are only finitely many reducts of $M$ which are homogeneous in a relational language of arity at most $r$.
\end{lemma}
\begin{proof}
	A reduct $M'$ of $M$ that is homogeneous in a relational language of arity at most $r$ is entirely determined by its $r$-types, or equivalently the orbits of $\aut(M')$ acting on $M^r$. Any such orbit is a finite union of orbits of $\aut(M)$ acting on $M^r$. Hence there are only finitely many possibilities.
\end{proof}

We can now prove our main theorem.

\begin{thm}\label{th:main}
	Let $M$ satisfy $(\star)$ and assume that $M$ has $n$ 4-types. Then $M$ is distal and is inter-definable with a structure $M_0$ in a finite relational language which is homogeneous and finitely axiomatizable. After naming a finite set of points, $M$ admits elimination of quantifiers in a finite binary language. Furthermore, for a given $n$, there are finitely many possibilities for $M_0$.
\end{thm}
\begin{proof}
	We have already seen that after naming a finite set $A$ as above $M$ admits quantifier elimination in a finite binary language. It is easy to see that the structure $W$ in $L^1_A$ satisfies the definition of distality given after definition \ref{def:distal} with $k=4m$ (the quantifier-free type of an element $a$ over a finite set $B=\dcl(B)$ given by the restriction of that type to $B_0\subseteq B$ composed of the points in $B$ closest to each point in $\dcl(a)$). Distality and non-distality are preserved under naming constants and are preserved by bi-interpretations, hence $M$ is distal.

	All $W/E$-classes are definable over $\acl^{eq}(\emptyset)$ and for any set $A\subseteq M$ of size 3, there are at least 3 $\acl^{eq}(A)$-definable elements in each $E$-class. We can then build the languages $L^0_{\acl^{eq}(A)}$ and $L^1_{\acl^{eq}(A)}$ in the same way that we built $L^ 0_A$ and $L^1_A$, replacing everywhere $A$ by $\acl^{eq}(A)$. We obtain quantifier elimination in $L^1_A$ and hence the expansion of $M$ obtained by naming all $\acl^{eq}(A)$-definable sets is finitely homogeneous. By Lemma \ref{lem:finite homogeneity of reducts}, $M$ itself is finitely homogenizable.
	
	Assume that $M$ is equipped with such a finite relational language $L$ for which it is homogeneous. We have seen that after naming some appropriate finite set of points $A$, $M$ becomes homogeneous in a binary language for which it is finitely axiomatizable. It follows that $M$ is finitely axiomatizable in the language $L(A)$ equal to $L$ augmented by a finite set of constants to name the elements of $A$. Then by quantifying on $A$, we see that $M$ is finitely axiomatizable in $L$.
	
	It remains to prove that there are only finitely many such $M$ having $n$ 4-types. By Proposition \ref{prop:finitely many skeletal}, there are finitely many possibilities for the skeletal structure. Then by Proposition \ref{prop:finitely many classes}, there are finitely many possibilities for the number of $\mathcal E$-classes. By Proposition \ref{prop:finitely many local}, there are finitely many possibilities for $\mathcal E$ itself (given the skeletal structure). Hence there are finitely many possibilities for the $L_A$-structure described above. Since $M$ is a reduct of that structure, we conclude by Lemma \ref{lem:finitely many reducts}.
\end{proof}

\subsection{Reducts}

Using the classification, one can relatively easily describe the reducts of a given structure satisfying $(\star)$, at least when no local relations are present. (We have avoided giving an explicit classification of the possible local relations. This should be doable, but might be a bit tricky especially if non-trivial automorphisms of $W/E$ are also present.)

Let $M$ be such a structure and let $W$ be the finite cover associated to it. Let $M'$ be a reduct of $M$. If $M'$ is stable, then by Corollary \ref{cor:stable reduct} it is pure equality. If $M'$ is unstable, then we can construct a finite cover $W'$ of it as above (with the properties listed at the beginning of Section \ref{sec:skeletal}). Then $W'$ is interpretable in $M$.

In what follows, we work in $M$, over $\acl^{eq}(\emptyset)$. Let $2m=|W'/E'|$. Let $F\subseteq W/E$ be as usual a set containing exactly one class from each pair in order-reversing bijection and let $W_F\subseteq W$ be the union of the $E$-classes in $F$. Define similarly $F'$ and $W'_{F'}$. Now $\acl^{eq}(a)$ contains $n$ elements in $W_F$ and say $m$ elements in $W'_{F'}$. We have $\dlr(a) = n$. Hence by Corollary \ref{cor:op dimension in orders} each $E'$-class $V'$ in $W'_{F'}$ is intertwined with an $E$-class $V$ in $W_F$ and furthermore that intertwining must send points of $\acl^{eq}(a)\cap V'$ to points in $\acl^{eq}(a)\cap V$. It follows that $V'$ is in increasing bijection with a (necessarily dense) subset of $V$. In particular, if $E$-classes are transitive over $\acl^{eq}(\emptyset)$, then the $E'$-classes are naturally a subset of the $E$-classes. We note however that the structure on $W'$ induced from $M'$ could be  proper reduct of the one induced from $M$. For instance, classes that are linear seen in $M$ could become circular in $M'$.

For a given $W$, one can then by inspection determine all the possibilities for $W'$. Instead of attempting to write a general statement, we will examine two special cases: the case where $M=(M;\leq_1,\ldots,\leq_n)$ is equipped with $n$ independent linear orders and the case where $W$ has just two circular orders in order-reversing bijection, each extending to a unique strong type over $\emptyset$.

\smallskip
Assume that $M=(M;\leq_1,\ldots,\leq_n)$ is the Fra\"iss\'e limit of sets equipped with $n$ linear orders and define $W$ and $E$ as usual. Then $W$ is composed of $2n$ linear orders pairwise in order-reversing bijection and otherwise independent, and the fibers of the projection $\pi\colon W\to M$ pick out exactly one element per linear order.
By what we have explained, the $E'$-classes are a subset of the $E$-classes. The induced structure on $W'$ from $M'$ is a reduct of the structure induced from $M$: some classes that are linear in the $M$-induced structure might be circular in the $M'$-induced structure. Also the automorphism group $\aut(M')$ might induce more automorphisms of $W'/E'$ as $\aut(M)$ does (indeed $W'/E'$ being a subset of $W/E$ is fixed pointwise by $\aut(M)$).

To summarize, one can associate to each reduct of $M$ a triple $(V_l,V_c,G)$ where $V_l,V_c$ are two disjoint subsets of $\{1,\ldots,n\}$ of cardinalities $m_l$ and $m_c$ respectively, and $G$ is a subgroup of the  wreath product $\mathbb Z_2 \wr (\mathfrak S_{m_l} \times \mathfrak S_{m_c})$. The subsets $V_l$, $V_c$ indicate respectively which of the $n$ pairs of orders in $W$ are kept as linear orders in $W'$ and which are kept, but become circular in $W'$. The subgroup $G$ is the group of automorphism on the quotient $W'/E'$. The reducts of $M$ are completely classified by such triples and every triple corresponds to a reduct (the triple where $m_l=m_c=0$ corresponding to the trivial reduct to pure equality).

For instance for $n=2$, we have $3^ 2=9$ choices for the pair $(V_l,V_r)$. If either of the two sets has cardinality 2, then we get 10 possibilities for $G$ (the group $\mathbb Z_2 \wr \mathfrak S_2$ is isomorphic to the dihedral group $D_8$ and has 10 subgroups). If the two sets have cardinality 1, we get 5 possibilities for $G$ corresponding to subgroups of $\mathbb Z_2 \times \mathbb Z_2$, if one set has cardinality 1 and the other 0, we have two possibilities for $G$ and finally, if both sets are empty, we have one possibility for $G$. Summing it all up, we obtain 10*2+5*2+2*4+1=39 reducts. We thus recover the result of Linman and Pinsker \cite{Linman-Pinsker}.

\medskip
Let us now turn to the second example. Assume that $W$ has two $E$-classes, which are circular, in order-reversing bijection, conjugated by an automorphism, and the fibers of the projection $\pi$ contain exactly $n$ points per class. The associated $M$ can be obtained by taking the Fra\" iss\'e limit of separations relations with an equivalence relation $F$ having classes of cardinality $n$ and quotienting by $F$.

Let $M'$ be an unstable reduct of $M$ and $W'$ its associated finite cover, which we again think of as interpreted in $M$. Let $V$ be any one of the two $E$-classes of $W$. Every $E'$-class is in definable bijection with $V$. Since the map $\pi'\colon W' \to M'$ is also interpretable in $M$, fibers of $\pi'$ have to contain at least $n$ points from each $E'$-class (otherwise there would be in $W$ an $\acl^{eq}(\emptyset)$-definable equivalence relation on $V$ with classes of size $<n$, which is not the case). Hence as above, since algebraic closure cannot be larger in $M'$ as it is in $M$, $W'$ has two $E'$-classes in order reversing bijection and $\pi'$ is $n$-to-one on each of them. But then we see that $W'$ is isomorphic to $W$ and there can be no additional automorphisms on the set of classes. So $M'$ is equal to $M$.

This shows that $M$ has no proper non-trivial reduct. This gives a new example of an infinite family of $\omega$-categorical structures with no proper reduct, or equivalently of maximal closed (oligomorphic) permutation groups. (See e.g., \cite{Bodirsky2013} or \cite{maxClosed} for more about maximal closed permutation groups.)

\section{Binary structures and multi-orders}\label{sec:permutations}

We say that a structure $M$ is \emph{binary} if it eliminates quantifiers in a finite binary relational language.

\begin{lemma}
Let $M$ be a binary structure. Then $M$ has finite rank.
\end{lemma}
\begin{proof}
Assume not and fix some integer $N$ large enough. Then as $\rk(M)>N$, we can build:

\begin{itemize}
	\item an increasing sequence of finite tuples $(c(n):n<N)$;
	\item for each $n<N$, a $c(n)$-definable set $D_n$, transitive over $c(n)$;
	\item an infinite $c(n)$-definable set of parameters $E_n$, transitive over $c(n)$;
	\item a $c(n)$-uniformly definable family $(X_t:t\in E_n)$ of infinite subsets of $D_n$ which is $k(n)$-inconsistent for some $k(n)<\omega$, such that
 if $n<N-1$, then for some $t\in E_n$, $D_{n+1}\subseteq X_t$.
\end{itemize}

Why is this possible? If we drop the transitivity assumptions on $D_n$ and $E_n$, then this is precisely the definition of rank, where we take $D_n$ to be equal to some $X_t$ built at stage $n-1$ (and $D_0=M$). We can enforce the transitivity hypotheses by first replacing $D_n$ by a transitive $c(n)$-definable subset of it: since there are only finitely many such subsets, it must be true for one of them, say $D'_n$ that the family $(X_t\cap D'_n:t\in E_n)$ has infinitely many infinite sets. Also, there are finitely many $c(n)$-definable transitive subsets of $E_n$. Again, there must be one of them, say $E'_n$ for which $(X_t\cap D'_n:t\in E_n)$ has infinitely many infinite sets. Then replace $E_n$ by $E'_n$ and $(X_t:t\in E_n)$ by $(X_t\cap D'_n:t\in E'_n)$.

\smallskip
\underline{Claim}: For each $n$, there are $x,y\in D_n$ such that for no $t\in E_n$ do we have both $x\in X_t$ and $y\in X_t$.

\emph{Proof}: For any $x\in D_n$ there is a finite tuple $(t_1,\ldots,t_k)$ of elements of $E_n$ such that $x$ is in each $X_{t_i}$ and in no other $X_t$. Since $E_n$ is transitive over $c(n)$, no element of $E_n$ is algebraic over $c(n)$ and we can find a tuple $(t'_1,\ldots,t'_k)\equiv (t_1,\ldots,t_k)$ with $t'_i\neq t_j$ for all $i,j\leq k$ (this follows from Neumann's separation lemma, see \cite[Lemma 1.4(ii)]{FiniteCovers}). Now take $y$ so that $(y,t'_1,\ldots,t'_k)\equiv (x,t_1,\ldots, t_k)$.


\smallskip
For each $n$, let $\phi_n(x;y)$ be the formula expressing that for some $t\in E_n$, we have $x,y\in X_t$. This formula is definable over $c(n)$. We now need to get rid of the parameters in $\phi(x;y)$ and this where we use of the assumption that $M$ is binary. By that assumption, we can write the formula $\phi_n(x;y)$ as a boolean combination of atomic formulas, each of which involves only two elements from the tuple $c(n)$ of parameters and the variables $x,y$. Any atomic formula involving two parameters is uniformly true or false, hence can be removed. Since all elements of $D_n$ have the same type over $c(n)$, any atomic formula involving $c(n)$ and one of the variables $x$ or $y$ also has a constant truth value on $D_n$. Hence we can remove such terms from $\phi_n(x;y)$. In this way, we obtain a formula $\psi_n(x,y)$ which is equivalent to $\phi_n(x;y)$ when evaluated on elements from $D_n$ and has no parameters.

For every $n$, there are $a,b\in D_{n}$ with $\neg \psi_n(a;b)$. However we must have $\psi_m(a;b)$ for all $m<n$. Hence the formulas $\psi_n(x;y)$ define distinct relations. Taking $N$ large enough, this is a contradiction: there are up to logical equivalence only finitely many quantifier-free formulas one can construct from a finite relational language.
\end{proof}

\begin{question}
Let $M$ be a primitive binary structure. Must $M$ have rank 1?
\end{question}

We say that $(M,\leq)$ is \emph{topologically primitive}, where $\leq$ is a linear order, if it does not admit a $0$-definable convex non-trivial equivalence relation.

\begin{lemma}
Let $(M,\leq,\ldots)$ be a ranked $\omega$-categorical structure, where $\leq$ is a linear order on $M$. Assume that $(M,\leq)$ is topologically primitive. Then $(M,\leq)$ has topological rank 1.
\end{lemma}
\begin{proof}
Assume that over parameters $\bar a$, there is some definable convex equivalence relation $E_{\bar a}$ with infinitely many  classes. By $\omega$-categoricity, the order induced by $\leq$ on the quotient $M/E_{\bar a}$ is not discrete. Thus there are $c<d$ in $M$ such that there are infinitely many $E_{\bar a}$-classes between $c$ and $d$. The relation $R(x,y)$ saying that for every $\bar b\equiv \bar a$, there are finitely many $E_{\bar b}$ classes between $x$ and $y$ is a $0$-definable equivalence relation with convex classes. As $M$ is topologically primitive, $R$ is trivial: its classes are singletons. It follows that for every open interval $I$, we can find some $\bar b\equiv \bar a$ such that $E_{\bar b}$ has infinitely many classes in $I$. This easily implies that $M$ has unbounded rank.
\end{proof}



\begin{thm}\label{th:primitive case}
Let $(M,\leq_1,\ldots,\leq_n)$ be a homogeneous multi-order such that no two orders $\leq_i$ and $\leq_j$ are equal or opposite of each other. Assume that each $(M,\leq_i)$ is topologically primitive, then $M$ is the Fra\"iss\'e limit of finite sets equipped with $n$ orders.
\end{thm}
\begin{proof}
The assumptions along with the previous lemmas imply that each order $(M,\leq_i)$ has topological rank 1 and is a complete type over $\emptyset$. Proposition \ref{prop:iso types of n orders} describes the possibilities. The only homogeneous structures in the list are the ones with no intertwining (other than equalities between orders), since the intertwining relations $R_{ij}$ are not quantifier-free definable from the orders.
\end{proof}

The classification of imprimitive homogeneous multi-orders is carried out in \cite{perm_imprimitive}, making further use of  techniques from this paper.

More generally, a primitive set equipped with $n$ orders definable in a binary structure satisfies the hypotheses of Proposition \ref{prop:iso types of n orders}. This might help in classifying other classes of ordered homogeneous structures.

\section{The general NIP case}\label{sec:conjectures}

We hope to be eventually able to classify all finitely homogeneous NIP structures, and possibly even all $\omega$-categorical structures having polynomially many types over finite sets.

\begin{conjecture}
	Let $M$ be finitely homogeneous and NIP, then:
	
	\begin{enumerate}
		\item The automorphism group $\aut(M)$ acts oligomorphically on the space of types $S_1(M)$.
		\item The structure $M$ is interpretable in a distal, finitely homogeneous structure.
		\item There is $M'$ bi-interpretable with $M$ whose theory is quasi-finitely axiomatizable.
		\item If $M$ is not distal, then its theory is not finitely axiomatizable.
	\end{enumerate}
\end{conjecture}

Points (2) and (3) each imply that there are only countably many such structures (for point (2), this follows from Theorem \ref{th:distal is fin axiom} below). If $M$ is stable finitely homogeneous, then it is $\omega$-stable and the conjecture is known to be true: (1) by \cite[Theorem 6.2]{CHL}, (2) by \cite{Lachlan:coord}, (3) by \cite{Hr_totallycategorical} and (4) by \cite[Corollary 7.4]{CHL}.

Note that we cannot expect an analogue of Theorem \ref{th:main}: For $k<\omega$, let $M_k$ be the Fra\" iss\'e limit of finite trees with $\leq k$ branching at each node. Then for $k\geq 4$, the structures $M_k$ all have the same $4$-types.

The previous conjecture was stated for the finitely homogeneous case, but we could have stated it also for $\omega$-categorical  structures with polynomially many types over finite sets, or finite dp-rank, which is \emph{a priori} weaker. (For a definition of dp-rank, see e.g. \cite[Chapter 4]{NIPbook}.) However, even the stable case is then unkown.

\begin{question}
	Let $M$ be $\omega$-categorical, stable of finite dp-rank. Is $M$ $\omega$-stable?
\end{question}

One intuition we have on NIP structures is that they are somehow combinations of stable and distal ones. At the very least, we expect that reasonable statements that hold true for stable and distal structures are true for all NIP structures. If $M$ is finitely homogeneous and stable, then we know that it is quasi-finitely axiomatizable. Somewhat surprisingly, the distal case can be proved directly rather easily: see Theorem \ref{th:distal is fin axiom} below. We consider this as strong evidence towards this part of the conjecture. It is possible that the other parts could also be proved directly for distal structures, without having any kind of classification, but we have not managed to do so.

\begin{thm}\label{th:distal is fin axiom}
	Let $M$ be homogeneous in a finite relational language $L$ and distal. Then the theory of $M$ is finitely axiomatizable.
\end{thm}
\begin{proof}
	
Let $r$ be the maximal arity of a relation in $L$. By distality, there is $k$ such that for any finite set $A\subseteq M$ and element $a\in M$, there is $A_0\subseteq A$ of size $\leq k$ with $\tp(a/A_0)\vdash \tp(a/A)$. Let $n_0=kr+k+r+1$. Consider the theory $T_*$ composed of:
	\begin{enumerate}
		\item all formulas of the form $(\forall \bar x) \phi(\bar x)$, with $|\bar x|\leq n_0$ and $\phi$ quantifier-free that are true in $M$;
		\item all formulas of the form $(\forall \bar x)( \theta (\bar x) \to (\exists y) \phi(\bar x,y))$ with $|\bar x|\leq k$, $|y|=1$ and $\theta,\phi$ quantifier-free that are true in $M$.
	\end{enumerate}	
	Up to logical equivalence, $T_*$ contains finitely many formulas. Since $M$ is a model of $T_*$, that theory is consistent. Let $N$ be any countable model of it and we will show that $N$ is isomorphic to $M$.
	
	\smallskip
	\underline{Claim 0}: Let \[\Upsilon \equiv (\forall x,\bar y,\bar z)( \theta(x,\bar y)\wedge \psi(\bar y,\bar z) \to \phi(x,\bar z)),\] with $|x|=1$, $|\bar y|\leq k$ and where each of $\theta,\psi,\phi$ is quantifier-free and describes a complete type. Then if $M$ satisfies $\Upsilon$, so does $N$.
	
	\emph{Proof}: Since the arity of $L$ is bounded by $r$, $\phi( x,\bar z)$ is a conjunction of formulas of the form $\phi'(x,\bar z')$, where $\bar z'\subseteq \bar z$ is a subtuple of size $\leq r$. For each such formula, we have \[M\models (\forall x,\bar y,\bar z')(\theta(x,\bar y)\wedge \psi'(\bar y,\bar z')\to \phi'(x,\bar z'))\] where $\psi'(\bar y,\bar z')$ is a complete quantifier-free formula implied by $\psi(\bar y,\bar z)$ with variables $(\bar y,\bar z')$. This formula is in $T_*$, since it is universal and has less than $n_0$ variables, so $N$ also satisfies it.
	
	\smallskip
	\underline{Claim 1}: $N$ satisfies the universal theory of $M$: for any finite set $B\subseteq N$, there is $B'\subseteq M$ which is isomorphic to it.
	
	\emph{Proof}:  We prove the result by induction on the cardinality of $B$. For $|B|\leq n_0$, this follows from the construction of $T_*$. Assume that we know the result for some $n\geq n_0$ and are given a finite subset $B\subseteq N$ of size $n$ and an additional point $d\in N$. We want to find an isomorphic copy of $B\cup \{d\}$ in $M$. Pick any $r$ distinct elements $b_0,\ldots,b_{r-1}$ in $B$. For $i< r$, set $B_i = B\setminus \{b_i\}$. The set $B_i\cup \{d\}$ has an isomorphic copy in $M$. It follows by distality of $M$ that there is $B'_i\subseteq B_i$ of size $\leq k$ such that \[(\triangle) \qquad M\models \tp(d,B'_i) \wedge \tp(B'_i,B_i) \to \tp(d,B_i).\] By Claim 0, $N$ also satisfies that formula. Let $B_r=\bigcup_{i<r} B'_i$. By the case $n=kr+1<n_0$, the set $B_r\cup \{d\}$ is isomorphic to some $A_r\cup \{c\}$ in $M$. By homogeneity of $M$ and induction, we can find $A\supseteq A_r$ such that $\tp(A_r,A)=\tp(B_r,B)$. For $i<r$, define $A_i$ is the image of $B_i$ under this isomorphism. By $(\triangle)$, which holds both in $M$ and in $N$, we have $\tp(d,B_i)=\tp(c,A_i)$ for each $A$. Since the arity of the language is at most $r$ and any $r$ elements from $Bd$ either lie in $B$ or in some $B_id$, we conclude that $Bd$ and $Ac$ are isomorphic. This finishes the induction.
	
	 \smallskip

	We now show by back-and-forth that $N$ is isomorphic to $M$. Assume we have a partial isomorphism $f$ from a finite subset $A\subseteq M$ to $N$. Let $c\in M$. By distality, there is $A_0\subseteq A$ of size $\leq k$ such that $\tp(c/A_0)\vdash \tp(c/A)$. Let $B_0$ be the image of $A_0$ in $B$. By assumption on $T_*$, there is $d\in N$ such that $\tp(d,B_0)=\tp(c,A_0)$. By Claim 0, we have $\tp(d,B)=\tp(c,A)$, hence we can extend the partial isomorphism $f$ by setting $f(c)=d$. The back direction follows at once from Claim 1 and homogeneity of $M$.
\end{proof}

\bibliography{tout.bib}

\begin{thebibliography}{KLM89}

\bibitem[BM16]{Bodirsky2013}
Manuel Bodirsky and Dugald Macpherson.
\newblock Reducts of structures and maximal-closed permutation groups.
\newblock {\em Journal of Symbolic Logic}, 81(3):1087--1114, 2016.

\bibitem[Bra18]{Braunfeld:3d}
Samuel Braunfeld.
\newblock Homogeneous 3-dimensional permutation structures.
\newblock {\em The Electronic Journal of Combinatorics}, 25(2):Paper 52, 2018.

\bibitem[BS20]{perm_imprimitive}
Samuel Braunfeld and Pierre Simon.
\newblock The classification of homogeneous finite-dimensional permutation
  structures.
\newblock {\em Electronic Journal of Combinatorics}, 27(1):Paper 38, 2020.

\bibitem[Cam76]{Cameron:linear}
Peter Cameron.
\newblock Transitivity of permutation groups on unordered sets.
\newblock {\em Mathematische Zeitschrift}, 148:127--139, 1976.

\bibitem[Cam81]{Cameron:orbits2}
Peter Cameron.
\newblock Orbits of permutation groups on unordered sets, {II}.
\newblock {\em J. Lond. Math. Soc. (2)}, 23:249--264, 1981.

\bibitem[Cam87]{Cameron:trees}
Peter Cameron.
\newblock Some treelike objects.
\newblock {\em Quart. J. Math. Oxford (2)}, 38:155--183, 1987.

\bibitem[Cam02]{cameron:permutations}
Peter Cameron.
\newblock Homogeneous permutations.
\newblock {\em The Electronic Journal of Combinatorics}, 9(2), 2002.

\bibitem[CH03]{CherHru}
G.L. Cherlin and E.~Hrushovski.
\newblock {\em Finite Structures with Few Types}.
\newblock Annals of mathematics studies. Princeton University Press, 2003.

\bibitem[CHL85]{CHL}
G.~Cherlin, L.~Harrington, and A.H. Lachlan.
\newblock $\omega$-categorical, $\omega$-stable structures.
\newblock {\em Annals of Pure and Applied Logic}, 28(2):103 -- 135, 1985.

\bibitem[CS15]{ExtDef2}
Artem Chernikov and Pierre Simon.
\newblock Externally definable sets and dependent pairs {II}.
\newblock {\em Transactions of the American Mathematical Society},
  367(7):5217--5235, 2015.

\bibitem[EH93]{FiniteCovers}
David~M. Evans and Ehud Hrushovski.
\newblock On the automorphism groups of finite covers.
\newblock {\em Annals of Pure and Applied Logic}, 62(2):83--112, 1993.

\bibitem[GH15]{GuinHill}
Vincent Guingona and Cameron~Donnay Hill.
\newblock On a common generalization of shelah's 2-rank, dp-rank, and o-minimal
  dimension.
\newblock {\em Annals of Pure and Applied Logic}, 166(4):502--525, 2015.

\bibitem[Hru89]{Hr_totallycategorical}
Ehud Hrushovski.
\newblock Totally categorical structures.
\newblock {\em Trans. Amer. Math. Soc.}, 313:131--159, 1989.

\bibitem[KLM89]{KLM}
W.M. Kantor, Martin~W. Liebeck, and H.~D. Macpherson.
\newblock $\aleph_0$-categorical structures smoothly approximated by finite
  substructures.
\newblock {\em Proc. London Math. Soc. (3)}, 59:439--463, 1989.

\bibitem[KS16]{maxClosed}
Itay Kaplan and Pierre Simon.
\newblock The affine and projective groups are maximal.
\newblock {\em Transaction of the American Mathematical Society},
  368(7):5229--5245, 2016.

\bibitem[Lac84]{Lachlan:stable_homogeneous}
A.~H. Lachlan.
\newblock On countable stable structures which are homogeneous for a finite
  relational language.
\newblock {\em Israel Journal of Mathematics}, 49(1):69--153, Sep 1984.

\bibitem[Lac87]{Lachlan:coord}
A.~H. Lachlan.
\newblock Structures coordinatized by indiscernible sets.
\newblock {\em Annals of Pure and Applied Logic}, 34:245--273, 1987.

\bibitem[LP15]{Linman-Pinsker}
Julie Linman and Michael Pinsker.
\newblock Permutations of the random permutation.
\newblock {\em The Electronic Journal of Combinatorics}, 22(2), 2015.

\bibitem[Mac85]{Mac:orbits}
H.~D. Macpherson.
\newblock Orbits of infinite permutation groups.
\newblock {\em Proc. London Math. Soc. (3)}, 51:246--284, 1985.

\bibitem[Mac87]{Mac:rapid_growth}
H.~D. Macpherson.
\newblock Infinite permutation groups of rapid growth.
\newblock {\em J. London Math. Soc. (2)}, 35:276--286, 1987.

\bibitem[Mac11]{Mac_survey}
Dugald Macpherson.
\newblock A survey of homogeneous structures.
\newblock {\em Discrete Mathematics}, 311(15):1599 -- 1634, 2011.
\newblock Infinite Graphs: Introductions, Connections, Surveys.

\bibitem[Mar02]{Marker}
David Marker.
\newblock {\em Model Theory: An Introduction}.
\newblock Springer, 2002.

\bibitem[Ons06]{On1}
Alf Onshuus.
\newblock Properties and consequences of thorn-independence.
\newblock {\em J. Symbolic Logic}, 71(1):1--21, 2006.

\bibitem[OS21]{RosFinHom}
Alf Onshuus and Pierre Simon.
\newblock Dependent finitely homogeneous rosy structures.
\newblock preprint, 2021.

\bibitem[Pil96]{PillayBook}
A.~Pillay.
\newblock {\em Geometric stability theory}.
\newblock Oxford logic guides. Clarendon Press, 1996.

\bibitem[She90]{Sh:c}
Saharon Shelah.
\newblock {\em Classification theory and the number of nonisomorphic models},
  volume~92 of {\em Studies in Logic and the Foundations of Mathematics}.
\newblock North-Holland Publishing Co., Amsterdam, second edition, 1990.

\bibitem[Sim13]{distal}
Pierre Simon.
\newblock Distal and non-distal theories.
\newblock {\em Annals of Pure and Applied Logic}, 164(3):294--318, 2013.

\bibitem[Sim15]{NIPbook}
Pierre Simon.
\newblock {\em A Guide to {NIP} theories}.
\newblock Lecture Notes in Logic. Cambridge University Press, 2015.

\bibitem[Sim22]{linear_orders}
Pierre Simon.
\newblock Linear orders in {NIP} structures.
\newblock {\em Advances in Mathematics}, 395, 2022.

\bibitem[TZ12]{TentZieg}
K.~Tent and M.~Ziegler.
\newblock {\em A Course in Model Theory}.
\newblock Lecture Notes in Logic. Cambridge University Press, 2012.

\bibitem[Zil]{Zilber_book}
B.~Zilber.
\newblock {\em Uncountably Categorical Theories}.
\newblock Translations of mathematical monographs. American Mathematical
  Society.

\end{thebibliography}

\end{document}